\documentclass[12pt,a4paper,twoside,reqno]{amsart}
\usepackage[a4paper,top=3cm,bottom=3cm,inner=2.5cm,outer=2.5cm]{geometry}
\usepackage[utf8]{inputenc}
\usepackage{amssymb,amsthm,amsfonts}
\usepackage[british]{babel}
\usepackage{xparse}
\usepackage{multirow,enumitem,array,longtable}
\usepackage{pst-node,placeins,xspace,fix-cm}
\usepackage{accents}
\usepackage[nomessages]{fp}
\usepackage[colorlinks,linkcolor=blue,citecolor=violet,urlcolor=darkgray,final,hyperindex,linktoc=page,hyperfootnotes=true]{hyperref}
\usepackage{graphicx}
\usepackage{eucal}
\usepackage{amscd}
\usepackage[all,2cell]{xy}
\usepackage{tikz-cd}
\usepackage{quiver}
\usepackage{mathrsfs}

 \usepackage{tikz}
 \usetikzlibrary{positioning}
 \tikzset{mynode/.style={draw,circle,inner sep=1pt,outer sep=0pt}}

\newcolumntype{F}{>{$}c<{\hspace{-0.9ex}$}}
\newcolumntype{:}{>{$}m{0.8ex}<{$}}
\newcolumntype{R}{>{$}r<{$}}
\newcolumntype{C}{>{$}c<{$}}
\newcolumntype{L}{>{$}l<{$}}
\newcolumntype{N}{@{}>{$}l<{$}}
\newlength\horspace
\newcommand{\h}[1][1.0]{\hspace*{#1\horspace}}
\newlength\verspace
\newcommand{\vvv}[1]{\v{#1}}
\tikzset{iso/.style={draw=none,every to/.append style={edge node={node [sloped, allow upside down, auto=false]{$\cong$}}}}}
\tikzset{adjunction/.style={draw=none,every to/.append style={edge node={node [sloped, allow upside down, auto=false]{$\dashv$}}}}}
\tikzset{simeq/.style={draw=none,every to/.append style={edge node={node [sloped, allow upside down, auto=false]{$\simeq$}}}}}
\tikzset{simeqS/.style={draw=none,every to/.append style={edge node={node [sloped, allow upside down, auto=false]{$\raisebox{0.8em}{$\simeq$}$}}}}}
\tikzset{aiso/.style={simeqS,preaction={draw,->}}}
\tikzset{proarrowS/.style={draw=none,every to/.append style={edge node={node [sloped, allow upside down, auto=false]{\raisebox{1.4pt}{\small$\shortmid$}}}}}}
\tikzset{proarrow/.style={proarrowS,preaction={draw,->}}}
\tikzset{dotdot/.style={dash pattern=on 0.25ex off 0.2ex, dash phase=0ex}}
\tikzset{RightA/.style={double distance=3.5pt,>={Implies},->},%
	triple/.style={-,preaction={draw,RightA}},%
	quadruple/.style={preaction={draw,RightA,shorten >=0pt},shorten >=1pt,-,double,double distance=0.2pt}}

\newtheorem{teo}{Theorem}[section]
\newtheorem{theo}[teo]{Theorem}

\newtheorem{cor}[teo]{Corollary}

\newtheorem{lem}[teo]{Lemma}

\newtheorem{prop}[teo]{Proposition}

\theoremstyle{definition}
\newtheorem{defi}[teo]{Definition}
\usepackage{amsmath}
\newtheorem{defne}[teo]{Definition}

\newtheorem{rem}[teo]{Remark}
\newtheorem{remark}[teo]{Remark}
\newtheorem{exampl}[teo]{Example}

\newtheorem{assumption}[teo]{Assumption}
\newtheorem{cons}[teo]{Construction}
\newtheorem{notation}[teo]{Notation}

\def\nameit#1{\textrm{#1}~}
\def\thex{\nameit{Theorem}}
\def\prox{\nameit{Proposition}}
\def\corx{\nameit{Corollary}}
\def\lemx{\nameit{Lemma}}
\def\defx{\nameit{Definition}}
\def\remx{\nameit{Remark}}
\def\exax{\nameit{Example}}
\def\conx{\nameit{Construction}}
\def\dfn#1{{\itshape #1}}

\newcommand{\refs}[1]{\textup{(}\ref{#1}\textup{)}}

\def\T{\mathcal{T}}

\newdir{ |>}{{}*!/-3.5pt/@{|}*!/-8pt/:(1,-.2)@^{>}*!/-8pt/:(1,+.2)@_{>}}
\newcommand{\TT}{\mathbf{T}}
\newcommand{\EE}{\mathbb{E}}

\NewDocumentEnvironment{cd}{s O{7} O{7} b}{%
	\IfBooleanF{#1}{\begin{equation*}}\begin{tikzcd}[row sep=#2ex,column sep=#3ex,ampersand replacement=\&]
			#4
		\end{tikzcd}\IfBooleanF{#1}{\end{equation*}}\ignorespacesafterend}{}

\newenvironment{fun}{\[\begin{tabular}{F:RCL}}{\end{tabular}\]\ignorespacesafterend}
\newenvironment{eqD}[1]{\begin{equation}\label{#1}}{\end{equation}\ignorespacesafterend}
\newenvironment{eqD*}{\begin{equation*}}{\end{equation*}\ignorespacesafterend}

\def\:{\colon}

\newcommand{\p}[1]{\textup{(}{#1}\kern2pt\textup{)}}

\def\c{\circ}
\newcommand{\iso}{\cong}
\def\phi{\varphi}

\newcommand{\scaleu}[2][1.2]{{\scalebox{#1}{$#2$}}}

\newcommand{\st}{^{\ast}}

\DeclareFontFamily{OT1}{pzc}{}
\DeclareFontShape{OT1}{pzc}{m}{it}{<->s*[1.19]pzcmi7t}{}
\DeclareMathAlphabet{\mathpzc}{OT1}{pzc}{m}{it}

\DeclareFontFamily{U}{dutchcal}{\skewchar\font=45}
\DeclareFontShape{U}{dutchcal}{m}{n}{<->s*[1.05] dutchcal-r}{}
\DeclareMathAlphabet{\mathlcal}{U}{dutchcal}{m}{n}
\newcommand{\catfont}[1]{\ensuremath{\mathpzc{#1}}\xspace}

\newcommand{\A}{\catfont{A}}

\newcommand{\C}{\catfont{C}}

\newcommand{\E}{\catfont{E}}

\renewcommand{\1}{\catfont{1}}

\newcommand{\Set}{\catfont{Set}}

\newcommand{\Cat}{\catfont{Cat}}
\newcommand{\Catadj}{\catfont{Cat_{\,\operatorname{adj}}}}

\newcommand{\Mndmo}{\catfont{Mnd}}
\newcommand{\Mndco}{\catfont{Mnd}_{\operatorname{co}}}
\newcommand{\Mndcoadj}{\catfont{Mnd}_{\operatorname{co}}^{\operatorname{adj}}}
\newcommand{\GAMnd}[1]{\catfont{ActMnd}({#1})}

\newcommand{\Eq}{\operatorname{Eq}}

\NewDocumentCommand{\Fib}{t' t" t+ t? O{n} O{n} o}{
	\ensuremath{\ifx#5t{2\mbox{-}\Set\mbox{-}}\fi\catfont{\ifx#5d{D}\else{\ifx#5o{Op}\else{\ifx#5b{DOp}\else{\ifx#5t{Op}\else{\ifx#5c{Cl\h[-3]}\else{\ifx#5s{Sp}\fi}\fi}\fi}\fi}\fi}\fi{Fib}}\IfBooleanT{#3}{^{\h[3.7]\opn{s}\h[-3.2]}}\IfBooleanT{#4}{^{\h[3.7]\opn{P}\h[-3,7]}}{\IfBooleanTF{#1}{_{\h[0.4]\opn{cart}\ifx#6n{}\else{\h[-1.4],\h[0.4]{#6}}\fi}}{\IfBooleanTF{#2}{_{\h[0.4]\opn{clov}\ifx#6n{}\else{\h[-1.4],\h[0.4]{#6}}\fi}}{\ifx#6n{}\else{_{\h[0.4]{#6}}}\fi}}}\IfNoValueF{#7}{\h[-1]\left({#7}\right)}}
}

\NewDocumentCommand{\Sh}{o m}{
	\ensuremath{\catfont{Sh}\hspace{-0.15ex}\left({#2}\IfNoValueF{#1}{,{#1}}\right)}
}

\NewDocumentCommand{\St}{t+ o m}{
	\ensuremath{\catfont{St}\IfBooleanT{#1}{_{\opn{Ps}}}\hspace{-0.15ex}\left({#3}\IfNoValueF{#2}{,{#2}}\right)}
}

\newcommand{\Desc}{\catfont{Desc}}
\newcommand{\Alg}{\catfont{Alg}}

\newcommand{\Algs}[1]{\Alg({#1})}

\newcommand{\slice}[2]{{#1}/{#2}}

\newcommand{\opn}[1]{\operatorname{#1}}
\newcommand{\Hom}[3][]{\operatorname{Hom}_{\mkern1mu #1}\mkern-1.5mu\left({#2},{#3}\right)}
\newcommand{\HomC}[3]{{#1}\left({#2},\h[1]{#3}\right)}

\newcommand{\id}[1]{\operatorname{id}_{#1}}
\newcommand{\Id}[1]{\operatorname{Id}_{#1}}

\newcommand{\op}{\ensuremath{^{\operatorname{op}}}}

\newcommand{\x}[1][]{\h[-1]\times_{#1}\h[-1]}
\newcommand{\restr}[2]{{\left.\kern-\nulldelimiterspace {#1}\vphantom{\big|} \right|_{#2}}}

\newcommand{\dom}{\operatorname{dom}}
\newcommand{\cod}{\operatorname{cod}}

\newcommand{\Int}[1]{\ensuremath{\int \hspace{-0.35ex} #1}}

\newcommand{\Groth}[1]{\Int{#1}}

\newcommand{\too}{\longrightarrow}
\newcommand{\mto}{\mapsto}
\newcommand{\mtoo}{\longmapsto}
\newcommand{\ito}{\hookrightarrow}

\makeatletter
\newcommand{\aar}[2][]{\xrightarrow[#1]{#2}} 
\def\xlongrightarrowfill@{\arrowfill@\relbar\relbar\longrightarrow}
\newcommand{\arr}[2][]{%
	\ext@arrow 0099\xlongrightarrowfill@{#1}{#2}}
\newcommand{\aarr}[2][]{%
	\ext@arrow 0099\xlongrightarrowfill@{#1}{#2}} 
\def\xprorightarrowfill@{\arrowfill@{\relbar\joinrel\raisebox{0.6pt}{\small$\shortmid$}\joinrel{\relbar}}\relbar\rightarrow}
\newcommand{\aproarrow}[2][]{%
	\ext@arrow 0099\xprorightarrowfill@{#1}{#2}}

\newcommand{\aR}[2][]{%
	\ext@arrow 0055{\Rightarrowfill@}{#1}{#2}}
\def\xLongrightarrowfill@{\arrowfill@\Relbar\Relbar\Longrightarrow}
\newcommand{\aRR}[2][]{%
	\ext@arrow 0099\xLongrightarrowfill@{#1}{#2}}
\def\aitofill@{\arrowfill@{\lhook\joinrel\relbar}\relbar\rightarrow}
\newcommand{\aito}[2][]{%
	\ext@arrow 3095\aitofill@{#1}{#2}}
\def\Longaitofill@{\arrowfill@{\lhook\joinrel\relbar\joinrel\relbar}\relbar\rightarrow}
\newcommand{\aitoo}[2][]{%
	\ext@arrow 0099\Longaitofill@{#1}{#2}}

\def\xlongleftarrowfill@{\arrowfill@\longleftarrow\relbar\relbar}
\newcommand{\all}[2][]{%
	\ext@arrow 0099\xlongleftarrowfill@{#1}{#2}}
\newcommand{\aL}[2][]{%
	\ext@arrow 0055{\Leftarrowfill@}{#1}{#2}}
\def\xLongleftarrowfill@{\arrowfill@\Longleftarrow\Relbar\Relbar}
\newcommand{\aLL}[2][]{%
	\ext@arrow 0099\xLongleftarrowfill@{#1}{#2}}
\def\xmapstofill@{\arrowfill@{\mapstochar\relbar}\relbar\rightarrow}
\newcommand{\am}[2][]{%
	\ext@arrow 0395\xmapstofill@{#1}{#2}}
\def\xlongmapstofill@{\arrowfill@\relbar\relbar\longmapsto}
\newcommand{\amm}[2][]{%
	\ext@arrow 0399\xlongmapstofill@{#1}{#2}}

\newcommand{\eqq}{\DOTSB\protect\Relbar\protect\joinrel\Relbar}
\def\xeqqfill@{\arrowfill@\Relbar\Relbar\eqq}
\newcommand{\aeqq}[2][]{%
	\ext@arrow 0099\xeqqfill@{#1}{#2}}
\def\xRrightarrowfill@{\arrowfill@\equiv\equiv\Rrightarrow}
\newcommand{\aM}[2][]{\ext@arrow 0359\xRrightarrowfill@{#1}{#2}}
\newcommand{\Llongrightarrow}{%
	\DOTSB\protect\equiv\protect\joinrel\Rrightarrow}
\def\xLlongrightarrowfill@{\arrowfill@\equiv\equiv\Llongrightarrow}
\newcommand{\aMM}[2][]{%
	\ext@arrow 0099\xLlongrightarrowfill@{#1}{#2}}
\makeatother

\newcommand{\aiso}{\aar{\scriptstyle\simeq}}

\newcommand{\aequi}{\ensuremath{\stackrel{\raisebox{-1ex}{\kern-.3ex$\scriptstyle\sim$}}{\rightarrow}}}
\newcommand{\aequii}{\ensuremath{\stackrel{\raisebox{-1ex}{\kern-.3ex$\scriptstyle\sim$}}{\longrightarrow}}}

\newcommand{\PB}[1]{\arrow[#1,phantom,"\scalebox{1.6}{\color{black}$\lrcorner$}",very near start]}
\newcommand{\Ar}[4][]{\arrow[#2,"{#3}"{#1},""{name=#4, anchor=center}]}
\newcommand{\Ars}[4][]{\arrow[#2,"{#3}"'{#1},""{name=#4, anchor=center}]}
\newcommand{\Arb}[6][]{\arrow[#2,"{#3}"{#1},from=#4,to=#5,shorten <= #6 em, shorten >= #6 em]}
\newcommand{\Arbs}[6][]{\arrow[#2,"{#3}"'{#1},from=#4,to=#5,shorten <= #6 em, shorten >= #6 em]}

\NewDocumentCommand{\fib}{O{n} O{2.3} mmm}{%
	\begin{cd}*[#2][5]
		{#3}\ifx#1n{\arrow[d,"{\,\scaleu{#4}}"]}\else{\ifx#1i{\arrow[d,hookrightarrow,"{\,\scaleu{#4}}"]}\else{\ifx#1e{\arrow[d,equal,"{\,\scaleu{#4}}"]}\else{\ifx#1R{\arrow[d,Rightarrow,"{\,\scaleu{#4}}"]}\fi}\fi}\fi}\fi\\
		{#5}\ifx#1o{\arrow[u,"{\,\scaleu{#4}}"']}\fi
	\end{cd}\xspace
}
\NewDocumentCommand{\fibdiag}{O{n} O{2.3} mmm}{%
	\begin{cd}*[#2][5]
		{#3}\ifx#1n{\arrow[d,"{\,{#4}}"]}\else{\ifx#1i{\arrow[d,hookrightarrow,"{\,{#4}}"]}\else{\ifx#1e{\arrow[d,equal,"{\,{#4}}"]}\else{\ifx#1R{\arrow[d,Rightarrow,"{\,{#4}}"]}\fi}\fi}\fi}\fi\\
		{#5}\ifx#1o{\arrow[u,"{\,{#4}}"']}\fi
	\end{cd}\xspace
}
\NewDocumentCommand{\sq}{s O{n} O{7} O{7} O{} O{2.7} O{2.2} O{0.5} O{n}}{%
	\def\foosq##1##2##3##4##5##6##7##8{%
		\IfBooleanTF{#1}{\begin{cd}*}{\begin{cd}}[#3][#4]
				{##1}\ifx#2p{\PB{rd}}\fi\arrow[r,"{##5}"]\ifx#9l{\arrow[d,equal,"{##6}"']}\else{\arrow[d,"{##6}"']}\fi\&{##2}\ifx#9r{\arrow[d,equal,"{##7}"]}\else{\arrow[d,"{##7}"]}\fi\ifx#2l{\arrow[ld,Rightarrow,shorten <=#6ex,shorten >=#7ex,"{#5}"{pos=#8}]}\fi\\
				{##3}\ifx#9d{\arrow[r,equal,"{##8}"']}\else{\arrow[r,"{##8}"']}\fi\ifx#2o{\arrow[ur,Rightarrow,shorten <=#6ex,shorten >=#7ex,"{#5}"{pos=#8}]}\fi\&{##4}
		\end{cd}}%
		\foosq}
\NewDocumentCommand{\sqs}{s O{n} O{7} O{7} O{} O{} O{} O{}}{%
	\def\foosqs##1##2##3##4##5##6##7##8{%
		\IfBooleanTF{#1}{\begin{cd}*}{\begin{cd}}[#3][#4]
				{##1}\ifx#2p{\PB{rd}}\fi\arrow[r,"{##5}"#5]\arrow[d,"{##6}"'#6]\&{##2}\arrow[d,"{##7}"#7]\\
				{##3}\arrow[r,"{##8}"'#8]\&{##4}
		\end{cd}}%
		\foosqs}
\NewDocumentCommand{\nat}{s O{n} O{7} O{7} O{2.7} O{2.2} O{0.5} O{n}}{%
	\def\foonat##1##2##3##4##5##6{%
		\IfBooleanTF{#1}{\sq*}{\sq}[#2][#3][#4][{##1}_{##4}][#5][#6][#7][#8]{{##2}\ifx#8l{}\else{({##5})}\fi}{{##3}\ifx#8r{}\else{({##5})}\fi}{{##2}\ifx#8l{}\else{({##6})}\fi}{{##3}\ifx#8r{}\else{({##6})}\fi}{{##1}_{##5}}{\ifx#8l{}\else{{##2}({##4})}\fi}{\ifx#8r{}\else{{##3}({##4})}\fi}{{##1}_{##6}}}%
	\foonat}
\NewDocumentCommand{\tr}{s O{4.5} O{6.5} O{0} O{0} O{n} O{0} O{} O{0}}{%
	\def\footr##1##2##3##4##5##6{%
		\IfBooleanTF{#1}{\begin{cd}*}{\begin{cd}}[#3][#2]
				{##1}\arrow[rr,"{##4}"]
				\Ars[inner sep =0.2ex]{dr}{##5}{A}\&[#4ex]\&[#5ex]{##2}\Ar[inner sep =0.2ex]{ld}{##6}{B}\\
				\&{##3}
				\ifx#6l{\Arb{Rightarrow,shift right=#7em}{#8}{A}{B}{#9}}\else{\ifx#6o{\Arbs{Rightarrow,shift right=#7em}{#8}{B}{A}{#9}}\else{\ifx#6i{\Arbs[inner sep=0.9ex]{iso,shift right=#7em}{#8}{A}{B}{#9}}\else{\ifx#6e{\Arb{equal,shift right=#7em}{#8}{A}{B}{#9}}\else{}\fi}\fi}\fi}\fi
		\end{cd}}%
		\footr}
\NewDocumentCommand{\trslice}{s t+ O{7} O{7}}{%
	\def\footrslice##1##2##3##4##5##6{%
		\IfBooleanTF{#1}{\begin{cd}*}{\begin{cd}}[#3][#4]
				{##1}\arrow[r,"{##4}"]\IfBooleanF{#2}{\arrow[d,"{##5}"']}\&{##2}\\
				{##3}\IfBooleanT{#2}{\arrow[u,"{##5}"]}\arrow[ru,"{##6}"']
		\end{cd}}%
		\footrslice}
\NewDocumentCommand{\tc}{s t+ O{7} O{30} O{} O{} O{} o}{
	\def\footc##1##2##3##4##5{%
		\FPmul\Mulresulttwo{#3}{#3}%
		\FPmul\Mulresult{0.0026}{\Mulresulttwo}%
		\IfBooleanTF{#1}{\begin{cd}*}{\begin{cd}}[#3][#3]
				{##1}\Ar[#5]{r,bend left=#4}{##3}{A}\Ars[#6]{r,bend right=#4}{##4}{B}\&{##2}
				\IfBooleanTF{#2}{\Arb[description,pos=0.49]}{\Arb}{Rightarrow #7}{\mkern1mu {##5}}{A}{B}{\IfNoValueTF{#8}{\Mulresult}{#8}}
		\end{cd}}%
		\footc}
\NewDocumentCommand{\tcwl}{s t+ O{7} O{30} O{} O{} O{} o O{-2}}{
	\def\footcwl##1##2##3##4##5##6##7{%
		\FPmul\Mulresulttwo{#3}{#3}%
		\FPmul\Mulresult{0.0026}{\Mulresulttwo}%
		\IfBooleanTF{#1}{\begin{cd}*}{\begin{cd}}[#3][#3]
				##6 \arrow[r,"{##7}"]\&[#9ex]{##1}\Ar[#5]{r,bend left=#4}{##3}{A}\Ars[#6]{r,bend right=#4}{##4}{B}\&{##2}\IfBooleanTF{#2}{\Arb[description,pos=0.49]}{\Arb}{Rightarrow #7}{\mkern1mu {##5}}{A}{B}{\IfNoValueTF{#8}{\Mulresult}{#8}}
		\end{cd}}%
		\footcwl}
\NewDocumentCommand{\tcwr}{s t+ O{7} O{30} O{} O{} O{} o O{-2}}{
	\def\footcwr##1##2##3##4##5##6##7{%
		\FPmul\Mulresulttwo{#3}{#3}%
		\FPmul\Mulresult{0.0026}{\Mulresulttwo}%
		\IfBooleanTF{#1}{\begin{cd}*}{\begin{cd}}[#3][#3]
				{##1}\Ar[#5]{r,bend left=#4}{##3}{A}\Ars[#6]{r,bend right=#4}{##4}{B}\&{##2}\arrow[r,"{##7}"]\&[#9ex]##6\IfBooleanTF{#2}{\Arb[description,pos=0.49]}{\Arb}{Rightarrow #7}{\mkern1mu {##5}}{A}{B}{\IfNoValueTF{#8}{\Mulresult}{#8}}
		\end{cd}}%
		\footcwr}
\NewDocumentCommand{\tcv}{s t' O{7} O{30} mmmmm}{
	\FPmul\Mulresulttwo{#3}{#3}%
	\FPmul\Mulresult{0.0026}{\Mulresulttwo}%
	\IfBooleanTF{#1}{\begin{cd}*}{\begin{cd}}[#3][#3]
			{#5}\IfBooleanTF{#2}{\Ars{d,leftarrow,bend right=#4}{#7}{A}\Ar{d,leftarrow,bend left=#4}{#8}{B}}{\Ars{d,bend right=#4}{#7}{A}\Ar{d,bend left=#4}{#8}{B}}\\{#6}
			\Arb{Rightarrow}{#9}{A}{B}{\Mulresult}
		\end{cd}}
\NewDocumentCommand{\twonats}{s O{2.2} O{8} O{7} O{1.05} O{3.45} O{2}}{%
	\def\footwonats##1##2##3##4##5##6##7##8##9{%
		\def\foofootwonats####1####2####3####4####5{%
			\IfBooleanTF{#1}{\begin{cd}*}{\begin{cd}}[#3][#4]
					##1 \Ar{r}{##9}{} \Ars{d,bend right=40}{##5}{A} \Ar{d,bend left=40}{##6}{B} \&
					##2 \Ars{d,bend left}{##8}{Q} \arrow[ld,Rightarrow,shift left=#7,"{####4}"{pos=0.48},shorten <=#5ex, shorten >=#6ex]\&[-2ex]
					##1 \Ar{r}{##9}{} \Ar{d,bend right}{##5}{R} \&
					##2 \Ars{d,bend right=40}{##7}{C} \Ar{d,bend left=40}{##8}{D} \arrow[ld,Rightarrow,shift right=#7,"{####5}"'{pos=0.52},shorten <=#6ex, shorten >=#5ex] \\
					##3 \Ars{r}{####1}{} \&
					##4 \&
					##3 \Ars{r}{####1}{} \& 
					##4
					\Arbs{Rightarrow}{\,{####2}}{B}{A}{0.3}
					\Arbs{Rightarrow}{\,{####3}}{D}{C}{0.3}
					\Arb{equal}{}{Q}{R}{#2}
			\end{cd}}%
			\foofootwonats}\footwonats}

\newcommand{\comp}{\operatorname{comp}}
\newcommand{\Act}[2][]{\mathbf{Act}_{#1}(#2)}
\newcommand{\Actp}[3][]{\mathbf{Act}_{#1}^{#2}(#3)}
\newcommand{\DD}{\mathbb{D}}
\newcommand{\BB}{\mathbb{B}}

\begin{document}

\title{
Generalized action monads and descent
}

\author{Elena Caviglia}
\address[Elena Caviglia]{Department of Mathematical Sciences, Stellenbosch University, Stellenbosch 7600, South Africa, and National Institute for Theoretical and Computational Sciences, South Africa.}
\email{elena.caviglia@outlook.com}
\author{Sophie Marques}
\address[Sophie Marques]{Department of Mathematics, School of Sciences, University of Minho, 4710-057 Braga, Portugal, and Centre of Mathematics (CMAT), University of Minho, Portugal.}
\email{sophie.marques@math.uminho.pt}
\author{Luca Mesiti}
\address[Luca Mesiti]{Department of Mathematical Sciences, Stellenbosch University, Stellenbosch 7600, South Africa.}
\email{luca.mesiti@outlook.com}

\keywords{action, monad, internal category, descent}
\subjclass[2020]{18C15, 18D40, 18F20}
\date{}
\begin{abstract}
    We develop a monadic approach to actions of internal categories. Given an internal category, we construct a monad on a slice whose algebras are the actions of that internal category. Then, we give a complete characterization of the monads that arise in this way; we call them generalized action monads. Finally, we prove that these monads yield a well-behaved notion of generalized descent.
\end{abstract}

\maketitle
\setcounter{tocdepth}{1}
\tableofcontents

\section{Introduction}

Actions of internal groups are one of the most fruitful and widely used categorical concepts in algebra and geometry. They successfully abstract the concept of symmetry of a set to general categories, unifying continuous group actions, smooth group actions, and algebraic group actions, among others. In \cite{JanelidzeTholenII}, an important generalization of actions of internal groups is proposed: the notion of action of an internal category. This considerably widens the examples, as internal categories are a powerful notion capturing topological categories, crossed modules, Lie groupoids, double categories and many more mathematical structures \cite{Borceux2,KellyStreet2Cat,StreetFormalTheory}. Internal groups are particular internal categories, with object of objects given by the terminal object. So the theory of actions of internal categories really generalizes that of actions of internal groups. Interestingly, internally to the category of sets, actions of internal categories capture another of the most fruitful and widely used categorical concepts: that of a (co)presheaf (see \cite{JanelidzeTholenII}). So general actions of internal categories may be thought of as internal presheaves.

A key theorem in categorical algebra is that actions of internal groups can be captured as algebras for a monad. More precisely, the actions of an internal group $G$ in a category $\C$ are the algebras for the monad $G\x -$ on $\C$. We generalize this fundamental theorem to the setting of actions of internal categories. This raises a natural problem: the usual action monad of a group object acts on the ambient category itself, but an internal category has many objects. This suggests considering slices of the ambient category. Given an internal category $\DD=($\begin{cd}*
    D_1 \arrow[r,shift left=1.7,"s"]\arrow[r,shift right=1.7,"t"']\& D_0 \arrow[l,"i"{description}]
\end{cd}$)$ in a category $\C$, we construct a monad on the slice over the object of objects $D_0$, whose algebras are the actions of the internal category $\DD$. We call such monads \emph{generalized action monads}. The theorem has important consequences; for example, we apply this to prove the functoriality of taking the category of actions of an internal category. The theorem also has useful applications to descent, as we can use actions of internal categories to capture descent data.

We then investigate the properties and structures of generalized action monads. We investigate how much of an internal category is remembered by its action monad. The objective is to understand what extra structure on the monad is needed to successfully reconstruct the internal category. We arrive at a complete characterization of which monads arise as generalized action monads associated to an internal category. So that we have an equivalence of categories between internal categories and monads with appropriate extra structures. Interestingly, the analogue of the fact that, given an internal group $G$, every object can be equipped with a trivial $G$-action plays a major role in the characterization of generalized action monads. As it was pointed out to us, our equivalence of categories between internal categories and monads with extra structures could also be recovered by the general theory developed in \cite{Leinster} for $T$-multicategories, choosing $T=\Id{}$. By focusing on internal categories, we realize that the corresponding monads generalize the well-known action monads for internal group actions. Moreover, in the case of internal categories it is not necessary to require cartesian monads a priori.

Furthermore, we develop a fibrational generalization of our constructions. We explain that many of our results work relatively to a general bifibration satisfying the Beck-Chevalley condition at pullback squares rather than just to the codomain bifibration. In recent years, similar endeavours have been proved to be very fruitful. And this generalization is very useful for our applications.

Finally, we apply our results to descent theory. Classical
descent along a morphism asks when an object can be reconstructed from local data satisfying compatibility conditions on overlaps, which are captured by the kernel pair of the starting morphism. B\'{e}nabou and Roubaud \cite{BenabouRoubaud} then managed to capture this theory of descent from a monadic perspective. The category of descent data is precisely the category of Eilenberg--Moore algebras of a monad associated to the starting morphism $p$, given by the composite $p^\ast p_!$ where $p^\ast$ is pulling back along $p$ and $p_!$ is pushing forward along $p$ (in the case of the codomain bifibration, $p_!$ is simply postcomposition by $p$). And the comparison functor of classical descent theory is captured by the comparison functor of monad theory (see \cite{Beck}). Monadic aspects of descent have then been extensively explored, see for example \cite{StreetFormalTheory,StreetDescent,JanelidzeTholenI,JanelidzeTholenII,NunesPseudoKan,NunesDescentKan,NunesSemanticFactorization}.

Our applications to descent naturally embed into this general monadic framework. Rather than considering any monad to yield a notion of descent, we restrict ourselves to generalized action monads. This has the advantage of yielding a notion of generalized descent that satisfies better properties than those of general monadic descent. In particular, we construct a comparison functor using the extra structure of a generalized action monad. And we construct a realization functor that is left adjoint to the comparison functor using coequalizers, in a similar way to the realization functor of classical descent along a morphism. The idea is that every internal category, equivalently corresponding to a generalized action monad, encodes the instructions for the compatibility that the local data need to satisfy. Classical descent along a morphism is then captured by the \vvv{C}ech internal groupoid, built using the kernel pair of the starting morphism. The usual descent monad is precisely the generalized action monad associated to this internal
groupoid. By using internal categories to encode the compatibility instructions for descent, we can then also simultaneously capture Galois descent.

Galois descent was brought forward by \cite{MarquesGaloisDescent} following the standard point of view in categorical approaches to descent theory to consider the compatibilities on local data to be given by symmetries \cite{JanelidzeTholenI,JanelidzeTholenII,StreetDescent}. It is a notion of descent associated to actions of group objects. The compatibilities are encoded into an action groupoid associated to the starting action of a group object. This motivated us to develop notions of descent associated to actions of internal categories, with the aim to capture at the same time classical descent along a morphism and Galois descent, while enjoying as many as possible of the good properties that those notions of descent satisfy.

Another general descent framework in which our theory embeds is that of lax descent objects, brought forward by Lucatelli Nunes \cite{NunesDescentKan}. There, the starting point is an internal precategory, and its image along a pseudofunctor into $\Cat$. The (lax) 2-dimensional limit of the obtained diagram in $\Cat$ then represents a general category of descent data. As explained in \cite{NunesDescentKan}, when considering the pseudofunctor into $\Cat$ that takes slices of the domain category, the obtained category of descent data coincides with the category of actions of the starting internal category. But the theory presented in this paper, associated to generalized action monads, enjoys better properties. In fact, we have monadicity of the category of descent data, while the general framework of \cite{NunesDescentKan} lacks monadicity.

So generalized action monads provide an intermediate well-behaved notion of generalized descent that is still general enough to simultaneously capture both the classical descent along a morphism and Galois descent. In this sense, generalized action monads stand in a sweet spot for descent theory.

\subsection*{Outline of the paper}

The paper is organized as follows.  Section~\ref{secprelim} recalls the
background on monads and their morphisms, internal categories and their actions, classical and
Galois descent, and a useful pullbacks lemma. Section~\ref{secmonadicity}
constructs the action monad of an internal category and identifies its
algebras with internal actions. Then Section~\ref{sec:functoriality} applies the monadicity result to prove functoriality of the action construction. In Section~\ref{sec:chargenactmnd}, we completely characterize generalized action monads, presenting a reconstruction theorem for internal categories. In Section~\ref{sec:general-fibrational-setting}, we then develop a fibrational generalization of the theory. In Section~\ref{sec:monadic-descent-right-modules}, we apply our theory to descent. We present a notion of monadic descent associated to generalized action monads, or equivalently to internal categories. We also show how to recover classical descent along a morphism and Galois descent. Finally, in Section~\ref{sec:ex}, we illustrate some further examples.

\subsection*{Notations}

Throughout this paper, we fix $\C$ a category with finite limits. 

Moreover, we will often denote morphisms in slice categories in the same way as their underlying morphisms in the category.

\section{Preliminaries}
\label{secprelim}

In this section, we recall some preliminaries and notations that we will use throughout the paper. The guiding point is that the descent constructions considered in
this paper will be expressed by comparison functors into categories of algebras for suitable monads.

We recall monads and the two kinds of morphisms between them. We then recall internal categories and their actions. The category
of actions will be useful later in this paper because it will give the category of descent data.

We review classical descent along a morphism, following the monadic approach derived by \cite{Beck,BenabouRoubaud,StreetDescent}. We also recall the Galois form of descent, associated with actions of group objects. Later in the paper, we will merge this form of descent with the classical descent along a morphism.

Finally, we recall a useful pullbacks lemma that will be used in the following sections.

\subsection{Monad morphisms, comorphisms, and algebras}
\label{subsec:monad-morphisms-comorphisms} For monads and Eilenberg--Moore algebras, we refer the reader to \cite{Borceux2,StreetFormalTheory}. Given a monad $T$ on $\mathcal E$, we write $\Alg(T)$ for the
Eilenberg--Moore category of $T$-algebras. Thus an object of $\Alg(T)$ is a
pair $(X,a)$, where $X\in\mathcal E$ and $a:TX\to X$ is a $T$-algebra
structure, and morphisms are $T$-algebra morphisms.

We recall here that monads give rise to two categories, depending on the direction of the 2-cell considered in the morphisms between them. Both the directions will be useful later in the paper. With the convention
used here, \emph{morphisms of monads} are those that induce functors between Eilenberg--Moore categories. \emph{Comorphisms} induce functors between the Kleisli categories, but they also induce functors between Eilenberg--Moore categories assuming that their underlying functor is a left adjoint. The idea is that we can turn such comorphisms of monads into morphisms of monads, via the calculus of mates.

We will often write monads as pairs $(\mathcal E,T)$,
where $\mathcal E$ is a category and $T$ is a monad on $\mathcal E$. We will need morphisms that also allow us to change the base category of monads.

\begin{defi}\label{defmormnd}
    A \emph{morphism of monads} $(\mathcal E,T)\longrightarrow(\mathcal C,S)$ consists of a pair $(F,\lambda)$ where $F:\mathcal E\to\mathcal C$ is a functor (called the \emph{underlying functor}) and $\lambda$ is a natural transformation
    \begin{cd}
        {\mathcal{E}} \& {\mathcal{E}} \\
	{\mathcal{C}} \& {\mathcal{C}}
	\arrow["T", from=1-1, to=1-2]
	\arrow["F"', from=1-1, to=2-1]
	\arrow["F", from=1-2, to=2-2]
	\arrow["\lambda"'{inner sep=0.35ex}, Rightarrow, from=2-1, to=1-2,shorten <=2ex,shorten >=2ex]
	\arrow["S"', from=2-1, to=2-2]
    \end{cd}
    such that the following axioms hold:
    \begin{eqD*}
        \begin{cd}*
        {\mathcal{E}} \& {\mathcal{E}} \\
	{\mathcal{C}} \& {\mathcal{C}}
	\arrow["T", from=1-1, to=1-2]
	\arrow["F"', from=1-1, to=2-1]
	\arrow["F", from=1-2, to=2-2]
	\arrow["\lambda"'{inner sep=0.35ex}, Rightarrow, from=2-1, to=1-2,shorten <=2ex,shorten >=2ex]
	\arrow[""{name=0, anchor=center, inner sep=0}, "S", from=2-1, to=2-2]
	\arrow[""{name=1, anchor=center, inner sep=0}, "{\Id{\mathcal{C}}}"', curve={height=30pt}, from=2-1, to=2-2]
	\arrow["{\eta^S}"', between={0.2}{0.8}, Rightarrow, from=1, to=0]
    \end{cd} \quad = \quad 
    \begin{cd}*
        {\mathcal{E}} \& {\mathcal{E}} \\
	{\mathcal{C}} \& {\mathcal{C}}
	\arrow[""{name=0, anchor=center, inner sep=0}, "T", curve={height=-30pt}, from=1-1, to=1-2]
	\arrow[""{name=1, anchor=center, inner sep=0}, "{\Id{\mathcal{E}}}"', from=1-1, to=1-2]
	\arrow["F"', from=1-1, to=2-1]
	\arrow["F", from=1-2, to=2-2]
	\arrow["{\Id{\mathcal{C}}}"', from=2-1, to=2-2]
	\arrow["{\eta^T}"', between={0.2}{0.8}, Rightarrow, from=1, to=0]
    \end{cd}
    \end{eqD*}
    \begin{eqD*}
    \begin{cd}*
        {\mathcal{E}} \&\& {\mathcal{E}} \\
	{\mathcal{C}} \&\& {\mathcal{C}} \\[-4ex]
	\& {\mathcal{C}}
	\arrow["T", from=1-1, to=1-3]
	\arrow["F"', from=1-1, to=2-1]
	\arrow["F", from=1-3, to=2-3]
	\arrow["\lambda", Rightarrow, from=2-1, to=1-3,shorten <=4ex,shorten >=4ex]
	\arrow[""{name=0, anchor=center, inner sep=0}, "S", from=2-1, to=2-3]
	\arrow["S"', from=2-1, to=3-2]
	\arrow["S"', from=3-2, to=2-3]
	\arrow["{\mu^S}"', between={0}{0.8}, Rightarrow, from=3-2, to=0,shorten <=0.5ex]
    \end{cd} \quad = \quad
        \begin{cd}*
        {\mathcal{E}} \& {\mathcal{E}} \& {\mathcal{E}} \\
	{\mathcal{C}} \& {\mathcal{C}} \& {\mathcal{C}}
	\arrow["T", from=1-1, to=1-2]
	\arrow[""{name=0, anchor=center, inner sep=0}, "T", curve={height=-30pt}, from=1-1, to=1-3]
	\arrow["F"', from=1-1, to=2-1]
	\arrow["T", from=1-2, to=1-3]
	\arrow["F", from=1-2, to=2-2]
	\arrow["F", from=1-3, to=2-3]
	\arrow["\lambda"'{inner sep=0.35ex}, Rightarrow, from=2-1, to=1-2,shorten <=2ex,shorten >=2ex]
	\arrow["S"', from=2-1, to=2-2]
	\arrow["\lambda"'{inner sep=0.35ex}, Rightarrow, from=2-2, to=1-3,shorten <=2ex,shorten >=2ex]
	\arrow["S"', from=2-2, to=2-3]
	\arrow["{\mu^T}", between={0}{0.8}, Rightarrow, from=1-2, to=0]          
        \end{cd}
    \end{eqD*}
    where $\eta^T,\mu^T$ and $\eta^S,\mu^S$ denote the units and multiplications of the monads $T$ and $S$.
\end{defi}

We denote as $\Mndmo$ the category of monads and morphism of monads.

Morphisms of monads $(F,\lambda)$ enjoy the following property. If $(X,a)$ is a $T$-algebra, then
$FX$ becomes an $S$-algebra, exhibited by the composite
$$
SFX\xrightarrow{\lambda_X}FTX\xrightarrow{F a}FX.
$$
Thus a morphism $(F,\lambda):(\mathcal E,T)\to(\mathcal C,S)$ induces a
functor
$$
\Alg(T)\longrightarrow\Alg(S).
$$

\begin{defi}
A \emph{comorphism of monads} $(\mathcal E,T)\longrightarrow(\mathcal C,S)$ consists of a pair $(F,\chi)$ where $F:\mathcal E\to\mathcal C$ is a functor (called the \emph{underlying functor}) and $\chi$ is a natural transformation
\begin{cd}
    {\mathcal{E}} \&{\mathcal{E}} \\
	{\mathcal{C}} \&{\mathcal{C}}
	\arrow["T", from=1-1, to=1-2]
	\arrow["F"', from=1-1, to=2-1]
	\arrow["\chi"'{inner sep=0.35ex}, Rightarrow, from=1-2, to=2-1,shorten <=2ex, shorten >=2ex]
	\arrow["F", from=1-2, to=2-2]
	\arrow["S"', from=2-1, to=2-2]
\end{cd}
such that dual axioms to the ones depicted in \defx\ref{defmormnd} hold.
\end{defi}

We denote as $\Mndco$ the category of monads and comorphisms of monads. 

A comorphism of monads does not by itself induce a functor between Eilenberg--Moore
categories, unlike morphisms of monads. However, it does
so when its underlying functor is a left adjoint. We therefore denote by
$\Mndcoadj$ the subcategory of $\Mndco$ whose morphisms are those
comorphisms of monads $(F,\chi):(\mathcal E,T)\longrightarrow(\mathcal C,S)
$ for which $F$ is a left adjoint.

Let such a comorphism be given, and choose an adjunction
$$
F:\mathcal E\rightleftarrows\mathcal C:G.
$$
Let $\iota:\Id{\mathcal E}\Rightarrow GF$ and
$\varepsilon:FG\Rightarrow\Id{\mathcal C}$ be the unit and counit of this
adjunction. The comorphism structure natural transformation $\chi:F T\Rightarrow S F$ has a mate
$$
\chi^\sharp\:T G\Longrightarrow G S,
$$
given at $Y\in\mathcal C$ by the composite
$$
TGY
\xrightarrow{\iota_{TGY}}
GFTGY
\xrightarrow{G\chi_{GY}}
GSFGY
\xrightarrow{GS\varepsilon_Y}
GSY.
$$
The usual calculus of mates shows that $\chi^\sharp$ is compatible with the
units and multiplications. Hence
$$
(G,\chi^\sharp):(\mathcal C,S)\longrightarrow(\mathcal E,T)
$$
is a morphism of monads in $\Mndmo$.

It follows that an arrow in $\Mndcoadj$ induces a functor between
Eilenberg--Moore categories in the opposite direction. Namely, the above
comorphism induces a functor $
\Alg(S)\longrightarrow\Alg(T)$, sending an $S$-algebra $(Y,b:S Y\to Y)$ to the $T$-algebra
$$
\bigl(GY,\;G b\circ\chi^\sharp_Y:TGY\to GY\bigr).
$$
After choosing right adjoints for the left adjoints occurring in
$\Mndcoadj$, this yields a functor
$$
\Alg(-):(\Mndcoadj)^{\op}\longrightarrow\Cat.
$$

\subsection{Internal categories and actions}

Let $\C$ be a category with pullbacks. Internal
categories and internal functors in $\C$ are standard notions; see, for
example, \cite{Borceux2,KellyStreet2Cat,StreetFormalTheory}. We write
$\Cat(\C)$ for the category of internal categories in $\C$ and internal
functors between them. An object of $\Cat(\C)$ will often be denoted
$
\DD=(D_1\rightrightarrows D_0),
$
with source and target maps $s,t:D_1\to D_0$, identity map
$e:D_0\to D_1$, and composition map
$
m:D_1\times_{s,D_0,t}D_1\to D_1.
$
The pullback $D_1\times_{s,D_0,t}D_1$ abstractly captures composable pairs of morphisms $(g,f)$.

\cite{JanelidzeTholenII} introduced a very interesting notion: that of action of an internal category. This generalizes the well-known actions of group objects and opens the way to a broader theory of actions. Actions of internal categories will be one of the central notions of this paper, and a key element for our applications to descent theory.

\begin{defi}\label{defactint}
Let $\DD=(D_1\rightrightarrows D_0)$ be an internal category in $\C$, with structure maps $s,t:D_1\to D_0$, $e:D_0\to D_1$, and $m:D_1\times_{s,D_0,t}D_1\to D_1$ as described above.

A \emph{(left) action} of $\DD$ is a triple $(C,\gamma,\xi)$ consisting of
\begin{itemize}
\item an object $C$ of $\C$;
\item a morphism $\gamma:C\to D_0$;
\item a morphism
$$
\xi:D_1\times_{s,D_0,\gamma} C\longrightarrow C
$$
\end{itemize}
such that the following three axioms are satisfied:

\begin{enumerate}
\item The action lies over the target map:
$$
\gamma\xi=t\pi_1,
$$
where $\pi_1:D_1\times_{s,D_0,\gamma}C\to D_1$ is the first projection.

\item Identity arrows act trivially:
$$
\xi(e\gamma,\id{C})=\id{C}.
$$
\item The action is compatible with composition:
$$
\xi(m,\id{C})=\xi(\id{D_1},\xi).
$$
\end{enumerate}

In elementwise notation, if $f:y\to x$, $g:x\to z$, and $\gamma(c)=y$, then
the last axiom says
$$
\xi(gf,c)=\xi(g,\xi(f,c)).
$$
\end{defi}

\begin{defi}
Let $(C,\gamma,\xi)$ and $(C',\gamma',\xi')$ be actions of $\DD$. A
\emph{morphism of actions}
$$
h:(C,\gamma,\xi)\longrightarrow(C',\gamma',\xi')
$$
is a morphism $h:C\to C'$ in $\C$ such that $\gamma'h=\gamma$ and
$$
h\xi=\xi'(\id{D_1}\times h).
$$
Thus $h$ is a morphism over $D_0$ which commutes with the action of every
arrow of $\DD$.
\end{defi}

Actions of $\DD$ and morphisms of actions form a category, denoted $\Act[\C]{\DD}$.

\begin{exampl}
    Every group object $G$ in a category $\C$ with a terminal object $1$ can be seen as a special kind of internal category in $\C$, for which the object of objects $D_0$ coincides with $1$. It is straightforward to see that actions of $G$ seen as an internal category precisely coincides with the usual actions of $G$ seen as a group object. In particular, the pullback $D_1\times_{s,D_0,\gamma} C$ simplifies to the product $G\x C$.
\end{exampl}

\begin{exampl}\label{exaactionsofinternalcatinSet}
Let $\DD$ be a small category, regarded as an internal category in $\Set$.
Then an action of $\DD$ is a set $C$ equipped with a map
$\gamma:C\to D_0$ and, for each arrow $f:y\to x$ of $\DD$, a function
$$
\xi_f:C_y\longrightarrow C_x,
$$
where $C_x=\gamma^{-1}(x)$. The identity and associativity axioms say exactly
that
$$
\xi_{\id{x}}=\id{C_x}
\qquad\text{and}\qquad
\xi_{gf}=\xi_g\xi_f.
$$
It is well-known (see \cite{JanelidzeTholenII}) that an action of $\DD$ in $\Set$ is the same thing as a copresheaf
$\DD\to\Set$.

We have an equivalence of categories
$$
\Act[\Set]{\DD}\simeq \Set^\DD.
$$
Of course one can obtain presheaves as well by taking actions of the opposite category $\DD$.

This example lets us think of actions of internal categories in general as a notion of internal copresheaves. 
\end{exampl}

\subsection{Classical descent along a morphism}

Descent theory studies the following question: when can an object be
reconstructed from a local form together with compatibility data? In the
classical situation, the local form is obtained by pulling back along a
morphism. The compatibility data live on the self-overlap of that morphism.
This is the \vvv{C}ech point of view on descent; see
\cite{BenabouRoubaud,Beck,JanelidzeTholenI,JanelidzeTholenII,StreetDescent}. Thanks to theory developed by B\'{e}nabou--Roubaud \cite{BenabouRoubaud}, the classical theory of descent along a morphism has been captured from a monadic point of view. We recall below the main ideas of this approach, as this is the point of view that is the most useful in this paper.

Let $\C$ be a category with pullbacks, and let $p\:E\to B$ be a morphism in $\C$.
Pullback along $p$ gives a functor
$$
p^\ast:\slice{\C}{B}\longrightarrow \slice{\C}{E}.
$$
An object over $B$ may therefore be inspected locally over $E$. The question
is whether an object over $E$ is obtained by pullback from an object over $B$. The self-overlap of $E$ over $B$, that expresses the compatibility of the local form, is the kernel-pair $E\times_B E$ of $p$. Call $\pi_1$ and $\pi_2$ the pullback projections.

A key element for descent along a morphism $p$ is that the pullback functor $p^\ast$ has a left adjoint  
$$
p_!:\slice{\C}{E}\longrightarrow \slice{\C}{B},
$$
given by post-composition with $p$. The adjunction $
p_!\dashv p^\ast$ then gives rise to a monad on $\slice{\C}{E}$:
$$
T_p=p^\ast p_!.
$$
Using the pullbacks lemma, it is straightforward to see that, equivalently, the monad $T_p$ is given by
$$
T_p=\pi_{1!}\pi_2^\ast.
$$
Indeed, given $\gamma:X\to E$ an object of $\slice{\C}{E}$, we have that
$$
T_p(X)
=
(E\times_B E)\times_{\pi_2,E,\gamma}X,
$$
viewed as an object over $E$ through the first projection
$\pi_1:E\times_B E\to E$.

\begin{defi}
The category of descent data along the morphism $p$ is
$$
\Desc(p)=\Alg(T_p).
$$
Thus a descent datum is an object $\gamma:X\to E$ together with a
$T_p$-algebra structure
$$
\theta:T_p(X)\longrightarrow X.
$$
\end{defi}

Unwinding the definition, $\theta$ is the operation which compares the two
pullbacks of $X$ to the overlap $E\times_B E$. The unit axiom says that the
comparison is the identity along the diagonal $E\to E\times_B E$. The
multiplication axiom says that the comparisons are coherent on the triple
overlap $E\times_B E\times_B E$. In the usual notation, this is the cocycle
condition: going directly from the first copy to the third copy agrees with
going through the second.

Every object over $B$ has a canonical descent datum. Indeed, if $Y\to B$ is
given, then $p^\ast Y\to E$ has two pullbacks to $E\times_B E$, and they are
canonically identified because $p\pi_1=p\pi_2$. This defines the comparison
functor
$$
K_p\:\slice{\C}{B}\longrightarrow \Desc(p).
$$

\begin{defi}
The morphism $p\:E\to B$ is called:
\begin{enumerate}
\item a \emph{descent morphism} if $K_p$ is fully faithful;
\item an \emph{effective descent morphism} if $K_p$ is an equivalence.
\end{enumerate}
\end{defi}

Thus descent means that morphisms between global objects are recovered from
their local pullbacks, provided these morphisms respect the descent data.
Effective descent means that every local object equipped with coherent
overlap data is induced by a global object, uniquely up to unique
isomorphism.

When the relevant coequalizers exist, descent data can be realized by gluing. This simplifies checking whether the comparison functor $K_p$ is an equivalence of categories. For a descent datum $(X,\theta)$, the realization $\Phi_p(X,\theta)$ is given
by the coequalizer
$$
\xymatrix{
T_p(X)
\ar@<0.5ex>[r]^-{\theta}
\ar@<-0.5ex>[r]_-{q}
&
X
\ar[r]
&
\Phi_p(X,\theta),
}
$$
where $q:T_p(X)\to X$ is the projection to the second factor. The resulting
object lies over $B$. This construction gives the usual realization functor
$$
\Phi_p\:\Desc(p)\longrightarrow\slice{\C}{B}.
$$
Under the stated hypotheses, $\Phi_p$ is left adjoint to $K_p$.

\subsection{Galois descent}\label{subsec:Galoisrecall}

Galois descent, as developed in \cite{MarquesGaloisDescent}, has the same shape of the theory of descent recalled above, but the source of the compatibility data for local forms is different. In ordinary descent, compatibility is imposed on overlaps
$E\times_B E$. In Galois descent, compatibility is imposed along symmetries. This point of
view is standard in categorical approaches to descent; see
\cite{JanelidzeTholenI,JanelidzeTholenII,StreetDescent}.

Let $G\to B$ be a group object in the slice category $\C/B$, and suppose that $G$ acts on
$p\:E\to B$ by a morphism
$$
h:G\times_B E\longrightarrow E
$$
over $B$. This action determines the action groupoid
$$
G\ltimes E=\bigl(G\times_B E\rightrightarrows E\bigr),
$$
whose source is the projection $\pi_2:G\times_B E\to E$ and whose target is
$h$. Thus an element $(g,e)$ is regarded as an arrow
$$
e\longrightarrow g\cdot e.
$$
The actions of the internal groupoid $G\ltimes E$ then play the role of descent data, associated to the action of the group object $G$ on $p$.

\begin{defi}
The category of $G$-descent data over $E$ is
$$
\Desc_G(E)= \Act[\C]{G\ltimes E}.
$$
\end{defi}

Thus an object of $\Desc_G(E)$ is an object $X\to E$ together with a
compatible action of $G$ on $X$ over the given action of $G$ on $E$. In other
words, $X$ is a local object over $E$ equipped with symmetry data.

There is again a comparison functor
$$
K_G:\slice{\C}{B}\longrightarrow \Desc_G(E).
$$
It sends an object $Y\to B$ to its pullback $E\times_B Y\to E$. The group
$G$ acts on $E\times_B Y$ through its action on the first factor:
$$
g\cdot(e,y)=(g\cdot e,y).
$$
This gives the canonical $G$-descent datum attached to $Y$.

\begin{defi}
The action of $G$ on $p\:E\to B$ is said to satisfy:
\begin{enumerate}
\item \emph{Galois descent} if $K_G$ is fully faithful;
\item \emph{effective Galois descent} if $K_G$ is an equivalence.
\end{enumerate}
\end{defi}
The relation with classical descent along a morphism is controlled by the torsor map
$$
G\times_B E\longrightarrow E\times_B E,
\qquad
(g,e)\longmapsto (g\cdot e,e).
$$
When this map is an isomorphism, the action groupoid $G\ltimes E$ is
identified with the \vvv{C}ech groupoid $\Eq(p)$ of $p\:E\to B$, given by the kernel pair of the morphism $p$. Equivalently, $(E,h)$ is a \emph{$G$-pseudotorsor} in the sense of \cite{MarquesGaloisDescent}. It is shown in \cite{MarquesGaloisDescent} that then effective Galois descent for the action $h$ of $G$ on $p\:E\to B$ can be reduced to the condition that the unique morphism from $E$ to the terminal object is an effective descent morphism in the classical sense. In fact, more is true. Galois effective descent for $h$ holds if and only if the torsor map is an isomorphism and $E\to 1$ is an effective descent morphism. This is proved in \cite[{\thex{3.7}}]{MarquesGaloisDescent}. Interestingly, similar rigidity criteria do not hold for descent, as opposed to effective descent.

\begin{rem}
    We will develop a notion of generalized descent that captures and generalizes both classical descent along a morphism and Galois descent, while satisfying as many as possible of the good properties enjoyed by those. For this, we need to combine the monadic approach of descent along a morphism with the approach of Galois descent using actions of an internal groupoid.
\end{rem}

\subsection{A useful pullbacks lemma} The following lemma involving pasting of pullbacks will be useful in the following sections. We can interpret it as an associativity result for pullbacks. It is well-known that this is actually captured by the fact that the codomain bifibration $\cod\:\C^2\to \C$ satisfies the Beck--Chevalley condition. This observation will be very useful to us in Section \ref{sec:general-fibrational-setting}, when we will explore generalizations of our theory to the general fibrational setting. We will recall the Beck--Chevalley condition in section \ref{sec:general-fibrational-setting}.

\begin{lem}\label{lem:assoc-pb}
Assume we are given morphisms
\[
A \xrightarrow{\,p\,} D,\qquad
B \xrightarrow{\,q\,} D,\qquad
B \xrightarrow{\,r\,} E,\qquad
C \xrightarrow{\,s\,} E.
\]
in the category $\C$. Then there is a canonical isomorphism $\psi$ given by the composite
\[
(A\times_D B)\times_E C \iso R \iso A\times_D (B\times_E C),
\]
where $R$ is the pullback of $p_2$ and $q_1$ as depicted in the following diagram:
\begin{cd}[6][6]
R \PB{rd} \ar[r,"r_2"] \ar[d,"r_1"'] \& B\times_E C \PB{rd}\ar[r,"q_2"] \ar[d,"q_1"'] \& C \ar[d,"f"] \\
A\times_D B \PB{rd} \ar[r,"p_2"'] \ar[d,"p_1"'] \& B \ar[r,"g"'] \ar[d,"h"'] \& E \\
A \ar[r,"k"'] \& D \&
\end{cd}
Moreover, the isomorphism $\psi$ appropriately commutes with the relevant projections.
\end{lem}
\begin{proof}
By the usual pullbacks lemma, we have that $R$ coincides with the pullback of $f$ along the composite $g\c \pi_2$. That is, there is a canonical isomorphism between $R$ and the pullback $(A\times_D B)\times_E C$. Moreover, this canonical isomorphism appropriately commutes with the relevant projections. Similarly, again by the pullbacks lemma, $R$ also coincides with the pullback of $h\c q_1$ along $k$. So that we have a canonical isomorphism from $R$ to $A\times_D (B\times_E C)$. Again, this isomorphism appropriately commutes with the relevant projections. Composing the two canonical isomorphisms, we conclude.
\end{proof}

It is well-known that the result above is actually a consequence of the fact that the codomain fibration $\cod\:\C^2\to \C$ is a bifibration that satisfies the Beck-Chevalley condition at every pullback square.

\section{Monadicity of actions of internal categories}\label{secmonadicity}

In this section, we prove a monadicity result for actions of internal categories. More precisely, given any internal category 
$\DD=($\begin{cd}*
    D_1 \arrow[r,shift left=1.7,"s"]\arrow[r,shift right=1.7,"t"']\& D_0 \arrow[l,"i"{description}]
\end{cd}$)$ in $\C$, we prove that the category of actions of $\DD$ is the category of Eilenberg--Moore algebras for a monad on $\C/{D_0}$. This result generalizes the fundamental well-known monadicity result for actions of group objects.

The monadicity theorem presented in this section will also be very useful for the applications of our theory to descent. Indeed, this will allow us to combine the monadic approach to classical descent along a morphism with the approach of Galois descent based on actions of an internal groupoid.

\begin{cons}\label{consTD}
    Let $\DD=($\begin{cd}*
    D_1 \arrow[r,shift left=1.7,"s"]\arrow[r,shift right=1.7,"t"']\& D_0 \arrow[l,"i"{description}]
\end{cd}$)$ be an internal category in $\C$. We construct the underlying functor of the monad that will capture the actions of $\DD$:
    \begin{fun}
 	\T_{\DD} & \: & \slice{\C}{D_0} & \too & \slice{\C}{D_0}\\[1ex]
    && \begin{cd}*[3]
         {C} \arrow[d,"\gamma"]\\
         {D_0}
     \end{cd} & \mto & \begin{cd}*[3]
         {D_1\x[D_0] C} \arrow[d,"\pi_1"]\\
         {D_1} \arrow[d,"t"]\\
         {D_0}
        \end{cd} \\
    && \begin{cd}*[4][1]
    C \&\& {C'} \\
	\& {D_0}
	\arrow["f", from=1-1, to=1-3]
	\arrow["\gamma"'{pos=0.35}, from=1-1, to=2-2]
	\arrow["{\gamma'}"{pos=0.35}, from=1-3, to=2-2]
\end{cd} & \mto & \begin{cd}*[2.5][2]
    {D_1\x[D_0] C} \& {D_1\x[D_0] C'} \\
	{D_1} \& {D_1} \\
	{D_0} \& {D_0}
	\arrow["{\id{}\x f}"{inner sep=1ex}, from=1-1, to=1-2]
	\arrow["{\pi_1}"', from=1-1, to=2-1]
	\arrow["{\pi_1}", from=1-2, to=2-2]
	\arrow[equals, from=2-1, to=2-2]
	\arrow["t"', from=2-1, to=3-1]
	\arrow["t", from=2-2, to=3-2]
	\arrow[equals, from=3-1, to=3-2]
    \end{cd}
\end{fun}
where $\pi_1$ is given by the pullback
\begin{cd}
    {D_1\x[D_0] C}\PB{rd} \& C \\
	{D_1} \& {D_0}
	\arrow["{\pi_2}", from=1-1, to=1-2]
	\arrow["{\pi_1}"', from=1-1, to=2-1]
	\arrow["\gamma", from=1-2, to=2-2]
	\arrow["s"', from=2-1, to=2-2]
\end{cd}
and similarly for $\gamma'$ in the place of $\gamma$. Notice that we write $D_1\x[D_0] C$ for the pullback of $\gamma$ along the source map $s$.

It is straightforward to show that $\T_{\DD}$ is a functor, using the universal property of the pullback. 
\end{cons}

\begin{theo} \label{thm:action-monad-algebras}
Let $\DD=($\begin{cd}*
    D_1 \arrow[r,shift left=1.7,"s"]\arrow[r,shift right=1.7,"t"']\& D_0 \arrow[l,"i"{description}]
\end{cd}$)$ be an internal category in $\C$. The functor 
\begin{fun}
 	\T_{\DD} & \: & \slice{\C}{D_0} & \too & \slice{\C}{D_0}\\[1ex]
     && (C\aar{\gamma}D_0) & \mto & (D_1\x[D_0] C \aar{\pi_1} D_1\aar{t}D_0)
\end{fun}
described in \conx\ref{consTD} extends to a monad $\T_{\DD}=(\T_{\DD},\eta^{\DD},\mu^{\DD})$. Moreover, the category $\Algs{\T_{\DD}}$ of Eilenberg--Moore algebras for $\T_{\DD}$ and morphisms between them is isomorphic to $\Act[\C]{\DD}$.
\end{theo}
\begin{proof}
We first prove that the functor $\T_{\DD}$ extends to a monad $\T_{\DD}=(\T_{\DD},\eta^{\DD},\mu^{\DD})$.
    The unit $\eta^{\DD}\:\Id{}\aR{} \T_{\DD}$ is defined to have general component on $\gamma\:C\to D_0$ induced by the universal property of the pullback $D_1\x[D_0] C$, as in the following diagram:
    \begin{cd}[4][5]
        C \&[-2.5ex]\\
	\& {D_1\x[D_0] C} \PB{rdd}\& C \\[-4ex]
	{D_0} \\[-3ex]
	\& {D_1} \& {D_0} \\
	\& {D_0}
	\arrow["{\exists ! \eta^{\DD}_{\gamma}}"', dashed, from=1-1, to=2-2]
	\arrow[curve={height=-12pt}, equals, from=1-1, to=2-3]
	\arrow["\gamma"', from=1-1, to=3-1]
	\arrow["{\pi_2}"', from=2-2, to=2-3]
	\arrow["{\pi_1}", from=2-2, to=4-2]
	\arrow["\gamma", from=2-3, to=4-3]
	\arrow["i"', from=3-1, to=4-2]
	\arrow[curve={height=12pt}, equals, from=3-1, to=5-2]
	\arrow["s"', from=4-2, to=4-3]
	\arrow["t", from=4-2, to=5-2]
    \end{cd}
    It is easy to show that $\eta^{\DD}$ is a natural transformation.

    The multiplication $\mu^{\DD}\:\T_{\DD}\c \T_{\DD}\aR{}\T_{\DD}$ is defined to have general component on $\gamma\:C\to D_0$ given by the composite $[\mu^{\DD}_{\gamma}]\c j$ where $j$ is the canonical isomorphism 
    $$D_1\x[D_0] (D_1\x[D_0] C) \iso (D_1\x[D_0] D_1)\x[D_0] C$$ described in \lemx\ref{lem:assoc-pb} and $[\mu^{\DD}_{\gamma}]$ is given by the universal property of the pullback $D_1\x[D_0] C$ as in the following diagram:
\begin{cd}[4][4]
	{D_1\x[D_0] (D_1\x[D_0] C)} \&[-4ex]\&[-1ex] {D_1\x[D_0] C} \\[-3ex]
	\& {(D_1\x[D_0] D_1)\x[D_0] C} \\
	\&\& {D_1\x[D_0] C} \PB{rdd} \& C \\[-4ex]
	\& {D_1\x[D_0] D_1} \\[-3ex]
	\&\& {D_1} \& {D_0} \\[-4ex]
	\& {D_1} \\[-3ex]
	\&\& {D_0}
	\arrow["{\pi_2}", from=1-1, to=1-3]
	\arrow[iso,"j"', from=1-1, to=2-2]
	\arrow["{\pi_1}"', curve={height=12pt}, from=1-1, to=6-2]
	\arrow["{\pi_2}", from=1-3, to=3-4, curve={height=-12pt}]
	\arrow["{\exists ! [\mu^{\DD}_{\gamma}]}"', dashed, from=2-2, to=3-3]
	\arrow["{\pi_2}", curve={height=-14pt}, from=2-2, to=3-4]
	\arrow["{\pi_1}"', from=2-2, to=4-2]
	\arrow["{\pi_2}"', from=3-3, to=3-4]
	\arrow["{\pi_1}", from=3-3, to=5-3]
	\arrow["\gamma", from=3-4, to=5-4]
	\arrow["\comp"', from=4-2, to=5-3]
	\arrow["{\pi_2}"', from=4-2, to=6-2]
	\arrow["s"', from=5-3, to=5-4]
	\arrow["t", from=5-3, to=7-3]
	\arrow["t"', from=6-2, to=7-3]
    \end{cd}
    It is straightforward to prove that $\mu$ is natural, using the universal property of the pullbacks involved. One can then check that $\T_{\DD}=(\T_{\DD},\eta^{\DD},\mu^{\DD})$ satisfies the monad axioms using associativity and unitality of the composition of the internal category $\DD$.

    An algebra for the monad $\T_{\DD}$ is a pair 
    \begin{eqD*}
      ( \begin{cd}*
      C \arrow[d, "\gamma"]\\[-2ex]
      D_0
      \end{cd} \hspace{2ex} ,  \hspace{2ex}  \begin{cd}*[4][4]
      {D_1 \times_{D_0} C } \&[-3ex] \& [-3ex]\& [-0.5ex] C \\[-2ex]
	\& {D_1} \\[-2ex]
	\&\& {D_0}
	\arrow["\xi", from=1-1, to=1-4]
	\arrow["{\pi_1}"', from=1-1, to=2-2]
	\arrow["\gamma", from=1-4, to=3-3]
	\arrow["t"', from=2-2, to=3-3]
      \end{cd} )
    \end{eqD*}
    such that $\xi \circ \eta^{\DD}_{\gamma}$ is $\id{C}$ seen as the identity of $\gamma$ in the slice category $\C/ {D_0}$ and $\xi \c \T_{\DD}(\xi)= \xi \c \mu^{\DD}_{\gamma}$ as morphisms in $\C/ {D_0}$. This precisely corresponds to the action $(C, \gamma, \xi)$ of $\DD$ on the object $C$ (see Definition \ref{defactint}). Indeed, the first axiom of action is precisely the fact that $\xi$ is a morphism in the slice category $\C/D_0$ from $t\c \pi$ to $\gamma$. The third axiom of action is then precisely the requirement that $\gamma \circ \eta^{\DD}_{\gamma}$ is $\id{C}$ seen as the identity of $\gamma$ in the slice category $\C/ {D_0}$ since $\eta_{\gamma}=\langle i \c \gamma, \id{C} \rangle$. Moreover, the second axiom of action is precisely the requirement that $\xi \c \T_{\DD}(\xi)= \xi \c \mu^{\DD}_{\gamma}$ as morphisms in $\C/ {D_0}$ since $\T_{\DD}(\xi)=D_1 \times \xi$ and the following square is commutative by definition of $\mu^{\DD}$
    \begin{cd}
        {D_1 \times_{D_0} (D_1 \times_{D_0} C)} \& {(D_1 \times_{D_0} D_1) \times_{D_0} C} \\
	{(D_1 \times_{D_0} D_1) \times_{D_0} C} \& {D_1 \times_{D_0} C.}
	\arrow["j"', from=1-1, to=1-2, iso]
	\arrow["j"', from=1-1, to=2-1, iso]
	\arrow["{[\mu^{\DD}_{\gamma}]}", from=1-2, to=2-2]
	\arrow["{\operatorname{comp} \times \id{C}}"', from=2-1, to=2-2]
    \end{cd}
    Furthermore, a morphism of algebras for the monad $\T_{\DD}$ 
    \begin{eqD*}
      ( \begin{cd}*
      C \arrow[d, "\gamma"]\\[-2ex]
      D_0
      \end{cd}, \begin{cd}*[4][4]
      {D_1 \times_{D_0} C } \&[-3ex] \& [-3ex]\& [-0.5ex] C \\[-2ex]
	\& {D_1} \\[-2ex]
	\&\& {D_0}
	\arrow["\xi", from=1-1, to=1-4]
	\arrow["{\pi_1}"', from=1-1, to=2-2]
	\arrow["\gamma", from=1-4, to=3-3]
	\arrow["t"', from=2-2, to=3-3]
      \end{cd} ) 
      \begin{cd}* {} \arrow[r,"f"] \& {} \end{cd}  
      ( \begin{cd}*
      C' \arrow[d, "\gamma'"]\\[-2ex]
      D_0
      \end{cd} , \begin{cd}*[4][4]
      {D_1 \times_{D_0} C' } \&[-3ex] \& [-3ex]\& [-0.5ex] C' \\[-2ex]
	\& {D_1} \\[-2ex]
	\&\& {D_0}
	\arrow["\xi", from=1-1, to=1-4]
	\arrow["{\pi_1}"', from=1-1, to=2-2]
	\arrow["\gamma'", from=1-4, to=3-3]
	\arrow["t"', from=2-2, to=3-3]
      \end{cd} ) \hspace{3ex} 
    \end{eqD*}
    is a morphism $C \arr{f} C'$ in $\C /D_0$ from $\gamma$ to $\gamma'$ such that $\xi' \c T(f)= f \c \xi$. This precisely corresponds to a morphism $(C,\gamma, \xi) \arr{f} (C', \gamma' \xi')$ between objects with an action of $\DD$ (see Definition \ref{defactint}), thanks to the fact that $\T_{\DD}(f)= \id{D_1} \times f$. 
    
    We have thus proved that the category $\Algs{\T_{\DD}}$ of Eilenberg--Moore algebras for $\T_{\DD}$ and morphisms between them is isomorphic to $\Act[\C]{\DD}$.
    \end{proof}

\begin{remark}
    The monads $\T_{\DD}$ generalize the well-known action monads associated to group objects. Indeed, when $\DD$ is the internal category
    \begin{cd}*
    G \arrow[r,shift left=1.7,"s"]\arrow[r,shift right=1.7,"t"']\& 1 \arrow[l,"i"{description}]
    \end{cd}
    in $\C$ corresponding to a group object $G$ in $\C$, the monad $\T_{\DD}$ becomes the well-known monad 
    \begin{fun}
 	\T_{G} & \: & {\C} & \too & {\C} \\[1ex]
    && X & \mto & G \x X 
\end{fun}
The result that $\Algs{\T_{\DD}}$ is equivalent to $\Act[\C]{\DD}$ generalizes the fundamental result for group objects that $\Algs{\T_G}$ is equivalent to the category of actions of $G$.

For this reason, we will refer to the monads of the form $\T_{\DD}$ associated to an internal category $\DD$ as \dfn{generalized action monads}.
\end{remark}

\section{Functoriality of $\Act[\C]{-}$}\label{sec:functoriality}

Building upon the results of the previous section, we prove that the assignment sending an internal category $\DD$ to its category of actions $\Act[\C]{\DD}$ is functorial. This functoriality is one of the structural ingredients needed for the descent-theoretic applications.

We first prove that the assignment $\T_{\DD}$ of the associated generalized action monad to every internal category $\DD$ in $\C$ is functorial.

\begin{theo}\label{theorTisafunctor}
    The assignment
    {$$\DD=(\text{\begin{cd}*
    D_1 \arrow[r,shift left=1.7,"s"]\arrow[r,shift right=1.7,"t"']\& D_0 \arrow[l,"i"{description}]
    \end{cd}})\in \Cat(\C) \quad \h[3] \mapsto \h[3] \quad \T_{\DD}$$}
    extends to a functor
    $$\T: \Cat(\C) \to \Mndcoadj$$
\end{theo}
\begin{proof}
    We need to associate, to every internal functor in $\C$, a comorphism of monads whose underlying functor is a left adjoint.
    
    Given an internal functor $F\:\DD\to \EE$ exhibited by
    \begin{cd}[5.5]
{D_1 \times_{D_0} D_1} \arrow[r,"{\comp}"] \arrow[d,"{F_1\times F_1}"'] \&[-1ex]
{D_1} \arrow[r,shift left=1.6ex,"s"] \arrow[r,shift right=1.8ex,"t"'] \arrow[d,"{F_1}"'] \&[0.5ex]
{D_0} \arrow[l,"{i}"{description}] \arrow[d,"{F_0}"] \\
{E_1 \times_{E_0} E_1} \arrow[r,"{\comp'}"] \&
{E_1} \arrow[r,shift left=1.6ex,"s'"] \arrow[r,shift right=1.8ex,"t'"'] \&
{E_0} \arrow[l,"{i'}"{description}]
\end{cd}
we define a comorphism of monads $\T_{F}\:\T_{\DD}\to \T_{\EE}$ as follows. Its underlying functor is $F_0\c -\:\C/{D_0}\to \C/{E_0}$. Note that $F_0\c -$ is a left adjoint, of the pullback functor ${F_0}^\ast\:\C/{E_0}\to \C/{D_0}$. We then define the second component of $\T_F$ to be the natural transformation
\begin{cd}
    {\C/{D_0}} \& {\C/{D_0}} \\
	{\C/{E_0}} \& {\C/{E_0}}
	\arrow["{\T_{\DD}}", from=1-1, to=1-2]
	\arrow["{F_0\c -}"', from=1-1, to=2-1]
	\arrow["{\tau_F}", Rightarrow, from=1-2, to=2-1, shorten <=2.7ex,shorten >=2.2ex]
	\arrow["{F_0\c -}", from=1-2, to=2-2]
	\arrow["{\T_{\EE}}"', from=2-1, to=2-2]
\end{cd}
that has component ${(\tau_F)}_{\gamma}$ on $\gamma\:C\to D_0$ given by the following morphism in the slice over $E_0$
\begin{cd}[4]
    {D_1\x[D_0] C} \& {E_1\x[E_0] C} \\
	{D_1} \& {E_1} \\
	{D_0} \& {E_0}
	\arrow["{F_1\x \id{}}", from=1-1, to=1-2]
	\arrow["{\pi_1}"', from=1-1, to=2-1]
	\arrow["{\pi_1}", from=1-2, to=2-2]
	\arrow["{F_1}"', from=2-1, to=2-2]
	\arrow["t"', from=2-1, to=3-1]
	\arrow["{t'}", from=2-2, to=3-2]
	\arrow["{F_0}"', from=3-1, to=3-2]
\end{cd}
It is straightforward to see that ${\tau_F}$ is a natural transformation. Moreover, it is easy to check that $(F_0 \c -, \tau_F)$ is a comorphism of monads using the universal property of the pullbacks involved. It remains to prove that $\T$ preserves identities and composition. If $F$ is the identity, then clearly $F_0 \c -=\id{}$ and $(\tau_F)_{\gamma}=\id{} \times \id{}=\id{}$ for every $\gamma$. So $\T$ preserves identities. Furthermore, given composable internal functors $F\: \DD \to \EE$ and $G\:\EE \to \BB$, we have that $(G_0 \c F_0)\c -=(G_0 \c -) \c (F_0 \c -)$. And $\tau_{G \c F}$ is equal to the pasting of $\tau_{F}$ and $\tau_{G}$ by construction, since $(G_1 \x \id{}) \c (F_1 \x \id{})=(G_1 \c F_1) \x \id{}$. So $\T$ preserves composition.
\end{proof}

Composing with the functor $\Alg(-)\:{(\Mndcoadj)}\op\to \Cat$ of subsection \ref{subsec:monad-morphisms-comorphisms}, we obtain the functoriality of taking the category of actions of an internal category. As explained in Section \ref{secprelim}, in order to associate to a comorphism of monads a functor between the categories of Eilenberg--Moore algebras, it is sufficient to have that the underlying functor of that comorphism of monads is a left adjoint. Indeed, the calculus of mates turns it into a morphism of monads and thus yields a functor between the Eilenberg--Moore categories. Therefore, it is a key property that the comorphism of monads $\T_F$ built above from any internal functor $F$ has as underlying functor a left adjoint.

\begin{theo}
    The assignment
    {$$\DD=(\text{\begin{cd}*
    D_1 \arrow[r,shift left=1.7,"s"]\arrow[r,shift right=1.7,"t"']\& D_0 \arrow[l,"i"{description}]
    \end{cd}})\in \Cat(\C) \quad \h[3] \mapsto \h[3] \quad \Act[\C]{\DD}\in \Cat$$}
    extends to a functor
    $$\Act[\C]{-}: {\Cat(\C)}\op \to \Cat.$$
    More precisely, $\Act[\C]{-}$ is given by the composite of functors
    $${\Cat(\C)}\op\aarr{\T\op} {(\Mndcoadj)}\op\aar{\Alg(-)} \Cat$$
\end{theo}
\begin{proof}
    Thanks to \thex\ref{thm:action-monad-algebras}, the category $\Act[\C]{\DD}$ of actions of an internal category $\DD$ can be captured as the category of Eilenberg--Moore algebras of the generalized action monad $\T_{\DD}$ associated to $\DD$. Then the composite $\Alg(-)\c \T\op$ is a functor extending the assignment $\DD\mto \Act[\C]{\DD}$.
\end{proof}

We provide an explicit description of the functor $\Act[\C]{-}$.

\begin{cons}
Given an internal category $\DD\in \Cat(\C)$, the category $\Act[\C]{\DD}$ has as objects the $\T_{\DD}$-algebras, that are pairs 
\begin{eqD*}
    (\begin{cd}* C \arrow[d,"\gamma"] \\ D_0 \end{cd}, \begin{cd}* {D_1 \x_{D_0} C} \&[-2ex] C \\[-2ex]
	{D_1} \\[-2ex]
	{D_0}
	\arrow["\xi", from=1-1, to=1-2]
	\arrow["{\pi_1}"', from=1-1, to=2-1]
	\arrow["\gamma", from=1-2, to=3-1]
	\arrow["t"', from=2-1, to=3-1]
    \end{cd})
\end{eqD*}
and as morphisms the corresponding morphisms of algebras.
  Furthermore, given an internal functor $F\:\DD\to \EE$ exhibited by
    \begin{cd}[5.5]
{D_1 \times_{D_0} D_1} \arrow[r,"{\comp}"] \arrow[d,"{F_1\times F_1}"'] \&[-1ex]
{D_1} \arrow[r,shift left=1.6ex,"s"] \arrow[r,shift right=1.8ex,"t"'] \arrow[d,"{F_1}"'] \&[0.5ex]
{D_0} \arrow[l,"{i}"{description}] \arrow[d,"{F_0}"] \\
{E_1 \times_{E_0} E_1} \arrow[r,"{\comp'}"] \&
{E_1} \arrow[r,shift left=1.6ex,"s'"] \arrow[r,shift right=1.8ex,"t'"'] \&
{E_0} \arrow[l,"{i'}"{description}]
\end{cd}
its image $\Act[\C]{
F}\: \Act[\C]{\EE}  \to \Act[\C]{\DD}$ sends the $\T_{\EE}$-algebra 
\begin{eqD*}
    (\begin{cd}* C' \arrow[d,"\gamma'"] \\ E_0 \end{cd}, \begin{cd}* {E_1 \x_{E_0} C'} \&[-2ex] C' \\[-2ex]
	{E_1} \\[-2ex]
	{E_0}
	\arrow["\xi'", from=1-1, to=1-2]
	\arrow["{\pi_1}"', from=1-1, to=2-1]
	\arrow["\gamma'", from=1-2, to=3-1]
	\arrow["t'"', from=2-1, to=3-1]
    \end{cd})
\end{eqD*}
to the $\T_{\DD}$-algebra given by $P\aar{F_0\st \gamma'} D_0$, where $P$ is the pullback of $\gamma'$ and $F_0$, together with the unique morphism induced by the universal property of $P$, as shown in the following diagram
\begin{cd}[6][6]
    {D_1 \x_{D_0} P} \& S \& {E_1 \x_{E_0} C'} \\
	{D_1} \& P \& {C'} \\
	{D_0} \& {D_0} \& {E_0}
	\arrow["\tau_F", from=1-1, to=1-2]
	\arrow["{\pi_1}"', from=1-1, to=2-1]
	\arrow["{\exists !}", dashed, from=1-1, to=2-2]
	\arrow["{\T_{\EE}(\operatorname{counit})}", from=1-2, to=1-3]
	\arrow["\beta", from=1-3, to=2-3]
	\arrow["t"', from=2-1, to=3-1]
	\arrow[from=2-2, to=2-3]
	\arrow[from=2-2, to=3-2]
	\arrow["\lrcorner"{anchor=center, pos=0.125}, draw=none, from=2-2, to=3-3]
	\arrow["{\gamma'}", from=2-3, to=3-3]
	\arrow[equals, from=3-1, to=3-2]
	\arrow["{F_0}"', from=3-2, to=3-3]
\end{cd}

where $S$ is the pullback of $F_0 \c F_0 \st \gamma'$ and $s'$ and $\operatorname{counit}$ is the component on $\gamma'$ of counit of the adjunction between $F_0\c -$ and $F_0\st$.
\end{cons}

\section{Characterizing generalized action monads}\label{sec:chargenactmnd}

In this section, we characterize which monads arise as generalized action monads $\T_{\DD}$ for an internal category $\DD$. We exhibit an equivalence of categories between the category $\Cat(\C)$ of internal categories in $\C$ and an appropriate subcategory of the category $\Mndcoadj$ of monads. As it was pointed out to us, this equivalence of categories could also be recovered by the general theory for $T$-multicategories of \cite{Leinster}.

We begin by studying the properties and structures of the generalized action monads $\T_{\DD}$ associated to internal categories $\DD$.

\begin{remark}\label{remintrodalpha}
    Notice that, given any internal category $\DD=($\begin{cd}*
    D_1 \arrow[r,shift left=1.7,"s"]\arrow[r,shift right=1.7,"t"']\& D_0 \arrow[l,"i"{description}]
    \end{cd}$)$ in $\C$, the monad $\T_{\DD}$ is naturally equipped with a natural transformation
    \begin{cd}[4.5]
        {\C/{D_0}} \& {\C/{D_0}} \\
	\C
	\arrow["{\T_{\DD}}", from=1-1, to=1-2]
	\arrow["\dom"', from=1-1, to=2-1]
	\arrow["\alpha"', shift left=5, xshift=-2ex, Rightarrow, from=1-2, to=1-1, shorten <= 2ex, shorten >= 1.7ex]
	\arrow["\dom", from=1-2, to=2-1]
    \end{cd}
    More precisely, the component of $\alpha$ on a general $\gamma\:C\to D_0$ is the morphism $\pi_2$ given by the pullback
    \begin{cd}
    {D_1\x[D_0] C} \PB{rd} \& C \\
	{D_1} \& {D_0}
	\arrow["{\pi_2}", from=1-1, to=1-2]
	\arrow["{\pi_1}"', from=1-1, to=2-1]
	\arrow["\gamma", from=1-2, to=2-2]
	\arrow["s"', from=2-1, to=2-2]
\end{cd}

Furthermore, the natural transformation $\alpha$ satisfies the following compatibility axioms with the unit $\eta^{\DD}$ and the multiplication $\mu^{\DD}$ of the monad $\T_{\DD}$:
\begin{eqD*}
    \begin{cd}*[3.7][3.5]
        {\C /D_0} \&\& {\C /D_0} \\[3ex]
	\& \C
	\arrow[""{name=0, anchor=center, inner sep=0}, curve={height=-16pt}, equals, from=1-1, to=1-3]
	\arrow[""{name=1, anchor=center, inner sep=0}, "\T_{\DD}"'{inner sep =0.35ex}, curve={height=10pt}, from=1-1, to=1-3]
	\arrow[""{name=2, anchor=center, inner sep=0}, "\dom"', from=1-1, to=2-2]
	\arrow[""{name=3, anchor=center, inner sep=0}, "\dom", from=1-3, to=2-2]
	\arrow["\eta^{\DD}"', between={0.2}{0.8}, Rightarrow, from=0, to=1]
	\arrow["\alpha", between={0.2}{0.8}, Rightarrow, from=3, to=2]
    \end{cd} \quad = \quad  \id{}
    \end{eqD*}

    \begin{eqD*}
    \begin{cd}*[6.3][4]
        {\C /D_0} \&\& {\C /D_0} \\
	\& \C
	\arrow[""{name=0, anchor=center, inner sep=0}, "{\T_{\DD}^2}", curve={height=-16pt}, from=1-1, to=1-3]
	\arrow[""{name=1, anchor=center, inner sep=0}, "\T_{\DD}"'{inner sep=0.35ex}, curve={height=10pt}, from=1-1, to=1-3]
	\arrow[""{name=2, anchor=center, inner sep=0}, "\dom"', from=1-1, to=2-2]
	\arrow[""{name=3, anchor=center, inner sep=0}, "\dom", from=1-3, to=2-2]
	\arrow["\mu^{\DD}"{inner sep = 0.8ex}, between={0.2}{0.8}, Rightarrow, from=0, to=1]
	\arrow["\alpha", between={0.2}{0.8}, Rightarrow, from=3, to=2]
    \end{cd} \quad = \quad
\begin{cd}*[6.3][4]
	{\C /D_0} \& {\C /D_0} \& {\C /D_0} \\
	\& \C
	\arrow["\T_{\DD}", from=1-1, to=1-2, curve={height=-8pt}]
	\arrow["\alpha"', shift left=4, Rightarrow, from=1-2, to=1-1, xshift=0.5ex]
	\arrow["\dom"', from=1-1, to=2-2]
	\arrow["\T_{\DD}", from=1-2, to=1-3, curve={height=-8pt}]
	\arrow["\alpha"', shift left=4, Rightarrow, from=1-3, to=1-2,xshift= -0.5ex]
	\arrow["\dom"{description}, from=1-2, to=2-2]
	\arrow["\dom", from=1-3, to=2-2]
\end{cd}
\end{eqD*}
Indeed, the two equalities are precisely given by the two upper commutative triangles formed by the universal property of the pullback in the construction of $\eta$ and $\gamma$ respectively.

Finally, notice that $\alpha$ is a \dfn{cartesian natural transformation}, i.e.\ all naturality squares for $\alpha$ are pullbacks. Indeed, given any morphism 
\begin{cd}*[2.5][-1]
    C \&\& {C'} \\
	\& {D_0}
	\arrow["f", from=1-1, to=1-3]
	\arrow["\gamma"'{pos=0.35}, from=1-1, to=2-2]
	\arrow["{\gamma'}"{pos=0.35}, from=1-3, to=2-2]
\end{cd}
in $\C/{D_0}$, the associated naturality square for $\alpha$ is
\begin{cd}
    {D_1\x[D_0] C} \& C \\
	{D_1\x[D_0] C'} \& {C'}
	\arrow["{\pi_2}", from=1-1, to=1-2]
	\arrow["{\id{}\x f}"', from=1-1, to=2-1]
	\arrow["f", from=1-2, to=2-2]
	\arrow["{\pi_2}"', from=2-1, to=2-2]
\end{cd}
This is a pullback square by the pullbacks lemma, since considering the composite
\begin{cd}
    C \& {C'} \& {D_0} \\
	\& {D_0}
	\arrow["f", from=1-1, to=1-2]
	\arrow["\gamma"', from=1-1, to=2-2]
	\arrow["{\gamma'}", from=1-2, to=1-3]
	\arrow["{\gamma'}"{description}, from=1-2, to=2-2]
	\arrow[equals, from=1-3, to=2-2]
\end{cd}
in $\C/{D_0}$ we obtain
\begin{cd}
    {D_1\x[D_0] C} \& C \\
	{D_1\x[D_0] C'} \& {C'} \\
	{D_1} \& {D_0}
	\arrow["{\pi_2}", from=1-1, to=1-2]
	\arrow["{\id{}\x f}"', from=1-1, to=2-1]
	\arrow["f"', from=1-2, to=2-2]
	\arrow["\gamma", curve={height=-20pt}, from=1-2, to=3-2]
	\arrow["{\pi_2}"', from=2-1, to=2-2]
	\arrow["{\pi_1}"', from=2-1, to=3-1]
	\arrow["{\gamma'}"', from=2-2, to=3-2]
	\arrow["s"', from=3-1, to=3-2]
\end{cd}
and both the bottom square and the outer square are pullback diagrams.
\end{remark}

\begin{prop}\label{propcomorphismofmonads}
    Let $D_0\in \C$ and let $\TT=(\C/{D_0},T,\eta,\mu)$ be a monad on the slice $\C/{D_0}$. A natural transformation
    \begin{cd}[4.5]
        {\C/{D_0}} \& {\C/{D_0}} \\
	\C
	\arrow["{T}", from=1-1, to=1-2]
	\arrow["\dom"', from=1-1, to=2-1]
	\arrow["\alpha"', shift left=5, xshift=-2ex, Rightarrow, from=1-2, to=1-1, shorten <= 2ex, shorten >= 1.7ex]
	\arrow["\dom", from=1-2, to=2-1]
    \end{cd}
    that satisfies the compatibility axioms with $\eta$ and $\mu$
    \begin{eqD*}
    \begin{cd}*[3.7][3.5]
        {\C /D_0} \&\& {\C /D_0} \\[3ex]
	\& \C
	\arrow[""{name=0, anchor=center, inner sep=0}, curve={height=-16pt}, equals, from=1-1, to=1-3]
	\arrow[""{name=1, anchor=center, inner sep=0}, "T"', curve={height=12pt}, from=1-1, to=1-3]
	\arrow[""{name=2, anchor=center, inner sep=0}, "\dom"', from=1-1, to=2-2]
	\arrow[""{name=3, anchor=center, inner sep=0}, "\dom", from=1-3, to=2-2]
	\arrow["\eta"', between={0.2}{0.8}, Rightarrow, from=0, to=1]
	\arrow["\alpha", between={0.2}{0.8}, Rightarrow, from=3, to=2]
    \end{cd} \quad = \quad  \id{}
    \end{eqD*}

    \begin{eqD*}
    \begin{cd}*[6.3][4]
        {\C /D_0} \&\& {\C /D_0} \\
	\& \C
	\arrow[""{name=0, anchor=center, inner sep=0}, "{T^2}", curve={height=-16pt}, from=1-1, to=1-3]
	\arrow[""{name=1, anchor=center, inner sep=0}, "T"', curve={height=12pt}, from=1-1, to=1-3]
	\arrow[""{name=2, anchor=center, inner sep=0}, "\dom"', from=1-1, to=2-2]
	\arrow[""{name=3, anchor=center, inner sep=0}, "\dom", from=1-3, to=2-2]
	\arrow["\mu", between={0.2}{0.8}, Rightarrow, from=0, to=1]
	\arrow["\alpha", between={0.2}{0.8}, Rightarrow, from=3, to=2]
    \end{cd} \quad = \quad
\begin{cd}*[6.3][4]
	{\C /D_0} \& {\C /D_0} \& {\C /D_0} \\
	\& \C
	\arrow["T", from=1-1, to=1-2, curve={height=-8pt}]
	\arrow["\alpha"', shift left=4, Rightarrow, from=1-2, to=1-1, xshift=0.5ex]
	\arrow["\dom"', from=1-1, to=2-2]
	\arrow["T", from=1-2, to=1-3, curve={height=-8pt}]
	\arrow["\alpha"', shift left=4, Rightarrow, from=1-3, to=1-2,xshift= -0.5ex]
	\arrow["\dom"{description}, from=1-2, to=2-2]
	\arrow["\dom", from=1-3, to=2-2]
\end{cd}
\end{eqD*}
    is the same thing as a comorphism of monads 
    $$(\dom,\alpha)\:(\C/{D_0},\TT,\eta,\mu)\to (\C,\Id{\C},\id{},\id{})$$
    
    Moreover, $(\dom,\alpha)$ induces a functor
    $$\C\iso \Algs{\Id{\C}}\to \Algs{\TT}.$$
\end{prop}
\begin{proof}
    The triangle that $\alpha$ fills may be seen as a square with $\Id{\C}$ below. It is readily seen that the axioms of comorphism of monads then precisely match the compatibility axioms with $\eta$ and $\mu$ drawn above. 
The last part of the statement is an application of the functoriality of $\Alg(-)\:{(\Mndcoadj)}\op\to \Cat$. Indeed, $\dom\:\C/{D_0}\to \C$ is a left adjoint, of the functor $\C\to \C/{D_0}$ that sends $C\in \C$ to the projection on second component from the product $C\x D_0$. So $(\dom,\alpha)$ is a morphism in $\Mndcoadj$. The functor $\Alg(-)$ then sends $(\dom,\alpha)$ to a functor $\C\iso \Alg(\Id{\C})\to \Alg(\TT)$, as desired.
\end{proof}

We can thus rephrase the results above in this section as follows.

\begin{cor} \label{comorp}
    Let $\DD=($\begin{cd}*
    D_1 \arrow[r,shift left=1.7,"s"]\arrow[r,shift right=1.7,"t"']\& D_0 \arrow[l,"i"{description}]
    \end{cd}$)$ be an internal category in $\C$. Then the domain functor $\dom\:\C/{D_0}\to \C$ extends to a comorphism of monads 
    $$(\dom,\alpha)\:(\C/{D_0},\T_{\DD},\eta^{\DD},\mu^{\DD})\to (\C,\Id{\C},\id{},\id{})$$
    from the associated generalized action monad $\T_{\DD}$ to the identity, with $\alpha$ a cartesian natural transformation. Moreover, $(\dom,\alpha)$ induces a functor
    $$\C\iso \Algs{\Id{\C}}\to \Algs{\T_{\DD}}.$$
\end{cor}

\begin{cons}
    We give an explicit description of the functor $\C\iso \Algs{\Id{\C}}\to \Algs{\T_{\DD}}$ associated to a starting internal category $\DD$ in $\C$.

    It is straightforward to show that this functor sends $X\in \C$ to the $\T_{\DD}$-algebra given by $D_0 \x X \aar{\pi_1} D_0$ together with the unique morphism $t\x \pi_2$ induced by the universal property of the product $D_0 \x X$ as in the following diagram  
\begin{cd}[6][6]
    {D_1 \x_{D_0} (D_0 \x X)} \& D_0 \x X \& {} \\
	{D_1} \arrow[rd, bend right=20,"t"'] \& {D_0 \x X} \& {X} \\
	{} \& {D_0} \& {}
	\arrow["\alpha_{D_0}\x X", from=1-1, to=1-2]
	\arrow["{\pi_1}"', from=1-1, to=2-1]
	\arrow["{\exists !t \x \pi_2}", dashed, from=1-1, to=2-2]
	\arrow["{\pi_2}", from=1-2, to=2-3, bend left=20]
	\arrow[from=2-2, to=2-3]
	\arrow[from=2-2, to=3-2]
\end{cd}
    Moreover, it is readily seen that the functor sends a morphism $X \aar{f} Y$ in $\C$ to $D_0 \x f$. 
\end{cons}

\begin{remark}
    In the case of actions of a group object $G$ in $\C$, this result recovers the important fact that every object $X$ of $\C$ can be equipped with a trivial $G$-action, given by the second projection $\pi_2\: G\x X\to X$.

    Interestingly, we will show that this property is a key property in determining (generalized) action monads.
\end{remark}

We have seen that the natural transformation $\alpha$ associated to an internal category forms a morphism $(\dom,\alpha)$ in $\Mndcoadj$. In fact, we show that more is true.

\begin{prop}\label{propalphaisimageofQ}
    The morphism $(\dom,\alpha)$ in $\Mndcoadj$ associated to an internal category $\DD$ in $\C$ coincides with the image along $\T\:\Cat(\C)\to \Mndcoadj$ of the unique internal functor from $\DD$ to the terminal internal category $\1$ in $\C$.
\end{prop}
\begin{proof}
It is straightforward to see that the image $\T_{\1}$ of the terminal internal category along $\T$ is the identity. 
    Consider now the unique internal functor $!$ from $\DD$ to the terminal internal category $\1$ in $\C$ is exhibited by
    \begin{cd}[5.5]
{D_1 \times_{D_0} D_1} \arrow[r,"{\comp}"] \arrow[d,"{!}"'] \&[-1ex]
{D_1} \arrow[r,shift left=1.6ex,"s"] \arrow[r,shift right=1.8ex,"t"'] \arrow[d,"{!}"'] \&[0.5ex]
{D_0} \arrow[l,"{i}"{description}] \arrow[d,"{!}"] \\
{1} \arrow[r,"{}", equals] \&
{1} \arrow[r,shift left=1.6ex,"", equals] \arrow[r,shift right=1.8ex,""', equals] \&
{1} \arrow[l, equals]
\end{cd}
Its image along $\T$ is then given by the functor $! \c -$ together with the natural transformation $\tau_{!}$ whose component on $C \aar{\gamma} D_0$ is the morphism $! \times \id{C}\: D_1 \times_{D_0} C \to 1 \times_{1} C$. Thanks to the well-known equivalence of categories between $\C$ and its slice over the terminal object $\C/1$, the functor $! \c -$ is the same as $\dom{}$. Moreover, by construction $(\tau_{!})_{\gamma}$ is equal to the second projection $\pi_2\: D_1 \x_{D_0} C \to C$ and thus it is equal to $\alpha_{\gamma}$. This shows that the image of $!$ along $\T$ is precisely $(\dom{}, \alpha)$.
\end{proof}

The proposition above allows us to give a deeper explanation of why the natural transformation $\alpha$ is cartesian.

\begin{prop}
    Let $F=(F_0,F_1)\:\DD\to \EE$ be an internal functor in $\C$. Then the image of $F$ along $\T\:\Cat(\C)\to \Mndcoadj$ is a comorphism of monads $(F_0\c -, \tau_F)$ with $\tau_F$ a cartesian natural transformation.
\end{prop}
\begin{proof}
Let $f\:C \to C'$ be a morphism in $\C /D_0$ from $\gamma\; C \to D_0$ to $\gamma' \: C' \to D_0$. The naturality square for $\tau_F$ on the morphism $f$ is the following square 
\begin{cd}
    {D_1 \times_{D_0} C} \& {E_1 \times_{E_0} C} \\
	{D_1 \times_{D_0} C'} \& {E_1 \times_{E_0} C'.}
	\arrow["{F_1 \times \id{C}}", from=1-1, to=1-2]
	\arrow["{\id{D_1} \times f}"', from=1-1, to=2-1]
	\arrow["{\id{E_1} \times f}", from=1-2, to=2-2]
	\arrow["{F_1 \times \id{C'}}"', from=2-1, to=2-2]
\end{cd}
In order to prove that such square is a pullback, let us consider morphisms $A \aar{r} E_1 \times_{E_0} C$ and $A \aar{q} D_1 \times_{D_0} C'$ such that $(\id{E_1} \times f) \c r= (F_1 \times \id{C'}) \c q$. Thanks to the commutativity of the following pasting diagram 
\begin{cd}[5][5]
  A \& {E_1 \times_{E_0} C} \&\& C \\
	\& {E_1 \times_{E_0} C'} \& {C'} \\
	{D_1 \times_{D_0} C'} \& {D_1} \&\& {D_0}
	\arrow["r", from=1-1, to=1-2]
	\arrow["q"', from=1-1, to=3-1]
	\arrow[from=1-2, to=1-4]
	\arrow["{\id{E_1} \times f}", from=1-2, to=2-2]
	\arrow["f"', from=1-4, to=2-3]
	\arrow["\gamma", from=1-4, to=3-4]
	\arrow[from=2-2, to=2-3]
	\arrow["{\gamma'}", from=2-3, to=3-4]
	\arrow["{F_1 \times \id{C'}}", from=3-1, to=2-2]
	\arrow[from=3-1, to=2-3]
	\arrow[from=3-1, to=3-2]
	\arrow["s"', from=3-2, to=3-4]
\end{cd}
there exists a unique morphism $A \aar{\phi} D_1 \times_{D_0} C$ such that the top and the left square of the following diagram commute
\begin{cd}[5][5]
    A \&[-2ex] {E_1 \times_{E_0} C} \& \\[-2ex]
	{D_1 \times_{D_0}C'} \& {D_1 \times_{D_0} C} \& C \\
	\& {D_1} \& {D_0}
	\arrow["r", from=1-1, to=1-2]
	\arrow["q"', from=1-1, to=2-1]
	\arrow["\phi", from=1-1, to=2-2]
	\arrow[from=1-2, to=2-3]
	\arrow[from=2-1, to=3-2]
	\arrow[from=2-2, to=2-3]
	\arrow[from=2-2, to=3-2]
	\arrow["\lrcorner"{anchor=center, pos=0.125}, draw=none, from=2-2, to=3-3]
	\arrow["\gamma", from=2-3, to=3-3]
	\arrow["s"', from=3-2, to=3-3]
\end{cd}
It is then straightforward to check, projecting to $E_1$ and $C'$, that the same morphism $\phi$ is such that $(F_1\times \id{C}) \c \phi =r$ and $(\id{D_1} \times f) \c \phi=q$. So the natural transformation $\tau_F$ is cartesian.
   \end{proof}

We now define the subcategory of $\Mndcoadj$ that will be proved to correspond to generalized action monads.

\begin{defne}
    We define the category $\GAMnd{\C}$ to be given by the following data:
    \begin{itemize}
  \item \emph{objects:} triples $(D_0,\TT,\alpha)$ where $D_0\in \C$ and $\TT$ is a monad on $\C/{D_0}$ such that $\dom\:\C/{D_0}\to \C$ extends to a comorphism of monads $(\dom,\alpha)\:\TT\to \Id{\C}$ and $\alpha$ is a cartesian natural transformation;
  \item \emph{morphisms $(D_0,\TT,\alpha)\to (E_0,\TT',\beta)$:} pairs $(h,\lambda)$ where $h\:D_0\to E_0$ is a morphism in $\C$ such that $h\c -\: \C/{D_0}\to \C/{E_0}$ extends to a comorphism of monads $(h\c -, \lambda)\:\TT\to \TT'$ with 
  \begin{cd}[5][5]
      {\C/{D_0}} \& {\C/{E_0}} \\
	{\C/{D_0}} \& {\C/{E_0}}
	\arrow["{h\c -}", from=1-1, to=1-2]
	\arrow[""{name=0, anchor=center, inner sep=0}, "T"', from=1-1, to=2-1]
	\arrow[""{name=1, anchor=center, inner sep=0}, "{T'}", from=1-2, to=2-2]
	\arrow["{h\c -}"', from=2-1, to=2-2]
	\arrow["\lambda"', between={0.2}{0.8}, Rightarrow, from=0, to=1]
  \end{cd}
  and moreover $\lambda$ makes the following diagram of comorphisms of monads commute:
  \begin{eqD}{compatibilitymorphismsActMnd}
  \begin{cd}*[2]
    {\TT} \\
	\& {\Id{\C}} \\
	{\TT'}
	\arrow["{(\dom,\alpha)}", from=1-1, to=2-2]
	\arrow["{(h\c -, \lambda)}"', from=1-1, to=3-1]
	\arrow["{(\dom,\beta)}"', from=3-1, to=2-2]
\end{cd}
\end{eqD}
\item \emph{composition of $(h,\lambda^h)$ and $(k,\lambda^k)$} given by $k\c h$ and the pasting $\lambda^h|\lambda^k$;
\item \emph{identity on $(D_0,\TT,\alpha)$} given by $(\id{D_0},\id{})$. 
\end{itemize}
These data are readily seen to form a category.
\end{defne}

\begin{remark}\label{remcompatibilitymorphismActMnd}
    Diagram \refs{compatibilitymorphismsActMnd} is equivalent to the following equality of natural transformations:
    \begin{eqD*}
        \begin{cd}*
           {\C/D_0} \& {\C/D_0} \\
	{\C/E_0} \& {\C/E_0} \\
	\& \C
	\arrow["T", from=1-1, to=1-2]
	\arrow["{h \circ -}"', from=1-1, to=2-1]
	\arrow["\lambda"', Rightarrow, from=1-2, to=2-1, shorten <= 3ex, shorten >= 3ex]
	\arrow["{h\circ -}", from=1-2, to=2-2]
	\arrow["{T'}"', from=2-1, to=2-2]
	\arrow[""{name=0, anchor=center, inner sep=0}, "\dom"', from=2-1, to=3-2]
	\arrow["\dom", from=2-2, to=3-2]
	\arrow["\beta", between={0}{0.8}, Rightarrow, from=2-2, to=0]
        \end{cd}
        \quad = \quad
        \begin{cd}*
            {\C/D_0} \& {\C/D_0} \\
	\& \C
	\arrow["T", from=1-1, to=1-2]
	\arrow[""{name=0, anchor=center, inner sep=0}, "\dom"', from=1-1, to=2-2]
	\arrow["\dom", from=1-2, to=2-2]
	\arrow["\alpha", between={0}{0.8}, Rightarrow, from=1-2, to=0]
        \end{cd}
    \end{eqD*}
\end{remark}

\begin{remark}
    $\GAMnd{\C}$ is a non-full subcategory of $\Mndcoadj$ (and also of $\Mndco$), in the sense that we have a faithful functor $\GAMnd{\C}\ito \Mndcoadj$. 
\end{remark}

We now prove the main theorem of this section, which characterizes generalized action monads.

\begin{teo}\label{theormainequivalence}
    The functor $\T\:\Cat(\C)\to \Mndcoadj$ factors through $\GAMnd{\C}\ito \Mndcoadj$. Moreover, $\T\:\Cat(\C)\to \GAMnd{\C}$ yields an equivalence of categories
    $$\T\:\Cat(\C)\simeq \GAMnd{\C}\:\Psi.$$
\end{teo}
\begin{proof} 
    We first prove that the functor $\T$ factors through $\GAMnd{\C}$. It suffices to prove that given an internal functor $F\:\DD\to \EE$ exhibited by
    \begin{cd}[5.5]
{D_1 \times_{D_0} D_1} \arrow[r,"{\comp}"] \arrow[d,"{F_1\times F_1}"'] \&[-1ex]
{D_1} \arrow[r,shift left=1.6ex,"s"] \arrow[r,shift right=1.8ex,"t"'] \arrow[d,"{F_1}"'] \&[0.5ex]
{D_0} \arrow[l,"{i}"{description}] \arrow[d,"{F_0}"] \\
{E_1 \times_{E_0} E_1} \arrow[r,"{\comp'}"] \&
{E_1} \arrow[r,shift left=1.6ex,"s'"] \arrow[r,shift right=1.8ex,"t'"'] \&
{E_0} \arrow[l,"{i'}"{description}]
\end{cd}
the comorphism of monads $\T_{F}$ satisfies the additional compatibility required for a morphism of generalized action monads, i.e.\ the condition described in \remx\ref{remcompatibilitymorphismActMnd}. But this is true by construction of $\T_{F}$, since given $C \aar{\gamma} D_0$ in $\C /D_0$, the morphism $(\T_{F})_{\gamma}$ is such that the following diagram commutes 
\begin{cd}
{D_1\times_{D_0}C} \& \\
	{E_1\times_{E_0}C} \& C.
	\arrow["{(\T_{F})_{\gamma}}"', from=1-1, to=2-1]
	\arrow["{\pi_2=\alpha_{\gamma}}", from=1-1, to=2-2]
	\arrow["{\pi_2=\beta_{h\circ\gamma}}"', from=2-1, to=2-2]
    \end{cd}
    
    We now construct the pseudo-inverse $\Psi$ of $\T$. So take $(D_0,T,\alpha)\in \GAMnd{\C}$, with $T=(\C/{D_0},T,\eta,\mu)$. We associate to it the following internal category in $\C$. We set its object of objects to be $D_0\in \C$. We then set the object of morphisms to be $D_1=\dom{(T(\id{D_0}))}$ and the target map $t\:D_1\to D_0$ to be $T(\id{D_0})\:\dom{(T(\id{D_0}))}\to D_0$. The source map $s\:D_1\to D_0$ is instead given by the component $\alpha_{\id{D_0}}\:\dom{(T(\id{D_0}))}\to \dom{(\id{D_0})}$ of $\alpha$ on the identity.

    We then set the identity map $i\:D_0\to D_1=\dom{(T(\id{D_0}))}$ to be $\dom(\eta_{\id{D_0}})$. Notice that indeed $\eta_{\id{D_0}}$ is a morphism in $\C/{D_0}$ as below:
    \begin{cd}[4][3]
        {D_0} \&\&[-1ex] {\dom{T(\id{D_0})}} \\
	    \& {D_0}
	\arrow["i", from=1-1, to=1-3]
	\arrow[equals, from=1-1, to=2-2]
	\arrow["{T(\id{D_0})=t}", from=1-3, to=2-2]
    \end{cd}
    The compatibility of the target map with the identity map is then guaranteed by construction. Moreover, also the source map is compatible with the identity map, because $(\dom,\alpha)$ is a comorphism of monads (thanks to \prox\ref{propcomorphismofmonads}).

    Finally, we build the composition map $\comp\:D_1\x[D_0]D_1\to D_1$ for 
    \begin{cd}*
    D_1 \arrow[r,shift left=1.7,"s"]\arrow[r,shift right=1.7,"t"']\& D_0 \arrow[l,"i"{description}]
    \end{cd}
    as follows. First, notice that the fact that $\alpha$ is a cartesian natural transformation guarantees that $D_1\x[D_0]D_1$ is given by the following pullback square:
    \sq{\dom(T^2(\id{D_0}))}{D_1}{D_1}{D_0}{\alpha_t}{\dom(T(\underline{t}))}{\dom(\underline{t})=t}{\alpha_{\id{D_0}}=s}
    Indeed the square above is the naturality square of $\alpha$ on the morphism $\underline{t}$ in $\C/{D_0}$ given by
    \begin{cd}*[2.5][-1]
    D_1 \&\& {D_0} \\
	\& {D_0}
	\arrow["t", from=1-1, to=1-3]
	\arrow["t"'{pos=0.35}, from=1-1, to=2-2]
	\arrow[equals, from=1-3, to=2-2]
\end{cd}. We then define $\comp\:D_1\x[D_0]D_1\to D_1$ to be $\dom(\mu_{\id{D_0}})$. Notice that indeed $\mu_{\id{D_0}}$ is a morphism in $\C/{D_0}$ as below:
\begin{cd}[3][4]
    {\dom(T^2({\id{D_0}}))} \&[-4ex]\& {D_1} \\[3ex]
	\& {D_0}
	\arrow["{\comp}", from=1-1, to=1-3]
	\arrow["{T(t)=T^2(\id{D_0})}"', from=1-1, to=2-2]
	\arrow["{T(\id{D_0})=t}", from=1-3, to=2-2]
\end{cd}

We prove that we have constructed an internal category
\begin{cd}[7][8]
\dom(T^2(\id{D_0})) \arrow[r,"{\mu_{\id{D_0}}}"] \&[-2ex] {\dom(T(\id{D_0}))} \arrow[r,"{\alpha_{\id{D_0}}}", shift left=1.6ex] \arrow[r,"{T(\id{D_0})}"', shift right=1.8ex] \&[1ex] {D_0} \arrow[l,"{\eta_{\id{D_0}}}"{description}]\\
\end{cd}
in $\C$. We have already shown above that source and target are compatible with the identity morphism. It remains to prove that $\comp$ is unital and associative. Unitality is guaranteed by the unitality of the monad $T$. Indeed, 
we have that 
\[
\dom(\eta_{T(\id{D_0})})
=
\langle i\circ t,\id{D_1}\rangle
:D_1\to D_1\times_{D_0}D_1,
\]
and
\[
\dom(T(\eta_{\id{D_0}}))
=
\langle \id{D_1},i\circ s\rangle
:D_1\to D_1\times_{D_0}D_1.
\]
Therefore the two unit laws
\[
\mu_{\id{D_0}}\circ \eta_{T(\id{D_0})}=\id{T(\id{D_0})},
\qquad
\mu_{\id{D_0}}\circ T(\eta_{\id{D_0}})=\id{T(\id{D_0})}
\]
give precisely
\[
\comp\circ \langle i\circ t,\id{D_1}\rangle=\id{D_1},
\qquad
\comp\circ \langle \id{D_1},i\circ s\rangle=\id{D_1}.
\]

To prove the associativity of $\comp$, we first observe that the cartesianity of $\alpha$ guarantees that $D_1\times_{D_0} D_1 \times_{D_0} D_1$ is given by any of the two isomorphic pullbacks of following diagram (see Lemma \ref{lem:assoc-pb}):
\[
\begin{tikzcd}[row sep=large, column sep=huge]
\dom(T^3(\id{D_0})) \ar[r,"\alpha_{T(t)}"] \ar[d,"\dom(T^2(\underline{t}))"'] & {\dom(T^2(\id{D_0}))} \ar[r,"\alpha_{t}"] \ar[d,"\dom(T(\underline{t}))"'] & D_1 \ar[d,"t"] \\
{\dom(T^2(\id{D_0}))} \ar[r,"\alpha_t"'] \ar[d,"T(\underline{t})"'] & D_1 \ar[r,"\alpha_{\id{D_0}}"'] \ar[d,"t"'] & D_0 \\
D_1 \ar[r,"\alpha_{\id{D_0}}"'] & D_0 &
\end{tikzcd}
\]
    Indeed the left side of the diagram above is the naturality square of $\alpha$ on the morphism $\underline{t\c T(\underline{t})}$ in $\C/{D_0}$ given by
    \begin{cd}[4]
    {\dom(T^2(\id{D_0}))} \& {D_1}\& {D_0} \\
	\& {D_0.}
	\arrow["t", from=1-2, to=1-3]
    \arrow["T(\underline{t})", from=1-1, to=1-2]
	\arrow["t\c T(\underline{t})"'{pos=0.35}, from=1-1, to=2-2]
	\arrow[equals, from=1-3, to=2-2]
\end{cd}
The associativity of $\comp$ then follows from the associativity of the multiplication of the monad $T$ thanks to the straightforward equalities $(\id{D_1} \times \comp) \c j =T(\mu_{\id{D_0}})$ and $(\comp \times \id{D_1}) \c j$, where $j$ is the unique isomorphism between $D_1\times_{D_0} D_1 \times_{D_0} D_1$ and $\dom(T^3(\id{D_0}))$. We have thus proved that we have constructed an internal category. 

Consider now a morphism $(h,\lambda)\:(D_0,T,\alpha)\to (E_0,T',\beta)$ in $\GAMnd{\C}$, where $T=(\C/{D_0},T,\eta,\mu)$ and $T'=(\C/{E_0},T',\eta',\mu')$, exhibited by $h\:D_0\to E_0$ in $\C$,
\begin{cd}[5][5]
    {\C/{D_0}} \& {\C/{E_0}} \\
	{\C/{D_0}} \& {\C/{E_0}}
	\arrow["{h\c -}", from=1-1, to=1-2]
	\arrow[""{name=0, anchor=center, inner sep=0}, "T"', from=1-1, to=2-1]
	\arrow[""{name=1, anchor=center, inner sep=0}, "{T'}", from=1-2, to=2-2]
	\arrow["{h\c -}"', from=2-1, to=2-2]
	\arrow["\lambda"', between={0.2}{0.8}, Rightarrow, from=0, to=1]
\end{cd}
and a triangle
  \begin{cd}[2]
    {T} \\
	\& {\Id{\C}} \\
	{T'}
	\arrow["{(\dom,\alpha)}", from=1-1, to=2-2]
	\arrow["{(h\c -, \lambda)}"', from=1-1, to=3-1]
	\arrow["{(\dom,\beta)}"', from=3-1, to=2-2]
\end{cd}
of comorphisms of monads. We construct as follows an internal functor in $\C$ between the internal categories associated to $(D_0,T,\alpha)$ and $(E_0,T',\beta)$. We take $h\:D_0\to E_0$ as the morphism between the objects of objects. We then construct the morphism $\widehat{h}\:\dom(T(\id{D_0}))\to \dom(T'(\id{E_0}))$ between the objects of morphisms as the composite
\begin{cd}[4][4.5]
    {\dom(T(\id{D_0}))} \& {\dom(T'(h))} \& {\dom(T'(\id{E_0}))} \\
	{D_0} \\[-3ex]
	\& {E_0}
	\arrow["{\lambda_{\id{D_0}}}", from=1-1, to=1-2]
	\arrow["{t=T(\id{D_0})}"', from=1-1, to=2-1]
	\arrow["{T'(\underline{h})}", from=1-2, to=1-3]
	\arrow["{T'(h)}"', from=1-2, to=3-2]
	\arrow["{T'(\id{E_0})=t'}", from=1-3, to=3-2]
	\arrow["h"', from=2-1, to=3-2]
\end{cd}
We prove that this yields an internal functor
\begin{cd}[6.5][7]
{\dom(T^2(\id{D_0}))} \arrow[r,"{\mu_{\id{D_0}}}"] \arrow[d,"{\widehat{h}\times \widehat{h}}"'] \&[-2ex]
{\dom(T(\id{D_0}))} \arrow[r,shift left=1.8ex,"{\alpha_{\id{D_0}}}"] \arrow[r,shift right=2ex,"{T(\id{D_0})}"'] \arrow[d,"{\widehat{h}}"'] \&[1ex]
{D_0} \arrow[l,"{\eta_{\id{D_0}}}"{description}] \arrow[d,"{h}"] \\
{\dom(T'^2(\id{E_0}))} \arrow[r,"{\mu'_{\id{E_0}}}"'] \&
{\dom(T'(\id{E_0}))} \arrow[r,shift left=1.8ex,"{\beta_{\id{E_0}}}"] \arrow[r,shift right=2ex,"{T'(\id{E_0})}"'] \&
{E_0} \arrow[l,"{\eta'_{\id{E_0}}}"{description}]
\end{cd}
The square with $h,\widehat{h}$ and the target maps commutes by construction of $\widehat{h}$. The square with $h,\widehat{h}$ and the source maps is commutative because it can be seen as the following pasting 
\begin{cd}
    {\dom(T(\id{D_0}))} \& {\dom(T'(h))} \& {\dom(T'(\id{E_0}))} \\
	{D_0} \&\& {E_0}
	\arrow["{\lambda_{\id{D_0}}}", from=1-1, to=1-2]
	\arrow["{\alpha_{\id{D_0}}}"', from=1-1, to=2-1]
	\arrow["{T'(\underline{h})}", from=1-2, to=1-3]
	\arrow["{\beta_h}"{description}, from=1-2, to=2-1]
	\arrow["{\beta_{\id{E_0}}}", from=1-3, to=2-3]
	\arrow["h"', from=2-1, to=2-3]
\end{cd}
where the triangle on the left is commutative thanks to the compatibility of $\lambda$ with $\alpha$ and $\beta$ and the square on the right is a naturality square for $\beta$. Moreover, the square with $h,\widehat{h}$ and the identity maps (given by $\eta$ and $\eta'$) is commutative because it can be seen as the following pasting 
\begin{cd}
{D_0} \&\& {E_0} \\
	{\dom(T(\id{D_0}))} \& {\dom(T'(h))} \& {\dom(T'(\id{E_0}))}
	\arrow["h", from=1-1, to=1-3]
	\arrow["{\eta_{\id{D_0}}}"', from=1-1, to=2-1]
	\arrow["{\eta'_h}"{description}, from=1-1, to=2-2]
	\arrow["{\eta'_{\id{D_0}}}", from=1-3, to=2-3]
	\arrow["{\lambda_{\id{D_0}}}"', from=2-1, to=2-2]
	\arrow["{T'(\underline{h})}"', from=2-2, to=2-3]
\end{cd}
where the triangle on the left commutes because $(h,\lambda)$ is a comorphism of monads and the square on the right is the naturality square for $\eta'$ on the morphism $\overline{h}$. It remains to prove that $(h,\widehat{h})$ preserves composition, that is the commutativity of the square on the left. Using the naturality of $\mu'$ on the morphism $\underline{h}$ and the fact that $(h,\lambda)$ is a comorphism of monads, proving this commutativity reduces to proving the commutativity of the following diagram
\begin{eqD}{diamond}
\begin{cd}*
    {\dom(T^2(\id{D_0}))} \& {\dom(T'^2(\id{E_0}))} \\
	{\dom(T'(h\circ t))= \dom(T'(t))}
	\arrow["{\widehat{h} \times \widehat{h}}", from=1-1, to=1-2]
	\arrow["{\mu_{\id{D_0}}}"', from=1-1, to=2-1]
	\arrow["{T'(\widehat{h})}"', from=2-1, to=1-2]
\end{cd}
\end{eqD}
Since $(E_0,T', \beta)$ is a generalized action monad the square   \sq{\dom(T'^2(\id{E_0}))}{E_1}{E_1}{E_0}{\beta_{t'}}{\dom(T'(\underline{t'}))}{\dom(\underline{t'})=t'}{\beta_{\id{E_0}}=s'}
is a pullback and so we can check the commutativity of (\ref{diamond}) after post-composition with $\beta_{t'}$ and $T'(\underline{t'})$. The post-composition with $\beta_{t'}$ works thanks to the fact that $(h,\lambda)$ is a morphism of generalized action monads. And the post-composition with $T'(\underline{t'})$ works thanks to the naturality of $\lambda$ and to the equality $\underline{t' \c \widehat{h}}= \underline{h \c t}$. So we have shown that $(h,\widehat{h})$ is an internal functor.

Furthermore, the construction we produced above is functorial, giving a functor $$\Psi\:\GAMnd{\C}\to \Cat(\C).$$ Indeed, it is straightforward to prove that identities and composition are preserved using how identities and compositions are defined in $\GAMnd{\C}$ as well as the naturality of the involved natural transformations.

It is then straightforward to prove that the composite $\Psi\c \T$ is the identity functor (when choosing identities as representatives for pullbacks of identities). We now look at the composite $\T \c \Psi$. Given a generalized action monad $(D_0,T,\alpha)\in \GAMnd{\C}$, with $\T=(\C/{D_0},T,\eta,\mu)$, its image along the composite $\Psi\c \T$ is the generalized action monad $(D_0,\T_{\Psi(D_0,T,\alpha)},\alpha')$ where $\T_{\Psi(D_0,T,\alpha)}$ sends $C \aar{\gamma} D_0$ to the composite $P^{\gamma} \aar{\pi_1} \dom{T(\id{D_0})}\aar{T(\id{D_0})} D_0$ where $P^\gamma$ and $\pi_1$ are given by the pullback 
\begin{cd}
    {P^{\gamma}} \& C \\
	{\dom(T(\id{D_0}))} \& {D_0}
	\arrow["{\pi_2}", from=1-1, to=1-2]
	\arrow["{\pi_1}"', from=1-1, to=2-1]
	\arrow["\lrcorner"{anchor=center, pos=0.125}, draw=none, from=1-1, to=2-2]
	\arrow["\gamma", from=1-2, to=2-2]
	\arrow["{\alpha_{\id{D_0}}}"', from=2-1, to=2-2]
    \end{cd}
    and the component of $\alpha'$ on $\gamma$ is $\pi_2$. Since $\alpha$ is cartesian, given any $C \aar{\gamma} D_0$ in $\C/D_0$ the following square is a pullback 
    \begin{cd}
        {\dom(T(\gamma))} \& C \\
	{\dom(T(\id{D_0}))} \& {D_0}
	\arrow["{\alpha_{\gamma}}", from=1-1, to=1-2]
	\arrow["{\dom(T(\underline{\gamma}))}"', from=1-1, to=2-1]
	\arrow["\lrcorner"{anchor=center, pos=0.125}, draw=none, from=1-1, to=2-2]
	\arrow["\gamma", from=1-2, to=2-2]
	\arrow["{\alpha_{\id{D_0}}}"', from=2-1, to=2-2]
    \end{cd}
    and thus there exists a unique isomorphism 
    $$\nu_{\gamma}^{(D_0, T, \alpha)}\: \dom{T(\gamma)} \aiso P^{\gamma}
    $$
    such that the following diagram is commutative 
    \begin{cd}
        {\dom(T(\gamma))} \&[-3ex] \&[-3ex] \\[-2ex]
	\& {P^{\gamma}} \& C \\[-3ex]
	\& {\dom(T^2(\id{D_0}))} \\[-3ex]
	\& {D_0}
	\arrow["{\nu_{\gamma}^{(D_0, T, \alpha)}}"',aiso, from=1-1, to=2-2]
	\arrow["{\alpha_{\gamma}}", bend left=20, from=1-1, to=2-3]
	\arrow["{\dom(T(\underline{\gamma}))}"{description},bend right=15, from=1-1, to=3-2]
	\arrow["{T(\gamma)}"',bend right=30, from=1-1, to=4-2]
	\arrow["{\pi_2}"{description}, from=2-2, to=2-3]
	\arrow["{\pi_1}", from=2-2, to=3-2]
	\arrow["{T(\id{D_0})}", from=3-2, to=4-2]
    \end{cd}
It is straightforward to check that such isomorphisms $\nu_{\gamma}^{(D_0, T, \alpha)}$ are the components of a natural isomorphism $\nu^{(D_0, T, \alpha)}\:T \Rightarrow \T_{\Psi(D_0,T,\alpha)}$ and that such natural isomorphism extends to an isomorphism of generalized action monads $(\id{D_0}, \nu)$.  
It remains to prove that the family of natural isomorphisms $\nu^{(D_0, T, \alpha)}$ is natural in $(D_0, T, \alpha)$. To prove this, it suffices to show that given a morphism of generalized action monads 
$$(h, \lambda)\: (D_0,T, \alpha) \to (E_0,T', \beta)$$
the following equality of 2-cells holds
\begin{eqD*}
    \begin{cd}*
        {\C/D_0} \& {\C /D_0} \\
	{\C/D_0} \& {\C/D_0} \\
	{\C/E_0} \& {\C/E_0}
	\arrow["T", from=1-1, to=1-2]
	\arrow[""{name=0, anchor=center, inner sep=0}, "{\id{D_0} \circ -}"', equals, from=1-1, to=2-1]
	\arrow[""{name=1, anchor=center, inner sep=0}, "{\id{D_0} \circ -}", equals, from=1-2, to=2-2]
	\arrow["{\T_{\Psi(D_0,T,\alpha)}}", from=2-1, to=2-2]
	\arrow["{h \circ -}"', from=2-1, to=3-1]
	\arrow["{\T_{\Psi(h,\lambda)}}", Rightarrow, from=2-2, to=3-1, shorten <= 3ex, shorten >= 3ex]
	\arrow["{h \circ -}", from=2-2, to=3-2]
	\arrow["{\T_{\Psi(E_0,T',\beta)}}"', from=3-1, to=3-2]
	\arrow["{\nu^{(D_0,T,\alpha)}}"{inner sep=0.2ex}, Rightarrow, from=1-2, to=2-1,shorten <= 3ex, shorten >= 3ex]
    \end{cd}
    \quad = \quad 
    \begin{cd}*
       {\C/D_0} \& {\C /D_0} \\
	{\C/E_0} \& {\C/E_0} \\
	{\C/E_0} \& {\C/E_0}
	\arrow["T", from=1-1, to=1-2]
	\arrow["{h \circ -}"', from=1-1, to=2-1]
	\arrow["\lambda", Rightarrow, from=1-2, to=2-1, shorten <= 3ex, shorten >= 3ex]
	\arrow["{\id{D_0} \circ -}", equals, from=1-2, to=2-2]
	\arrow["{T'}"', from=2-1, to=2-2]
	\arrow["{\id{E_0} \circ -}"', equals, from=2-1, to=3-1]
	\arrow["{\nu^{(E_0,T',\beta)}}"{inner sep=-0.2ex}, Rightarrow, from=2-2, to=3-1, shorten <= 3ex, shorten >= 3ex]
	\arrow["{\id{E_0} \circ -}", equals, from=2-2, to=3-2]
	\arrow["{\T_{\Psi(E_0,T',\beta)}}"', from=3-1, to=3-2]
    \end{cd}
\end{eqD*}
And it is easy to check that such equality holds on components using the naturality of $\lambda$ and the universal property of the pullbacks involved.

This concludes the proof that $\Psi$ is the pseudoinverse of $\T$.
\end{proof}

\begin{remark}
    The request that $\alpha$ is a cartesian natural transformation, in the definition of the objects of $\GAMnd{\C}$, ensures that we can build an internal category from any generalized action monad $(D_0,\TT,\alpha)$. More precisely, in the notations introduced in the proof of \thex\ref{theormainequivalence}, it is essential to ensure that $D_2=D_1\x[D_0]D_1$ is given by $\dom(T^2(\id{D_0}))$, exactly as $D_1$ is given by $\dom(T(\id{D_0}))$. It looks like a generalization of \thex\ref{theormainequivalence} could be obtained, where generalized action monads $(D_0,\TT,\alpha)$ without the requirement that $\alpha$ is cartesian correspond to a generalization of internal categories on the line of paracategories. Indeed, $\dom(T^n(\id{D_0}))$ could give the domains of the partial composition maps $D_n\rightharpoonup D_1$. As it was not needed for the applications we had in mind, we have not investigated the details. 
\end{remark}

\section{The general fibrational setting}
\label{sec:general-fibrational-setting}
We now explore a generalization of the results of the previous sections to the general fibrational setting. We consider an arbitrary bifibration $p\:\E\to \C$ satisfying the Beck--Chevalley condition at every pullback square, rather than necessarily the codomain bifibration $\cod\:\C^2\to \C$. A reference for bifibrations is \cite{JacobsCLTT}. In recent years, similar endeavours have been proved to be very fruitful. This is especially the case in the setting of descent theory, which is where we aim to apply our results. We prove that every internal category in $\C$ still yields an associated monad, with respect to the bifibration $p\:\E\to \C$, generalizing the action monad we constructed in Section \ref{secmonadicity}. However, trying to rebuild an internal category from a monad proves harder in the general fibrational setting. We conjecture that a fibrational generalization of the notion of internal category would be more suitable for this purpose.

For simplicity, we assume the bifibration $p\:\E\to \C$ to be normal (meaning that both the fibration structure and the opfibration structure are normal). 

\begin{remark}
    In the following, we identify and capture the key elements of the construction of the generalized action monad $\T_{\DD}$ associated to an internal category $\DD$ from a fibrational point of view. The slice categories $\C/X$ with $X\in \C$ can be captured as the fibres of the codomain fibration $\cod\:\C^2\to \C$ at $X$. The pullbacks along the source map $s\:D_1\to D_0$ can be captured as the change of base along $s$ with respect to the codomain fibration. Indeed, the cartesian liftings with respect to $\cod$ correspond precisely to pullbacks. The postcomposition with maps like the target map $t$ can be captured as the application of the left adjoint to the change of base functors. Indeed, the codomain fibration is a bifibration, and the direct image functors $\C/X\to \C/Y$ (left adjoint to the pullback functors $f\st\:\C/Y\to \C/X$ with $f\:X\to Y$) are precisely postcomposition functors.

    As already pointed out at the end of Section \ref{secprelim}, the useful pullbacks lemma presented in \lemx\ref{lem:assoc-pb} is a consequence of the fact that the codomain bifibration satisfies the Beck--Chevalley condition at every pullback square (recalled below).
\end{remark}

\begin{notation}
    We fix $p\:\E\to \C$ an arbitrary bifibration. We denote the fibre of $p$ at $X\in \C$ as $p^{-1}(X)$. Given $f\:X\to Y$ in $\C$, we denote the change of base functor along $f$ with respect to $p$ as $f\st\:p^{-1}(Y)\to p^{-1}(X)$. We then denote the direct image functor along $f$ with respect to $p$ as $f_!\:p^{-1}(X)\to p^{-1}(Y)$.
\end{notation}

\begin{rem}
    Recall that a bifibration $p\:\E\to\C$ corresponds via the Grothendieck construction to a pseudofunctor $\C\op\to \Catadj$ where $\Catadj$ is the 2-category of categories, right adjoint functors and natural transformations. Explicitly, $p$ corresponds to the pseudofunctor
    \begin{fun}
 	[p] & \: & \C\op & \too & \Catadj\\[1ex]
    && X & \mto & p^{-1}(X) \\ 
    && \begin{cd}*[3]
         Y \\
         X \ar[u,"f"]
     \end{cd} & \mto & \begin{cd}*[3]
         p^{-1}(Y) \ar[d,bend left=35,"{f\st}",""{name=B}]\\
         p^{-1}(X)\ar[u,bend left=35,dashed,"{f_!}",""{name=A}]
         \ar[from=A, to=B, adjunction]
        \end{cd}
\end{fun}
\end{rem}

We recall the definition of the Beck--Chevalley condition.

\begin{defi}[\cite{BenabouRoubaud}]
    Consider a pullback square
    \sq[p][6][6]{W}{X}{Y}{Z}{k}{h}{f}{g}
    in $\C$. Then taking the change of base functors we obtain a natural isomorphism 
    $$k\st \c f\st \iso (f\c k)\st = (g\c h)\st\iso h\st \c g\st$$
    that we call $\xi$. But notice that $f\st$ and $h\st$ have left adjoints $f_!$ and $h_!$ respectively. So $\xi$ has a mate
    $$\chi\:h_!\c k\st \aR{} g\st \c f_!$$
    More explicitly, $\chi$ is given by the pasting
    \begin{cd}[4]
        {p^{-1}(X)} \&[-5ex]\&\&[-5ex] \\[-3ex]
	\& {p^{-1}(Z)} \& {p^{-1}(X)} \\
	\& {p^{-1}(Y)} \& {p^{-1}(W)} \\[-3ex]
	\&\&\& {p^{-1}(Y)}
	\arrow["{f_!}"', from=1-1, to=2-2]
	\arrow[equal,bend left=15, from=1-1, to=2-3]
	\arrow["{f\st}", from=2-2, to=2-3]
	\arrow[""{name=0, anchor=center, inner sep=0}, "{g\st}"', from=2-2, to=3-2]
	\arrow[shift right=5, Rightarrow, from=2-3, to=2-2,"{\eta^f}"',xshift=-7ex,shorten <=2ex, shorten >=2ex]
	\arrow[""{name=1, anchor=center, inner sep=0}, "{k\st}", from=2-3, to=3-3]
	\arrow["{h\st}"', from=3-2, to=3-3]
	\arrow[equal,bend right=15, from=3-2, to=4-4]
	\arrow[shift left=5, Rightarrow, from=3-3, to=3-2,xshift=7ex,shorten <=2ex, shorten >=2ex,"{\epsilon^h}"]
	\arrow["{h_!}", from=3-3, to=4-4]
	\arrow[iso, from=0, to=1]
    \end{cd}
    where $\eta^f$ is the unit of the adjunction $f_!\dashv f\st$ and $\epsilon^h$ is the counit of the adjunction $h_!\dashv h\st$.
    
    $p$ is said to satisfy the \dfn{Beck--Chevalley condition at every pullback square} if for every starting pullback square in $\C$ the induced natural transformation $\chi$ is an isomorphism.
\end{defi}

\begin{remark}
    The Beck--Chevalley condition says that, over a pullback square, pulling back and then pushing forward along one side gives the same result as pushing forward and then pulling back along the opposite side. This is the base-change compatibility needed in the fibrational construction of generalized action monads.
\end{remark}

\begin{assumption}
    We assume that the fixed bifibration $p\:\E\to \C$ satisfies the Beck--Chevalley condition at every pullback square.
\end{assumption}

\begin{cons}\label{consTDfibrational}
    Let $\DD=($\begin{cd}*
    D_1 \arrow[r,shift left=1.7,"s"]\arrow[r,shift right=1.7,"t"']\& D_0 \arrow[l,"i"{description}]
\end{cd}$)$ be an internal category in $\C$. We generalize the construction of the generalized action monad $\T_{\DD}$ associated to $\DD$ (\conx\ref{consTD}), using $p\:\E\to \C$ in place of the codomain fibration. We upgrade the slices $\C/X$ with $X\in \C$ to the fibres $p^{-1}(X)$ of $p$ at $X$. We then upgrade the pullback along the source map $s\:D_1\to D_0$ to the change of base $s\st\:p^{-1}(D_0)\to p^{-1}(D_1)$ along $s$ with respect to $p$. And we upgrade the postcomposition with $t\:D_1\to D_0$ to the direct image functor $t_!\:p^{-1}(D_1)\to p^{-1}(D_0)$.

So the underlying functor $\T^p_{\DD}$ of the monad associated to $\DD$ with respect to $p$ becomes the composite of functors
$$p^{-1}(D_0)\aar{s\st}p^{-1}(D_1) \aar{t_!}p^{-1}(D_0)$$
Notice that $\T^p_{\DD}$, which sends 
$$\gamma\in p^{-1}(D_0) \mtoo t_! (s\st(\gamma))$$
really is a generalization of the generalized action monad $\T_{\DD}$ associated to $\DD$ that we built in \conx\ref{consTD}, as it is straightforward to check. In particular, the morphism $\id{}\x f$ involved in the assignment of $\T_{\DD}$ on morphisms coincides with $s\st (f)$ with respect to the codomain bifibration. The following diagram helps visualizing the action of $\T^p_{\DD}$ on $\gamma\in p^{-1}(D_0)$, in terms of cartesian and cocartesian liftings with respect to $p$:
\begin{cd}[2.3][3.7]
    {s\st(\gamma)} \&\&[-3ex] \gamma \\
	\& {t_!(s\st(\gamma))} \\
	{D_1} \&\& {D_0} \\
	\& {D_0}
	\arrow["{\overline{s}^{\gamma}}", from=1-1, to=1-3]
	\arrow["{^{s\st\gamma}\overline{t}}"', from=1-1, to=2-2]
	\arrow["p"', maps to, from=1-1, to=3-1]
	\arrow["p", maps to, from=1-3, to=3-3]
	\arrow["p"{pos=0.35}, maps to, from=2-2, to=4-2]
	\arrow["s"{pos=0.4}, from=3-1, to=3-3]
	\arrow["t"', from=3-1, to=4-2]
\end{cd}
Since we have built $\T^p_{\DD}$ as the composite of two functors, it is a functor as well.
\end{cons}

\begin{theo}
Let $\DD=($\begin{cd}*
    D_1 \arrow[r,shift left=1.7,"s"]\arrow[r,shift right=1.7,"t"']\& D_0 \arrow[l,"i"{description}]
\end{cd}$)$ be an internal category in $\C$. The functor 
\begin{fun}
 	\T^p_{\DD} & \: & p^{-1}(D_0) & \too & p^{-1}(D_0)\\[1ex]
     && \gamma & \mto & t_!(s\st(\gamma))
\end{fun}
described in \conx\ref{consTDfibrational} extends to a monad $\T^p_{\DD}=(\T^p_{\DD},\eta^{\DD},\mu^{\DD})$.
\end{theo}
\begin{proof}
    We define the unit $\eta^{\DD}\:\Id{}\aR{} \T^p_{\DD}$ to have as general component on $\gamma\in p^{-1}(D_0)$ the composite 
    $$\eta^{\DD}_{\gamma}:= {^{s\st\gamma} \overline{t}}\c [\eta_{\gamma}]$$
    where $^{s\st\gamma} \overline{t}$ is the cocartesian lifting of $t$ to $s\st(\gamma)$ and $[\eta_{\gamma}]$ is induced by the universal property of the cartesian lifting of $s$ to $\gamma$ as in the following diagram:
    \begin{cd}[2][3.5]
        \gamma \&\&\&[-3ex] \\
	\& {s\st(\gamma)} \&\& \gamma \\
	\&\& {t_!(s\st(\gamma))} \\[-3ex]
	{D_0} \\
	\& {D_1} \&\& {D_0} \\
	\&\& {D_0}
	\arrow["{[\eta_\gamma]}"'{inner sep=0.3ex,pos=0.55}, dashed, from=1-1, to=2-2, shorten >=-1ex]
	\arrow[equal,bend left=10, from=1-1, to=2-4]
	\arrow["p"', maps to, from=1-1, to=4-1]
	\arrow["{\overline{s}^{\gamma}}", from=2-2, to=2-4]
	\arrow["{^{s\st\gamma}\overline{t}}"'{inner sep=0.3ex}, from=2-2, to=3-3]
	\arrow["p"'{pos=0.4}, maps to, from=2-2, to=5-2]
	\arrow["p", maps to, from=2-4, to=5-4]
	\arrow["p"', maps to, from=3-3, to=6-3]
	\arrow["i"', from=4-1, to=5-2]
	\arrow[equal,bend left=10, from=4-1, to=5-4]
	\arrow["s"'{pos=0.4}, from=5-2, to=5-4]
	\arrow["t"', from=5-2, to=6-3]
    \end{cd}

    It is straightforward to prove, using the universal property of (co)cartesian liftings, that $\eta^{\DD}$ coincides with the following pasting of natural transformations:
    \begin{cd}
        {p^{-1}(D_0)} \&[-1ex] \&[-1ex] {p^{-1}(D_0)} \&[-1ex] \&[-1ex] {p^{-1}(D_0)} \\
	\& {p^{-1}(D_1)} \&\& {p^{-1}(D_1)}
	\arrow[equals, from=1-1, to=1-3]
	\arrow[shift right=7, iso, from=1-1, to=1-3]
	\arrow["{s\st}"', from=1-1, to=2-2]
	\arrow[equals, from=1-3, to=1-5]
	\arrow[shift right=7, iso, from=1-3, to=1-5]
	\arrow["{i_{!}}", from=1-3, to=2-4]
	\arrow["{i\st}", from=2-2, to=1-3]
	\arrow[""{name=0, anchor=center, inner sep=0}, equals, from=2-2, to=2-4]
	\arrow["{t_{!}}"', from=2-4, to=1-5]
	\arrow[between={0.2}{0.8}, Rightarrow, from=1-3, to=0,"\operatorname{counit}"]
    \end{cd}
    where $\operatorname{counit}$ is the counit of the adjunction $i_!\dashv i\st$. Whence $\eta^{\DD}$ is clearly natural.

    We then define the multiplication $\mu^{\DD}\:\T^p_{\DD}\c \T^p_{\DD}\aR{}\T^p_{\DD}$ as follows. Notice that 
    $$\mu^{\DD}\:t_!\c s\st \c t_!\c s\st\aR{} t_!\c s\st$$
    Since $p$ satisfies the Beck--Chevalley condition at the pullback square
    \sq[p][6][6]{D_1\x[D_0] D_1}{D_1}{D_1}{D_0}{\pi_1}{\pi_2}{t}{s}
    we obtain a natural isomorphism $s\st\c t_!\iso (\pi_2)_!\c\pi_1\st$. Moreover, $t_!\c (\pi_2)_!\iso (t\c \pi_2)_!$. So, in order to define the multiplication $\mu^{\DD}$, it suffices to define a natural transformation
    $$\widehat{\mu^{\DD}}\:(t\c \pi_2)_!\c\pi_1\st\c s\st\aR{} t_!\c s\st$$
    Let $\gamma\in p^{-1}(D_0)$. Consider the morphism $[\mu_{\gamma}]\:\pi_1\st (s\st (\gamma))\to s\st (\gamma)$  induced by the universal property of the cartesian lifting of $s$ to $\gamma$ as in the following diagram:
    \begin{cd}[2.3][3.7]
        {\pi_1\st(s\st(\gamma))} \&\& {s\st(\gamma)} \& \\
	\& {s\st(\gamma)} \&\& \gamma \\
	{D_1\x[D_0] D_1} \&\& {D_1} \\
	\& {D_1} \&\& {D_0}
	\arrow["{\overline{\pi_1}^{s\st \gamma}}", from=1-1, to=1-3]
	\arrow["{[\mu_\gamma]}"'{pos=0.6}, dashed, from=1-1, to=2-2,shorten >=-1ex]
	\arrow["p"', maps to, from=1-1, to=3-1]
	\arrow["{\overline{s}^{\gamma}}", from=1-3, to=2-4]
	\arrow["p"'{pos=0.3}, maps to, from=1-3, to=3-3]
	\arrow["{\overline{s}^{\gamma}}"'{pos=0.35}, from=2-2, to=2-4]
	\arrow["p"'{pos=0.3}, maps to, from=2-2, to=4-2]
	\arrow["p", maps to, from=2-4, to=4-4]
	\arrow["{\pi_1}"'{pos=0.7}, from=3-1, to=3-3]
	\arrow["{\operatorname{comp}}"', from=3-1, to=4-2]
	\arrow["s", from=3-3, to=4-4]
	\arrow["s"', from=4-2, to=4-4]
    \end{cd}
    We induce the component $\widehat{\mu^{\DD}}_{\gamma}$ of $\widehat{\mu^{\DD}}$ on $\gamma$ by the universal property of the cocartesian lifting of $t\c \pi_2$ to $\pi_1\st(s\st(\gamma))$ as in the following diagram:
    \begin{cd}[2.3][3.7]
       \& {s\st (\gamma)} \&\& {t_{!} s\st (\gamma)} \\
	{\pi_1\st (s \st (\gamma))} \&\& {(t \circ \pi_2)_{!}\pi_1\st( s\st(\gamma))} \\
	\& {D_1} \&\& {D_0} \\
	{D_1\times_{D_0}D_1} \&\& {D_0}
	\arrow["{\overline{t}}", from=1-2, to=1-4]
	\arrow["p"'{pos=0.4}, maps to, from=1-2, to=3-2]
	\arrow["p", maps to, from=1-4, to=3-4]
	\arrow["{[\mu_{\gamma}]}", from=2-1, to=1-2]
	\arrow["{\overline{t \circ \pi_2}}"'{pos=0.7}, from=2-1, to=2-3]
	\arrow["p", maps to, from=2-1, to=4-1]
	\arrow["{\exists !}"{description}, dashed, from=2-3, to=1-4]
	\arrow["p"{pos=0.4}, maps to, from=2-3, to=4-3]
	\arrow["t"{pos=0.3}, from=3-2, to=3-4]
	\arrow["{\operatorname{comp}}", from=4-1, to=3-2]
	\arrow["{t\circ \pi_2}"', from=4-1, to=4-3]
	\arrow[equals, from=4-3, to=3-4]
    \end{cd}

    It is straightforward to prove, using the universal property of (co)cartesian liftings, that $\mu^{\DD}$ coincides with the following pasting of natural transformations:
    \begin{cd}[5][5]
        {p^{-1}(D_0)} \& {p^{-1}(D_1)} \& {p^{-1}(D_0)} \& {p^{-1}(D_1)} \& {p^{-1}(D_0)} \\
	\& {p^{-1}(D_1)} \& {p^{-1}(D_1\times_{D_0}D_1)} \& {p^{-1}(D_1)}
	\arrow["{s\st}", from=1-1, to=1-2]
	\arrow[""{name=0, anchor=center, inner sep=0}, "{s\st}"', from=1-1, to=2-2]
	\arrow["{t_{!}}", from=1-2, to=1-3]
	\arrow[""{name=1, anchor=center, inner sep=0}, "{\pi_1\st}"', from=1-2, to=2-3]
	\arrow["{s\st}", from=1-3, to=1-4]
	\arrow["{t_{!}}", from=1-4, to=1-5]
	\arrow["{\operatorname{comp}\st}"', from=2-2, to=2-3]
	\arrow[""{name=2, anchor=center, inner sep=0}, shift right=5, bend right=30, equals, from=2-2, to=2-4]
	\arrow[""{name=3, anchor=center, inner sep=0}, "{{\pi_2}_{!}}"', from=2-3, to=1-4]
	\arrow["{\operatorname{comp}_{!}}"', from=2-3, to=2-4]
	\arrow[""{name=4, anchor=center, inner sep=0}, "{t_{!}}"', from=2-4, to=1-5]
	\arrow[between={0.2}{0.8}, iso, from=0, to=1]
	\arrow["{\operatorname{(BC)}}"'{inner sep=0.8ex}, between={0.2}{0.8}, iso, from=1, to=3, shift left=1ex]
	\arrow[between={0.2}{0.8}, iso, from=3, to=4]
	\arrow["{\operatorname{counit}}", between={0.3}{0.7}, Rightarrow, from=2-3, to=2]
    \end{cd}
    where the isomorphism labelled $\operatorname{(BC)}$ is given by the Beck--Chevalley condition as described above, and $\operatorname{counit}$ is the counit of the adjunction $\operatorname{comp}_!\dashv \operatorname{comp}\st$. Whence $\mu^{\DD}$ is clearly natural.

    We now prove that $(\T_{\DD}^p,\eta^{\DD}, \mu^{\DD})$ is a monad on $p^{-1}(D_0)$. In order to prove that $\mu^{\DD}\c\T_{\DD}^p\eta^{\DD}=\id{}$, we apply the Beck--Chevalley condition to the pullback square
    \sq[p][6][6]{D_1}{D_1\x[D_0]D_1}{D_0}{D_1}{\langle i\c s,\id{}\rangle}{s}{\pi_1}{i}
    This allows us to rewrite 
        \begin{cd}[2][2]
            {p^{-1}(D_1)} \& {p^{-1}(D_0)} \& \& \\
	\&\& {p^{-1}(D_1)} \arrow["{\pi_1\st}", r] \& {p^{-1}(D_1\times_{D_0}D_1)} 
	 \arrow["{i\st}", from=1-1, to=1-2]
	\arrow[""{name=0, anchor=center, inner sep=0}, bend right=30, equals, from=1-1, to=2-3]
	\arrow["{i_{!}}", from=1-2, to=2-3]
	\arrow["{\operatorname{counit}}"{pos=0.3}, between={0}{0.8}, Rightarrow, from=1-2, to=0]
        \end{cd}
    as the following pasting 
        \begin{cd}[3][3]
        {p^{-1}(D_1)} \& {p^{-1}(D_0)} \& {p^{-1}(D_1)} \& \\
	\& {p^{-1}(D_1\times_{D_0}D_1)} \& {p^{-1}(D_1)} \\
	\&\&\& {p^{-1}(D_1\times_{D_0}D_1)}
	\arrow["{i\st}", from=1-1, to=1-2]
	\arrow[""{name=0, anchor=center, inner sep=0}, "{\pi_1\st}"', from=1-1, to=2-2]
	\arrow["{i_{!}}", from=1-2, to=1-3]
	\arrow[""{name=1, anchor=center, inner sep=0}, "{s\st}", from=1-2, to=2-3]
	\arrow[""{name=2, anchor=center, inner sep=0}, "{\pi_1\st}", bend left=25, from=1-3, to=3-4]
	\arrow["{\langle i\circ s, \id{}\rangle \st}"'{inner sep=1ex}, from=2-2, to=2-3]
	\arrow[""{name=3, anchor=center, inner sep=0}, bend right=30, equals, from=2-2, to=3-4]
	\arrow["{\langle i\circ s, \id{}\rangle_{!}}"{inner sep=0 ex,pos=0.4}, from=2-3, to=3-4]
	\arrow[between={0.2}{0.8}, iso, from=0, to=1]
	\arrow["{\operatorname{(BC')}}"'{pos=0.8, inner sep=0 ex},shift left=1ex, between={0.2}{0.8}, iso, from=1, to=2]
	\arrow["{\operatorname{counit}}"{pos=0.3}, between={0}{0.8}, Rightarrow, from=2-3, to=3]
        \end{cd}

    It is then straightforward to conclude that the unit law $\mu^{\DD}\c\T_{\DD}^p\eta^{\DD}=\id{}$ holds, by using that the pasting of counits
    \begin{cd}[2.7][2.7]
        {p^{-1}(D_1)} \& {p^{-1}(D_1\times_{D_0}D_1)} \& {p^{-1}(D_1)} \& \\
	\&\&\& {p^{-1}(D_1\times_{D_0}D_1)} \\
	\&\&\& {p^{-1}(D_1)}
	\arrow["{\operatorname{comp}\st}", from=1-1, to=1-2]
	\arrow[""{name=0, anchor=center, inner sep=0}, bend right=30, equals, from=1-1, to=3-4]
	\arrow["{\langle i\circ s, \id{}\rangle \st}", from=1-2, to=1-3]
	\arrow[""{name=1, anchor=center, inner sep=0}, bend right=30, equals, from=1-2, to=2-4]
	\arrow["{\langle i\circ s, \id{}\rangle_{!}}", from=1-3, to=2-4]
	\arrow["{\operatorname{comp}_{!}}", from=2-4, to=3-4]
	\arrow["{\operatorname{counit}}"{pos=0.4}, between={0.2}{0.8}, Rightarrow, from=1, to=0]
	\arrow["{\operatorname{counit}}"{pos=0.3}, between={0}{0.8}, Rightarrow, from=1-3, to=1]
    \end{cd}
    coincides (up to isomorphisms that get appropriately cancelled by the rest of the pasting diagram for $\mu^{\DD}\c\T_{\DD}^p\eta^{\DD}$) with the counit of the composite adjunction
    $$\operatorname{comp}_!\c \langle i\c s,\id{}\rangle_! \dashv \langle i\c s,\id{}\rangle\st \c \operatorname{comp}\st$$
    which is the adjunction of the identity with itself. The other unit law is shown by dual arguments. Notice that, when building the Beck--Chevalley transformation, it is equivalent to either start from the square made by all inverse images or the square made by all direct images, to then paste with relevant units and counits.
    
    It remains to prove that the associativity condition holds for $(\T_{\DD}^p,\eta^{\DD}, \mu^{\DD})$. For this, we apply the Beck--Chevalley condition to the pullback square
    \sq[p][6][6]{(D_1\x[D_0]D_1)\x[D_0]D_1}{D_1\x[D_0]D_1}{D_1\x[D_0]D_1}{D_1}{\operatorname{comp}\x \id{}}{\pi_1}{\pi_2}{\operatorname{comp}}
    and to the dual one that involves $\id{}\x \operatorname{comp}$. It is then straightforward to conclude by similar arguments to the ones shown above, using that the composite adjunction
    $$\operatorname{comp}_!\c (\operatorname{comp}\x \id{})_!\dashv (\operatorname{comp}\x \id{})\st \c \operatorname{comp}\st$$
    has the same counit (up to appropriately cancelled isomorphisms) of the composite adjunction
    $$\operatorname{comp}_!\c (\id\x\operatorname{comp})_!\dashv (\id{}\x\operatorname{comp})\st \c \operatorname{comp}\st$$
    thanks to the associativity axiom of the internal category $\DD$.
    \end{proof}

\begin{remark}\label{remactionswrtp}
    It is straightforward to see that the monads $\T_{\DD}^p$ generalize the monads $\T_{\DD}$ that we built in \thex\ref{thm:action-monad-algebras}. More precisely, $\T_{\DD}^{\operatorname{cod}}$ coincides with $\T_{\DD}$. We can then interpret the monads $\T_{\DD}^p$ as fibrational generalized action monads. 
    
    Since the Eilenberg--Moore algebras of the generalized action monad $\T_{\DD}$ associated to an internal category $\DD$ precisely capture the actions of $\DD$, we can interpret the algebras of the monads $\T_{\DD}^p$ as a fibrational generalization of actions of internal categories.
\end{remark}

\begin{remark}\label{consactionswrtp}
    One could explicitly describe the actions of an internal category $\DD$ with respect to a bifibration $p$, corresponding with the algebras of the monad $\T_{\DD}^p$. They are given by pairs $(\gamma,\xi)$ where $\gamma\in p^{-1}(D_0)$ and $\xi\:t_!(s\st(\gamma))\to \gamma$ is a vertical morphism over $D_0$ such that the algebra axioms are satisfied.
    
    Morphisms between these generalized actions of the internal category $\DD$ can then also be explicitly described as vertical morphisms $f\:\gamma\to \gamma'$ over $D_0$ satisfying the relevant axiom of morphism of algebras.
\end{remark}
   
\begin{theo}\label{theorTisafunctorfib}
    Assume that the opfibration structure of $p$ is split. Then the assignment
    {$$\DD=(\text{\begin{cd}*
    D_1 \arrow[r,shift left=1.7,"s"]\arrow[r,shift right=1.7,"t"']\& D_0 \arrow[l,"i"{description}]
    \end{cd}})\in \Cat(\C) \quad \h[3] \mapsto \h[3] \quad \T^p_{\DD}$$}
    extends to a functor
    $$\T^p\: \Cat(\C) \to \Mndcoadj$$
\end{theo}
\begin{proof}
    For now, consider $p$ a bifibration satisfying the usual assumptions, whose opfibration structure is not necessarily split.
    
    Given an internal functor $F\:\DD\to \EE$ exhibited by
    \begin{cd}[5.5]
{D_1 \times_{D_0} D_1} \arrow[r,"{\comp}"] \arrow[d,"{F_1\times F_1}"'] \&[-1ex]
{D_1} \arrow[r,shift left=1.6ex,"s"] \arrow[r,shift right=1.8ex,"t"'] \arrow[d,"{F_1}"'] \&[0.5ex]
{D_0} \arrow[l,"{i}"{description}] \arrow[d,"{F_0}"] \\
{E_1 \times_{E_0} E_1} \arrow[r,"{\comp'}"] \&
{E_1} \arrow[r,shift left=1.6ex,"s'"] \arrow[r,shift right=1.8ex,"t'"'] \&
{E_0} \arrow[l,"{i'}"{description}]
\end{cd}
we define a comorphism of monads $\T^p_{F}\:\T^p_{\DD}\to \T^p_{\EE}$ as follows. Its underlying functor is $(F_0)_!\:p^{-1}(D_0)\to p^{-1}(E_0)$. Note that $(F_0)_!$ is a left adjoint, of $(F_0)\st$, since $p$ is a bifibration. We then define the second component of $\T^p_F$ to be the natural transformation
\begin{cd}
    p^{-1}(D_0) \& p^{-1}(D_0) \\
	p^{-1}(E_0) \& p^{-1}(E_0)
	\arrow["{\T^p_{\DD}}", from=1-1, to=1-2]
	\arrow["{(F_0)_!}"', from=1-1, to=2-1]
	\arrow["{\tau^p_F}", Rightarrow, from=1-2, to=2-1, shorten <=2.7ex,shorten >=2.2ex]
	\arrow["{(F_0)_!}", from=1-2, to=2-2]
	\arrow["{\T^p_{\EE}}"', from=2-1, to=2-2]
\end{cd}
that has component ${(\tau^p_F)}_{\gamma}$ on $\gamma\in p^{-1}(D_0)$ induced by the universal property of the cocartesian lifting of $F_0\c t$ to $s\st(\gamma)$ as in the following diagram:
\begin{cd}[3][3]
    \& {{s'}\st({F_0}_{!}(\gamma))} \&\& {t'_{!}{s'}\st({F_0}_{!}(\gamma))} \\
	{s\st \gamma} \&\& {(F_0 \circ t)_{!}s\st \gamma} \\
	\& {E_1} \&\& {E_0} \\
	{D_1} \& {D_0} \& {E_0}
	\arrow["{\overline{t'}}", from=1-2, to=1-4]
	\arrow["p"{pos=0.2}, maps to, from=1-2, to=3-2]
	\arrow["p", maps to, from=1-4, to=3-4]
	\arrow["v", from=2-1, to=1-2]
	\arrow["{\overline{F_0\circ t}}"'{pos=0.7}, from=2-1, to=2-3]
	\arrow["p"', maps to, from=2-1, to=4-1]
	\arrow["{\exists !\lambda_{\gamma}}"', dashed, from=2-3, to=1-4]
	\arrow["p"{pos=0.3}, maps to, from=2-3, to=4-3]
	\arrow["{t'}"'{pos=0.3}, from=3-2, to=3-4]
	\arrow["{F_1}", from=4-1, to=3-2]
	\arrow["t"', from=4-1, to=4-2]
	\arrow["{F_0}"', from=4-2, to=4-3]
	\arrow[equals, from=4-3, to=3-4]
\end{cd}
where $v$ is the morphism induced by the universal property of the cartesian lifting of $s'$ to $(F_0)_!$, as in the following diagram:
\begin{cd}[3][3]
    {s\st \gamma} \&\& \gamma \& \\
	\& {{s'}\st({F_0}_{!}(\gamma))} \&\& {{F_0}_{!}(\gamma)} \\
	{D_1} \&\& {D_0} \\
	\& {E_1} \&\& {E_0}
	\arrow["{\overline{s}}", from=1-1, to=1-3]
	\arrow["{\exists ! v}", dashed, from=1-1, to=2-2]
	\arrow[maps to, from=1-1, to=3-1]
	\arrow["{\overline{F_0}}", from=1-3, to=2-4]
	\arrow[maps to, from=1-3, to=3-3]
	\arrow["{\overline{s'}}"'{pos=0.7}, from=2-2, to=2-4]
	\arrow[maps to, from=2-2, to=4-2]
	\arrow[maps to, from=2-4, to=4-4]
	\arrow["s"{pos=0.6}, from=3-1, to=3-3]
	\arrow["{F_1}"', from=3-1, to=4-2]
	\arrow["{F_0}", from=3-3, to=4-4]
	\arrow["{s'}"', from=4-2, to=4-4]
\end{cd}

It is straightforward to prove, using the universal property of (co)cartesian liftings, that $\tau^p_F$ coincides with the following pasting of natural transformations:
\begin{cd}
    {p^{-1}(D_0)} \&\& {p^{-1}(D_0)} \& {p^{-1}(D_1)} \& {p^{-1}(D_0)} \& \\
	\& {p^{-1}(E_0)} \& {p^{-1}(E_1)} \&\& {p^{-1}(E_1)} \& {p^{-1}(E_1)}
	\arrow[""{name=0, anchor=center, inner sep=0}, equal, from=1-1, to=1-3]
	\arrow["{(F_0)_!}"'{inner sep=0.3ex}, from=1-1, to=2-2]
	\arrow["{s\st}", from=1-3, to=1-4]
	\arrow["{t_!}", from=1-4, to=1-5]
	\arrow["{(F_1)_!}"{inner sep=0.3ex}, from=1-4, to=2-5]
	\arrow["{(F_0)_!}"{inner sep=0.3ex}, from=1-5, to=2-6]
	\arrow["{(F_0)\st}"{inner sep=0.3ex}, from=2-2, to=1-3]
	\arrow["{(s')\st}"', from=2-2, to=2-3]
	\arrow[shift left=9,xshift=7ex, iso, from=2-2, to=2-3]
	\arrow["{(F_1)\st}"{inner sep=0.3ex}, from=2-3, to=1-4]
	\arrow[""{name=1, anchor=center, inner sep=0}, equal, from=2-3, to=2-5]
	\arrow["{(t')_!}"', from=2-5, to=2-6]
	\arrow[shift left=9,xshift=-7ex, iso, from=2-5, to=2-6]
	\arrow["{\operatorname{unit}}"', between={0.3}{0.75}, Rightarrow, from=0, to=2-2]
	\arrow["{\operatorname{counit}}", between={0.3}{0.75}, Rightarrow, from=1-4, to=1]
\end{cd}
Then $\tau^p_F$ is clearly natural. Notice that the left part of the pasting above is a Beck--Chevalley transformation, associated to a square that is not a pullback in general.

We now prove that $((F_0)_!,\tau^p_F)$ is a comorphism of monads $\T^p_{\DD}\to \T^p_{\EE}$. In order to prove the axiom involving units, we use that the composite adjunction
$$(F_1)_!\c i_! \dashv i\st\c (F_1)\st$$
has the same counit (up to appropriately cancelled isomorphisms) of the composite adjunction
$$(i')_!\c(F_0)_! \dashv (F_0)\st \c (i')\st$$
thanks to the identity axiom of the internal functor $F$. Similarly, in order to prove the axiom involving multiplications, we use that the composite adjunction
$$(F_1)_!\c \operatorname{comp}_! \dashv \operatorname{comp}\st\c (F_1)\st$$
has the same counit (up to appropriately cancelled isomorphisms) of the composite adjunction
$$(\operatorname{comp}')_!\c(F_1\x F_1)_! \dashv (F_1\x F_1)\st \c (\operatorname{comp}')\st$$
thanks to the composition axiom of the internal functor $F$. It is then straightforward to conclude that $((F_0)_!,\tau^p_F)$ is a comorphism of monads.

Now, we further assume that the opfibration structure of $p$ is split. We show that the assignment defined then yields a functor
\begin{fun}
 	\T^p & \: & \Cat(\C) & \too & \Mndcoadj\\[1ex]
    && \fib{\DD}{(F_0,F_1)}{\EE} & \mto & \fib{\T^p_{\DD}}{((F_0)_!,\,\tau^p_F)}{\T^p_{\EE}}
\end{fun}
Using the normality of $p$ and the assumption that the opfibration part of $p$ is split, it is clear that $\T^p$ preserves identities. It is then straightforward to see that $\T^p$ also preserves composition, using the interchange law and that the counit of a composite adjunction coincides with the pasting of the counits of the two adjunctions.
\end{proof}

\begin{remark}\label{remstillcomorph}
    Note that the proof of the theorem above shows that $((F_0)_!,\tau^p_F)$ is a comorphism of monads even if the opfibration structure of $p$ is not split.

    To ensure the functoriality of $\T^p$, though, we need the split condition, as we want $(G_0\c F_0)_!=(G_0)_!\c (F_0)_!$. The codomain bifibration trivially satisfies this, so the proof above really is a generalization of the proof of \thex\ref{theorTisafunctor}.
\end{remark}

\begin{remark}
    The composite of functors
    $${\Cat(\C)}\op\aarr{(\T^p)\op} {(\Mndcoadj)}\op \aar{\Alg(-)} \Cat$$
    provides a functor $\Actp[\C]{p}{-}\:{\Cat(\C)}\op\to \Cat$ that takes categories of actions of the internal category $\DD$ with respect to $p$, thanks to \remx\ref{remactionswrtp} (see also \remx\ref{consactionswrtp}).
\end{remark}

In section \ref{sec:chargenactmnd}, we saw that a key structure of generalized action monads is the associated natural transformation $\alpha\:\dom\c \T_{\DD}\aR{}\dom$, that extends $\dom\:\C/{D_0}\to \C$ to a comorphism of monads from $\T_{\DD}$ to $\Id{\C}$. We now present the fibrational generalization of such $\alpha$.

\begin{cons}
    Let $\DD$ be an internal category in $\C$ and consider the unique internal functor $Q$ from $\DD$ to the terminal category $\1$ in $\C$. Following \prox\ref{propalphaisimageofQ}, we can define $\alpha^p$, generalizing $\alpha$ of section \ref{sec:chargenactmnd} to the fibrational context, to be the second component of the comorphism of monads $\T^p_Q\:\T^p_{\DD}\to \Id{p^{-1}(1)}$ associated to $Q$ (see also \remx\ref{remstillcomorph}).

    Explicitly, $\alpha^p$ is given by the following pasting of natural transformations:
    \begin{cd}
    {p^{-1}(D_0)} \&\& {p^{-1}(D_0)} \& {p^{-1}(D_1)} \& {p^{-1}(D_0)} \& \\
	\& {p^{-1}(1)} \& {p^{-1}(1)} \&\& {p^{-1}(1)} \& {p^{-1}(1)}
	\arrow[""{name=0, anchor=center, inner sep=0}, equal, from=1-1, to=1-3]
	\arrow["{Q_!}"'{inner sep=0.3ex}, from=1-1, to=2-2]
	\arrow["{s\st}", from=1-3, to=1-4]
	\arrow["{t_!}", from=1-4, to=1-5]
	\arrow["{Q_!}"{inner sep=0.3ex}, from=1-4, to=2-5]
	\arrow["{Q_!}"{inner sep=0.3ex}, from=1-5, to=2-6]
	\arrow["{Q\st}"{inner sep=0.3ex}, from=2-2, to=1-3]
	\arrow[equal, from=2-2, to=2-3]
	\arrow[shift left=9,xshift=7ex, iso, from=2-2, to=2-3]
	\arrow["{Q\st}"{inner sep=0.3ex}, from=2-3, to=1-4]
	\arrow[""{name=1, anchor=center, inner sep=0}, equal, from=2-3, to=2-5]
	\arrow[equal, from=2-5, to=2-6]
	\arrow[shift left=9,xshift=-7ex, iso, from=2-5, to=2-6]
	\arrow["{\operatorname{unit}}"', between={0.3}{0.75}, Rightarrow, from=0, to=2-2]
	\arrow["{\operatorname{counit}}", between={0.3}{0.75}, Rightarrow, from=1-4, to=1]
\end{cd}
\end{cons}

\begin{prop}
    $\alpha^p$ coincides with the following pasting of natural transformations:
    \begin{cd}
        {p^{-1}(D_0)} \& {p^{-1}(D_1)} \& {p^{-1}(D_0)} \\
	\& {p^{-1}(D_0)} \& {p^{-1}(1)}
	\arrow["{s\st}", from=1-1, to=1-2]
	\arrow[""{name=0, anchor=center, inner sep=0}, curve={height=6pt}, equals, from=1-1, to=2-2]
	\arrow["{t_{!}}", from=1-2, to=1-3]
	\arrow[""{name=1, anchor=center, inner sep=0}, "{s_!}", from=1-2, to=2-2]
	\arrow[""{name=2, anchor=center, inner sep=0}, "{Q_!}", from=1-3, to=2-3]
	\arrow[from=2-2, to=2-3, "Q_!"']
	\arrow["{\operatorname{counit}}", between={0}{0.8}, Rightarrow, from=1-2, to=0]
	\arrow[between={0.2}{0.8}, iso, from=1, to=2]
    \end{cd}
  As a consequence, $\alpha^{\operatorname{cod}}$ coincides with the $\alpha$ constructed in \remx\ref{remintrodalpha}.
\end{prop}
\begin{proof}
    The proof is straightforward, using the universal property of the involved (co)cartesian liftings.
\end{proof}

\begin{rem}
    We have proved that every internal category in $\C$ still yields an associated monad, with respect to the bifibration $p\:\E\to \C$, generalizing the action monad we constructed in Section \ref{secmonadicity}. However, trying to rebuild an internal category from a monad proves harder in the general fibrational setting. For example, the way in which we built the target map of the internal category was to consider $T(\id{D_0})\in \C/{D_0}$. But for a general bifibration $p$, the identity of $D_0$ does not yield an object in $p^{-1}(D_0)$, and anyway $T(\id{D_0})\in p^{-1}(D_0)$ would not be of the right type to give the target map of an internal category. We believe that this opens the way to an interesting fibrational generalization of the concept of internal category. The pullback $D_1\x[D_0]D_1$ in the definition of an internal category, that produces pairs of composable morphisms, could be replaced with a cartesian lifting with respect to a fibration. And the target map could be replaced with an object in $p^{-1}(D_0)$. We conjecture that such a fibrational generalization of the notion of internal category would be more suitable to provide a generalization of \thex\ref{theormainequivalence}. We leave this to further work.
\end{rem}

 \section{Monadic descent data and realization via right modules}
\label{sec:monadic-descent-right-modules}
In this section, we apply our results to descent theory. We prove that every generalized action monad yields a well-behaved notion of generalized descent. This captures and gives new insights to both the classical theory of descent along a morphism and Galois descent (recalled in section \ref{secprelim}). At the same time, it opens the way to new applications. We will show some further examples in the next section. We compare our results with the existing literature, in particular with the theory of monadic lax descent originated from B\'{e}nabou--Roubaud~\cite{BenabouRoubaud}, and that of lax descent objects brought forward by Lucatelli Nunes~\cite{NunesDescentKan}. We show that, interestingly, generalized action monads stand in a sweet spot for descent theory. 

In the framework of monadic descent, a generalized notion of descent can be associated to any monad. Our theory will be embedded inside this general framework, but from another perspective it will actually provide a generalization of monadic descent. Restricting ourselves to generalized action monads gives an advantage: the theory then satisfies better properties and becomes closer to what happens in the classical descent theory. One important example of this is that we can use the structure of generalized action monad to construct a comparison functor, and we can then use coequalizers to also construct a realization functor, that is left adjoint to the comparison functor. 

We will see that a key ingredient in our theory will be that of a left adjoint right module over a monad, which slightly generalizes the concept of generalized action monad. Throughout this section, we will work with this slightly more general concept.

We first briefly recall monadic lax descent data and comparison
functors. Then, we define the comparison functor associated to a left adjoint right $T$-module and the relative realization functor; see
Definitions~\ref{def:left-right-T-modules}, \ref{def:tensor-over-T}, \ref{def:Phi-alpha}, and
\ref{def:comparison-functor-module}. Having a right module which is also a left adjoint will be a key ingredient for Theorem~\ref{thm:Phi-adj-K}, which establishes an adjunction between the realization functor and the comparison.

One of our main results is then to compare the generalized descent
situation associated with left adjoint right $T$-modules with ordinary \vvv{C}ech descent. In Theorem~\ref{thm:monadic-descent-rigidity}, we prove that the case of effective descent for left adjoint right $T$-modules can be reduced to the usual monadic descent. Applied to internal categories, this gives
Corollary~\ref{thm:internal-category-descent}: the comparison
\(\C\to\C^{\mathbb{D}}\) is an equivalence precisely when \(\mathbb{D}\) is the \vvv{C}ech groupoid
on \(D_0\) and \(D_0\to 1\) is an effective descent morphism. The final
equivariant case recovers Galois descent for torsors, as in
Corollary~\ref{cor:galois-descent-torsor}. Interestingly, similar rigidity criteria do not hold for descent, as opposed to effective descent.

\subsection{Monadic lax descent and comparison.}
\label{subsec:monadic-lax-descent-comparison}
 Let \(T=(T,\eta,\mu)\) be a monad on \(\mathcal E\). We shall use the term \emph{monadic lax descent datum} for a \(T\)-algebra, that is, a morphism \[ a:TX\to X \] satisfying \[ a\eta_X=\Id{X}, \qquad aT(a)=a\mu_X. \] These are the monadic forms of the unit and cocycle identities which occur in lax descent. Thus the Eilenberg--Moore category \(\Alg(T)\) may be regarded as the category of monadic lax descent data associated with \(T\). This is a formal monadic use of descent terminology: it does not assert that \(T\) comes from a geometric cover or from a \(\check{\mathrm C}\)ech groupoid. It is related to the classical monadic descent theorem of B\'{e}nabou--Roubaud, to Street's formal theory of monads, and to the 2-categorical theory of lax descent objects; see \cite{BenabouRoubaud,StreetFormalTheory,NunesDescentKan,NunesSemanticFactorization}. The category \(\Alg(T)\) always comes with its forgetful functor \[ U:\Alg(T)\to\mathcal E. \] A comparison functor from a category of global objects is given by a choice of a functor \(R:\mathcal C\to\mathcal E\) and a lift of \(R\) through \(U\). We recall this elementary notion first. 

\begin{defi} Let \(U:\mathcal A\to\mathcal E\) and \(R:\mathcal C\to\mathcal E\) be functors. A \emph{lift of \(R\) along \(U\)} is a functor \( \widetilde R:\mathcal C\to\mathcal A \) such that
\[ \begin{tikzcd} & \mathcal A \arrow[d,"U"] \\ \mathcal C \arrow[ur,dashed,"\widetilde R"] \arrow[r,swap,"R"] & \mathcal E . \end{tikzcd} \] 

In descent terminology, such a lift is the comparison functor: it sends a
global object \(C\in\mathcal C\) to the monadic lax descent datum carried by
its underlying object \(R(C)\in\mathcal E\).
\end{defi}

The next proposition is the form of the Eilenberg--Moore universal property used
below. In a \(2\)-category, the Eilenberg--Moore object of a monad represents
algebras for the induced monad on hom-categories; see
\cite[Section~3.1]{NunesSemanticFactorization} and
\cite{StreetFormalTheory}. In \(\mathbf{Cat}\), for a monad \(T\) on
\(\mathcal E\), the induced monad on \(\mathcal E^{\mathcal C}\) sends a
functor \(R:\mathcal C\to\mathcal E\) to \(TR\). Following the standard terminology of modules over monads as actions of monads
on functors~\cite{PirogWuGibbons}, we shall call an algebra for the induced
monad \(R\mapsto TR\) on \(\mathcal E^{\mathcal C}\) a left \(T\)-module.

\begin{defi}
\label{def:left-right-T-modules}
Let \(T=(T,\eta,\mu)\) be a monad on \(\mathcal E\).

A \emph{right \(T\)-module with values in \(\mathcal C\)} is a pair
\((L,\alpha)\), where \(L:\mathcal E\to\mathcal C\) is a functor and
\(\alpha:LT\Rightarrow L\) is a natural transformation satisfying
\[
\alpha L\eta=\Id{L},
\qquad
\alpha L\mu=\alpha\alpha_T.
\]
We call \((L,\alpha)\) a \emph{left adjoint right \(T\)-module} if \(L\) is a
left adjoint.

Dually, a \emph{left \(T\)-module on a
functor} \(R:\mathcal C\to\mathcal E\) is a natural transformation
\(\rho:TR\Rightarrow R\) satisfying
\[
\rho\eta_R=\Id{R},
\qquad
\rho\mu_R=\rho T\rho.
\]
We call \((R,\rho)\) a \emph{right adjoint left \(T\)-module} if \(R\) is a
right adjoint.

\end{defi}


\begin{rem}
\label{rem:modules-as-monad-morphisms}
With our conventions on morphisms and comorphisms of monads, a right
\(T\)-module structure
\(
\alpha:LT\Rightarrow L
\)
is equivalently a comorphism of monads
\[
(L,\alpha):(\mathcal E,T)\longrightarrow
(\mathcal C,\Id{\mathcal C}).
\]
Indeed, since the target monad is the identity, the monad-comorphism axioms are
exactly the unit and multiplication axioms for the right \(T\)-module
\(\alpha\).
Similarly, a left \(T\)-module structure
\(
\rho:TR\Rightarrow R
\)
is equivalently a morphism of monads
\[
(R,\rho):(\mathcal C,\Id{\mathcal C})
\longrightarrow
(\mathcal E,T).
\]

When \(L\dashv R\), by subsection \ref{subsec:monad-morphisms-comorphisms}, we can then turn a structure of right $T$-module on $L$ into a structure of left $T$-module on $R$, and vice versa. More precisely, the left $T$-module is the mate of the right $T$-module $\alpha$. This is very useful for defining the comparison functor associated to a left adjoint right $T$-module, presented below.
\end{rem}

The observations made in \remx\ref{rem:modules-as-monad-morphisms} provide the link between generalized action monads and the module-theoretic framework used throughout this section.

\begin{cor}\label{corollactmndisleftTmod}
Every generalized action monad $(D_0,T,\alpha)$ on $\C$ determines a left adjoint right $T$-module structure on the pair \((\dom,\alpha)\), where \(\dom:\C/D_0\to\C\).

As a consequence, under the equivalence of categories presented in \thex\ref{theormainequivalence}, every internal category $\DD$ in $\C$ determines a left adjoint right $T_{\DD}$-module, where $T_{\DD}$ denotes the generalized action monad associated with $\DD$.
\end{cor}

\begin{proof}
By definition of a generalized action monad, $T$ is a monad on $\C/D_0$, and $\alpha:\dom T\Rightarrow \dom$ extends $\dom$ to a comorphism of monads
$
(\dom,\alpha):(\C/D_0,T)\longrightarrow(\C,\Id{\C})
$. Equivalently, $\alpha$ satisfies
$\alpha\circ \dom\eta=\Id{\dom},
\
\alpha\circ \dom\mu=\alpha\circ \alpha T$, which are precisely the axioms for a right $T$-module structure on $\dom$ (see \remx\ref{rem:modules-as-monad-morphisms}). Moreover, the functor $\dom:\C/D_0\to\C$ is a left adjoint. Its right adjoint sends an object $A$ of $\C$ to the projection $\pi_1:D_0\times A\to D_0$. After all, section \ref{sec:chargenactmnd} shows indeed that $\GAMnd{\C}$ is a subcategory of $\Mndcoadj$.
\end{proof}

For an internal category $\DD=(D_1\rightrightarrows D_0)$, the module structure is the evident one: at an object $\gamma:X\to D_0$, the component of $\alpha_{\DD}$ is the second projection
$$
\alpha_{\DD,\gamma}:D_1\times_{s,D_0,\gamma}X\to X.
$$

\begin{prop}
\label{prop:left-modules-right-modules-lifts}
Let \(T=(T,\eta,\mu)\) be a monad on \(\mathcal E\), and let
\[
L:\mathcal E\rightleftarrows\mathcal C:R
\]
be an adjunction.
Take \(S=RL\), with its induced monad structure. Then the following data are
equivalent:
\begin{enumerate}
\item a right \(T\)-module structure on \(L\);
\item a comorphism of monads
\(
\theta:T\longrightarrow S=RL
\)
over $\mathcal E$;
\item a left \(T\)-module structure on \(R\);
\item a lift of \(R\) through the Eilenberg--Moore forgetful functor
\(
U:\Alg(T)\to\mathcal E.
\)
\end{enumerate}
\end{prop}
\begin{proof} 
First, a right \(T\)-module structure
\(
\alpha:LT\Rightarrow L
\)
is equivalently a comorphism of monads
\(
\theta:T\longrightarrow RL.
\)
over $\E$. The correspondence is given by
\[
\theta_X:
TX
\xrightarrow{\;\iota_{TX}\;}
RLTX
\xrightarrow{\;R(\alpha_X)\;}
RLX 
\]
where $\iota$ is the unit of $S$. Conversely,
\[
\alpha_X:
LTX
\xrightarrow{\;L\theta_X\;}
LRLX
\xrightarrow{\;\varepsilon_{LX}\;}
LX .
\]
Under this correspondence, the right \(T\)-module identities
\[
\alpha\circ L\eta=\Id{L},
\qquad
\alpha\circ L\mu=\alpha\circ \alpha_T
\]
are exactly the monad-comorphism axioms for
\(
\theta:T\to RL.
\)

(1) and (3) are equivalent by the standard mate correspondence for the adjunction \(L\dashv R\), together with its compatibility with horizontal and vertical composition; see, for instance, \cite[Section~4.3]{RiehlContext} and \cite{KellyStreet2Cat}. Under this correspondence, a transformation \(\alpha:LT\Rightarrow L\) has mate \(\rho:TR\Rightarrow R\), whose component at \(C\in\mathcal C\) is \[ TRC \xrightarrow{\;\iota_{TRC}\;} RLTRC \xrightarrow{\;R(\alpha_{RC})\;} RLRC \xrightarrow{\;R(\varepsilon_C)\;} RC . \] Conversely, a transformation \(\rho:TR\Rightarrow R\) has mate \(\alpha:LT\Rightarrow L\), whose component at \(X\in\mathcal E\) is \[ LTX \xrightarrow{\;LT(\iota_X)\;} LTRLX \xrightarrow{\;L(\rho_{LX})\;} LRLX \xrightarrow{\;\varepsilon_{LX}\;} LX . \] By compatibility of mates with composition, the identities \[ \alpha L\eta=\Id{L}, \qquad \alpha L\mu=\alpha\alpha_T \] correspond respectively to \[ \rho\eta_R=\Id{R}, \qquad \rho\mu_R=\rho T\rho . \] Thus right \(T\)-module structures on \(L\) correspond exactly to left \(T\)-module structures on \(R\).

The equivalence of (3) and (4) is an application of the universal property of
Eilenberg--Moore objects recalled in \cite[Section~3.1]{NunesSemanticFactorization};
see also \cite{StreetFormalTheory}. Applied in the \(2\)-category
\(\mathbf{Cat}\) to the monad \(T\) on \(\mathcal E\), it gives, for every
category \(\mathcal C\), an isomorphism
\[
\mathbf{Cat}(\mathcal C,\operatorname{Alg}(T))
\iso
\operatorname{Alg}\bigl(\mathcal E^{\mathcal C},\,R\mapsto TR\bigr),
\]
which is 2-natural. Under this isomorphism, a functor \(K:\mathcal C\to\operatorname{Alg}(T)\) is
sent to its underlying functor \(UK:\mathcal C\to\mathcal E\), equipped with the
algebra structure induced objectwise from the algebras \(K(C)\).

Thus, after fixing \(R:\mathcal C\to\mathcal E\), the lifts
\(K:\mathcal C\to\operatorname{Alg}(T)\) with \(UK=R\) are exactly the algebra
structures on \(R\) for the monad \(R\mapsto TR\) on \(\mathcal E^\mathcal C\).
Such a structure is precisely a natural transformation
\(\rho:TR\Rightarrow R\) satisfying
\[
\rho\eta_R=\Id{R},
\qquad
\rho\mu_R=\rho T\rho. \qedhere
\]
 \end{proof} 

We now describe explicitly how to construct a comparison functor from a left adjoint right $T$-module. This requires the choice of a right adjoint, but the resulting functor is canonical up to a unique natural isomorphism.

\begin{defi}
\label{def:comparison-functor-module}
Let $T$ be a monad on $\mathcal E$, let $(L,\alpha)$ be a left adjoint right $T$-module with values in $\mathcal C$, and choose an adjunction
$
L:\mathcal E\rightleftarrows\mathcal C:R.
$
Let $U_T:\Alg(T)\to\mathcal E$ be the forgetful functor. The \emph{comparison functor} associated with $(L,\alpha)$ and the adjunction $L\dashv R$ is the unique functor
$
K_\alpha:\mathcal C\to\Alg(T)
$
such that \(U_TK_\alpha=R\), and such that the induced \(T\)-algebra structure
on \(R\) is the mate \(\rho:TR\Rightarrow R\) of \(\alpha\).

Equivalently, if $\rho$ is the left $T$-module structure corresponding to $\alpha$ under $L\dashv R$, then
$
K_\alpha(C)=(R(C),\rho_C)
$
and
$
K_\alpha(u)=R(u).
$
\end{defi}

\begin{rem}
\label{rem:comparison-functor-as-em-functor}
With our convention on (co)morphisms of monads, the left $T$-module structure $\rho:TR\Rightarrow R$ is the same thing as a morphism of monads
$$
(R,\rho):(\mathcal C,\Id{\mathcal C})\longrightarrow(\mathcal E,T).
$$
The comparison functor $K_\alpha$ is precisely the functor induced on Eilenberg--Moore categories by this morphism of monads:
$$
K_\alpha=\Alg(R,\rho):\Alg(\Id{\mathcal C})\longrightarrow\Alg(T),
$$
after identifying $\Alg(\Id{\mathcal C})$ with $\mathcal C$.
\end{rem}

\begin{rem}\label{remclassicalcomparison}
It is useful to keep separate this construction from the usual comparison functor of an adjunction. In Definition~\ref{def:comparison-functor-module}, the monad $T$ is not assumed to be the monad induced by $L\dashv R$. Thus $K_\alpha$ is, in general, not the classical comparison functor of the adjunction. It is the comparison functor determined by the additional module structure $\alpha:LT\Rightarrow L$.

In the special case where $T=RL$ is the monad induced by $L\dashv R$, the counit gives the canonical right $T$-module structure
$
\varepsilon L:LRL\Rightarrow L.
$
Its mate is
$
R\varepsilon:RLR\Rightarrow R.
$
The corresponding comparison functor is therefore
$$
K_{\varepsilon L}:\mathcal C\to\Alg(RL),
\qquad
K_{\varepsilon L}(C)=(RC,R\varepsilon_C),
$$
which is the usual comparison functor of the adjunction $L\dashv R$. 

In this sense, our theory generalizes that of monadic descent, as it allows for the two monads $T$ and $RL$ to be different. The comparison between the two theories will be further expanded in subsection \ref{subsecmonadiceffdesctheor}.
\end{rem}

We shall measure the effectiveness of the descent data encoded by a left adjoint right $T$-module $(L,\alpha)$ through its comparison functor.

\begin{defi}\label{def:descent}
Let $(L,\alpha)$ be a left adjoint right $T$-module, and let
$
K_\alpha:\mathcal C\to\Alg(T)
$
be its comparison functor. We say that $(L,\alpha)$ satisfies \emph{descent} if $K_\alpha$ is fully faithful, and \emph{effective descent} if $K_\alpha$ is an equivalence of categories.
\end{defi}

\begin{rem}\label{rem}
The preceding definition is independent of the chosen right adjoint. Indeed, any two right adjoints of $L$ are uniquely isomorphic as right adjoints. If
$
L\dashv R
$
and
$
L\dashv R'
$
are two choices, the canonical isomorphism $R\simeq R'$ identifies the corresponding left $T$-module structures obtained as mates of $\alpha$. Hence the two comparison functors
$
\mathcal C\to\Alg(T)
$
are naturally isomorphic. In particular, the properties of being fully faithful or an equivalence are independent of the chosen adjunction.
\end{rem}

\begin{rem}\label{rem:triangle-UKR}
Let $U:\Alg(T)\to\mathcal E$ denote the forgetful functor. By construction, the comparison functor associated with a left adjoint right $T$-module is a lift of the right adjoint $R$ through $U$; in other words,
$
UK_\alpha=R.
$
Thus one has the commutative triangle
\[
\begin{tikzcd}
\mathcal C \arrow[rr,"K_\alpha"] \arrow[dr,swap,"R"]
& & \Alg(T) \arrow[dl,"U"] \\
& \mathcal E &
\end{tikzcd}
\]
This is the abstract analogue of the usual descent comparison functor: a global object is first viewed locally by applying $R$, and the module structure equips it with its canonical monadic descent datum.
\end{rem}

We now introduce the realization functor associated with a right $T$-module. This construction does not require the module functor to be a left adjoint.

\begin{defi}\label{def:tensor-over-T}
Let $T$ be a monad on $\mathcal E$, and let $(L,\alpha)$ be a right $T$-module with values in $\mathcal C$. For a $T$-algebra $(X,a:TX\to X)$, assume that the coequalizer
$$
\begin{tikzcd}
LTX \ar[r, shift left=.5ex, "La"]
\ar[r, shift right=.5ex, swap, "\alpha_X"]
&
LX \ar[r, "q_{(X,a)}"] &
L\otimes_T (X,a)
\end{tikzcd}
$$
exists in $\mathcal C$. We call $L\otimes_T(X,a)$ the tensor product of $L$ with $(X,a)$ over $T$.
\end{defi}

\begin{rem}\label{rem:what-is-needed-for-Phi}
The construction of $L\otimes_T(X,a)$ depends only on the monad $T$, the right $T$-module structure $\alpha:LT\Rightarrow L$, and the existence of the displayed coequalizer in $\mathcal C$. No adjunction is involved in the definition. The role of an adjunction appears later: when $L$ admits a right adjoint, the module structure also yields the comparison functor from global objects to $\Alg(T)$.
\end{rem}
Assume now that the coequalizers defining the tensor product exist for all $T$-algebras.

\begin{defi}\label{def:Phi-alpha}
Let $T$ be a monad on $\mathcal E$, and let $(L,\alpha)$ be a right $T$-module with values in $\mathcal C$. Suppose that, for every $T$-algebra $(X,a)$, the coequalizer
$$
\begin{tikzcd}
LTX \ar[r, shift left=.5ex, "La"]
\ar[r, shift right=.5ex, swap, "\alpha_X"]
&
LX \ar[r, "q_{(X,a)}"] &
L\otimes_T(X,a)
\end{tikzcd}
$$
exists in $\mathcal C$. The assignment
$$
(X,a)\longmapsto L\otimes_T(X,a)
$$
extends uniquely to a functor
$$
\Phi_\alpha:\Alg(T)\longrightarrow\mathcal C.
$$
We call $\Phi_\alpha$ the \emph{realization functor} associated with the right $T$-module $(L,\alpha)$.
\end{defi}

Indeed, if $f:(X,a)\to(Y,b)$ is a morphism of $T$-algebras, then $Lf:LX\to LY$ is compatible with the two defining pairs, and hence induces a unique morphism
$$
\Phi_\alpha(f):\Phi_\alpha(X,a)\to\Phi_\alpha(Y,b)
$$
satisfying
$$
\Phi_\alpha(f)q_{(X,a)}=q_{(Y,b)}Lf.
$$

We now assume that the right $T$-module is a left adjoint. We prove that the realization functor is then a left adjoint to the comparison functor.

\begin{theo}\label{thm:Phi-adj-K}
Let $T$ be a monad on $\mathcal E$, let $(L,\alpha)$ be a left adjoint right $T$-module with values in $\mathcal C$, and choose an adjunction
$
L:\mathcal E\rightleftarrows\mathcal C:R.
$
Assume that the coequalizers in Definition~\ref{def:Phi-alpha} exist. Then the comparison functor of Definition~\ref{def:comparison-functor-module} and the realization functor of Definition~\ref{def:Phi-alpha} form an adjunction
$$
\Phi_\alpha:\Alg(T)\rightleftarrows\mathcal C:K_\alpha,
$$
with $\Phi_\alpha\dashv K_\alpha$.
\end{theo}
\begin{proof}
Let $(X,a)$ be a $T$-algebra and let $C$ be an object of $\mathcal C$. By the universal property of the coequalizer defining $\Phi_\alpha(X,a)$, a morphism
$
\Phi_\alpha(X,a)\to C
$
is the same thing as a morphism
$
g:LX\to C
$
such that
$$
g\circ La=g\circ\alpha_X.
$$
Under the adjunction $L\dashv R$, the morphism $g$ corresponds to a unique morphism
$
\overline g:X\to RC,
$
with
$
g=\varepsilon_C\circ L\overline g.
$

Let $\rho:TR\Rightarrow R$ be the left $T$-module structure corresponding to $\alpha$. We claim that the equality
$
g\circ La=g\circ\alpha_X
$
is equivalent to saying that $\overline g$ is a morphism of $T$-algebras
$
(X,a)\longrightarrow (RC,\rho_C).
$
Indeed, the latter condition is
$$
\overline g\circ a=\rho_C\circ T\overline g.
$$
Transposing the left-hand side gives
$$
\varepsilon_C\circ L\overline g\circ La=
g\circ La.
$$
Transposing the right-hand side gives
$$
\varepsilon_C\circ L\rho_C\circ LT\overline g.
$$
Since $\rho$ is the mate of $\alpha$, we have
$
\varepsilon_C\circ L\rho_C=\varepsilon_C\circ\alpha_{RC}.
$
Using naturality of $\alpha$ with respect to $\overline g:X\to RC$, this becomes
$$
\varepsilon_C\circ\alpha_{RC}\circ LT\overline g= \varepsilon_C\circ L\overline g\circ\alpha_X
$$
Hence the algebra morphism condition for $\overline g$ transposes exactly to the coequalizing condition
$
g\circ La=g\circ\alpha_X.
$

Therefore morphisms
$
\Phi_\alpha(X,a)\to C
$
are naturally in bijection with morphisms of $T$-algebras
$$
(X,a)\longrightarrow (RC,\rho_C)=K_\alpha(C).
$$
Thus we have a natural bijection
$$
\mathcal C(\Phi_\alpha(X,a),C)
\cong
\Alg(T)\bigl((X,a),K_\alpha(C)\bigr),
$$
and this proves $\Phi_\alpha\dashv K_\alpha$.
\end{proof}

\begin{rem}
\label{rem:generalized-action-descent}
By Corollary~\ref{corollactmndisleftTmod}, every generalized action monad yields a left adjoint right $T$-module. Hence the preceding construction gives both a comparison functor $K_\alpha$ and, whenever the required coequalizers exist, a realization functor $\Phi_\alpha$, such that $\Phi_\alpha\dashv K_\alpha$. Thanks to the equivalence of categories proved in \thex\ref{theormainequivalence}, we then also obtain a corresponding formal notion of generalized descent associated to any internal category $\DD$.
\end{rem}

We finish this subsection by recording the fibrational generalization of the preceding
construction. This is the point at which the abstract module formalism meets
the bifibrational descent framework of
Section~\ref{sec:general-fibrational-setting}. In a bifibration, every map
\(u:X\to Y\) gives a base-change adjunction
\[
u_!:p^{-1}(X)\rightleftarrows p^{-1}(Y):u^\ast .
\]
Such adjunctions are the basic mechanism behind the classical
B\'{e}nabou--Roubaud comparison theorem and its fibrational refinements
\cite{BenabouRoubaud,JanelidzeTholenI,JanelidzeTholenII}. On the other hand,
the passage from a left module structure to a lift through an
Eilenberg--Moore object is a formal consequence of Street's theory of monads
\cite{StreetFormalTheory}, in the form recalled in
\cite[Section~3.1]{NunesSemanticFactorization}. The following proposition is
therefore the fibrewise instance of the preceding abstract construction: a
monad on a fibre, together with a compatible module structure along a direct
image functor, determines a comparison functor into its category of algebras,
and, when the required coequalizers exist, a left adjoint realization functor.
\begin{cor}
\label{prop:fibrewise-comparison-realization}
Let $p\:\E\to\C$ be a bifibration, and let $u\:X\to Y$ be a morphism in $\C$. The bifibration gives an adjunction
$u_!\:p^{-1}(X)\rightleftarrows p^{-1}(Y)\:u\st$. Let $T=(T,\eta,\mu)$ be a monad on $p^{-1}(X)$, and suppose that $u_!$ carries a right $T$-module structure $\alpha\:u_!T\Rightarrow u_!$. Let $\rho\:Tu\st\Rightarrow u\st$ be the mate of $\alpha$ under $u_!\dashv u\st$. Then there is a comparison functor
$K^p_{u,\alpha}\:p^{-1}(Y)\to\Alg(T)$ given by
$K^p_{u,\alpha}(A)=(u\st A,\rho_A)$ and
$K^p_{u,\alpha}(f)=u\st f$.

Assume moreover that, for every $T$-algebra $(Z,a)$, the coequalizer
$$
u_!TZ
\overset{u_!a}{\underset{\alpha_Z}{\rightrightarrows}}
u_!Z
\longrightarrow
u_!\otimes_T(Z,a)
$$
exists in $p^{-1}(Y)$. Then these coequalizers define a realization functor
$\Phi^p_{u,\alpha}\:\Alg(T)\to p^{-1}(Y)$ by
$\Phi^p_{u,\alpha}(Z,a)=u_!\otimes_T(Z,a)$, and one has an adjunction
$\Phi^p_{u,\alpha}\dashv K^p_{u,\alpha}$.
\end{cor}

\begin{proof}
This is Theorem~\ref{thm:Phi-adj-K} applied with the abstract categories
$\mathcal E$ and $\mathcal C$ replaced by the fibres $p^{-1}(X)$ and
$p^{-1}(Y)$, and with $L=u_!$ and $R=u\st$. Since
$\alpha\:u_!T\Rightarrow u_!$ is a right $T$-module structure, its mate
$\rho\:Tu\st\Rightarrow u\st$ is a left $T$-module structure. Hence each
$u\st A$ carries the $T$-algebra structure $\rho_A$, and naturality of
$\rho$ makes $u\st f$ a morphism of $T$-algebras. This gives the comparison
functor $K^p_{u,\alpha}$.

For a $T$-algebra $(Z,a)$, the realization is obtained by tensoring the right
$T$-module $u_!$ with $(Z,a)$, that is, by coequalizing the two maps
$u_!TZ\rightrightarrows u_!Z$ induced by $a\:TZ\to Z$ and
$\alpha_Z\:u_!TZ\to u_!Z$. The assumed coequalizers therefore define
$\Phi^p_{u,\alpha}$, and the adjunction
$\Phi^p_{u,\alpha}\dashv K^p_{u,\alpha}$ is exactly the adjunction of
Theorem~\ref{thm:Phi-adj-K} in this fibrewise situation.
\end{proof}

\begin{rem}\label{remappltogenactmndfibr}
    The above corollary can be applied to the situation of $u$ being the morphism $q\:X\to 1$ where $1$ is the terminal object in $\C$.
    
    By the theory developed in Section \ref{sec:general-fibrational-setting}, given a bifibration $p$ satisfying the Beck--Chevalley condition, every internal category $\DD$ in $\C$ yields a monad $\T^p_{\DD}$ on $p^{-1}(D_0)$ relative to $p$. Furthermore, the monad $\T^p_{\DD}$ is equipped with a natural transformation $\alpha^p\:q_!\c \T^p_{\DD}\to q_!$ that makes $(q_!,\alpha^p)$ into a comorphism of monads $\T^p_{\DD}\to \Id{p^{-1}(1)}$, or equivalently into a left adjoint right $\T^p_{\DD}$-module.

    The theory presented above on left adjoint right $T$-modules can then be applied to the fibrational context as well, yielding a generalized notion of descent, relative to a bifibration $p$, associated to any internal category $\DD$.
\end{rem}

\subsection{Recovering classical descent along a morphism}\label{subsec:recovclassicaldescent}

Consider for simplicity the codomain bifibration $\cod\:\C^2\to \C$. Everything in this subsection can also be done more in general, with respect to any bifibration satisfying the Beck--Chevalley condition.

We can recover the classical descent along a morphism $p\:E\to B$ by considering the \vvv{C}ech internal groupoid
$$
\Eq(p)=\bigl(E\times_B E\rightrightarrows E\bigr)
$$
associated to the morphism $p$. The object $E\times_B E$ is the kernel pair associated to $p$, and source and target of the internal groupoid are given by $\pi_2$ and $\pi_1$ respectively. The theory we developed in section \ref{sec:general-fibrational-setting} provides us with an associated generalized action monad
$$
T_p=\T_{\Eq(p)}=\pi_{1!}\pi_2^\ast
$$
Notice that this monad exactly coincides with the one of B\'{e}nabou--Roubaud recalled in Section \ref{secprelim}. The key element behind the fact that the two monads coincide is the Beck--Chevalley condition.

Therefore, the algebras for the generalized action monad $T_p$ are precisely the classical descent data:
$$
\Desc(p)=\Alg(T_p)\simeq \Act[\C]{\Eq(p)}.
$$
It is then straightforward to show that also the comparison functor associated to the left adjoint right $T_p$-module corresponding to the \vvv{C}ech internal groupoid perfectly coincides with the classical comparison functor
$$
K_p\:\slice{\C}{B}\longrightarrow \Desc(p).
$$
Similarly, assuming that the involved coequalizers exist, also the realization functor provided by our theory coincides with the classical realization functor. As a consequence, the induced notions of descent and effective descent match the classical ones.

This has been the first motivating example for the action-monad formalism developed in this paper.

\subsection{Recovering Galois descent}\label{subsec:recovGaloisdesc}

The guiding idea is that Galois descent has the same form as classical descent along a morphism, where the compatibility of the local form is captured by the symmetries of $E$ over $B$. So we can use the symmetries of $E$ over $B$ to construct an internal category, encoding the compatibilities we want, that replaces the \vvv{C}ech groupoid of classical descent along a morphism.

Let $G$ be an internal group over $B$, and suppose that $G$ acts on $p\:E\to B$ by a morphism $h\:G\x[B] E\to E$ over $B$. Then, as recalled in Subsection \ref{subsec:Galoisrecall}, this action gives rise to an action groupoid 
$$
G\ltimes E=\bigl(G\times_B E\rightrightarrows E\bigr),
$$
whose source is the projection $\pi_2:G\times_B E\to E$ and whose target is
$h$. The theory developed in this paper, relative to the codomain bifibration, then yields corresponding generalized action monad
$$
T_G=\T_{G\ltimes E}=h_!\pi_2^\ast
$$
on $\slice{\C}{E}$. For an object $\gamma:X\to E$, this can be written more concretely as
$$
T_G(X)=G\times_B X,
$$
viewed as an object over $E$ by the composite
$$
G\times_B X
\xrightarrow{\id{G}\times \gamma}
G\times_B E
\xrightarrow{h}
E.
$$
The unit of the monad is induced by the identity element of $G$, and the
multiplication is induced by multiplication in $G$. The Eilenberg--Moore algebras of such a monad are precisely the actions of the internal groupoid $G\ltimes E$, and then by construction, they are precisely the Galois descent data:
$$
\Desc_G(E)=\Act[\C]{G\ltimes E}\simeq\Alg(T_G).
$$
It is then straightforward to show that also the comparison functor associated to the left adjoint right $T_G$-module corresponding to the action groupoid $G\ltimes E$ perfectly coincides with the Galois descent comparison functor
$$
K_G\:\slice{\C}{B}\longrightarrow \Desc_G(E).
$$
Similarly, assuming that the involved coequalizers exist, also the realization functor provided by our theory coincides with the realization functor described in \cite{MarquesGaloisDescent}. As a consequence, the induced notions of descent and effective descent match those of Galois descent. 

\begin{remark}
    The notion of generalized descent associated to left adjoint right modules over a monad captures and generalizes both classical descent along a morphism and Galois descent. The advantage of this theory, as opposed to the monadic descent theory, for example, is that it enjoys better properties, closer to the ones satisfied by classical descent and Galois descent. In particular, we have a realization functor constructed using coequalizers in a similar way as for the one of classical descent along a morphism.

    Our theory can also be embedded inside the general framework of lax descent objects, brought forward by Lucatelli Nunes in \cite{NunesDescentKan}. Quite similarly to what we presented in this paper, such a theory starts with an internal precategory. Then it considers its image along a pseudofunctor into $\Cat$ and calculates the 2-dimensional (lax) limit of the obtained diagram in $\Cat$. The category of associated descent data is precisely given by this 2-dimensional limit. As explained in \cite{NunesDescentKan}, when considering the pseudofunctor into $\Cat$ that takes slices of the domain category, the obtained category of descent data coincides with the category of actions of the starting internal category. But the theory presented here enjoys better properties. In fact, we have monadicity of the category of descent data, exhibited by \thex\ref{thm:action-monad-algebras}, while the general framework of \cite{NunesDescentKan} lacks monadicity. It is proved there that a forgetful functor from the lax descent category creates absolute limits and colimits, but this is weaker than monadicity as the existence of a left adjoint is not guaranteed.
\end{remark}

\subsection{A monadic effective descent theorem}\label{subsecmonadiceffdesctheor}

In the following, we generalize the rigidity criterion recalled in subsection \ref{subsec:Galoisrecall} for Galois effective descent. More precisely, we will achieve a result of reduction of effective descent associated to a left adjoint right module over a monad to classical effective descent. Note that descent, as opposed to effective descent, cannot be similarly reduced. We first explore a rigidity criterion from a monadic point of view. Then, we will apply this to the generalized action monads associated to internal categories, thanks to \thex\ref{theormainequivalence}, to obtain a rigidity criterion for internal categories.

To achieve our goal, we relate the abstract comparison functor associated with a left adjoint right \(T\)-module to the classical comparison functor of an adjunction. The
preceding construction starts from a monad \(T\) on \(\mathcal E\), an
adjunction
\(
L:\mathcal E\rightleftarrows\mathcal C:R,
\)
and a right \(T\)-module structure
\(
\alpha:LT\Rightarrow L.
\)
By Proposition~\ref{prop:left-modules-right-modules-lifts}, when \(L\dashv R\)
such a structure is equivalently given by a comorphism of monads
\(
\theta:T\to RL
\) over $\mathcal{\E}$. Thus, in the adjunction case, the abstract comparison functor \(K_\alpha\)
can be compared directly with the usual comparison functor of the adjunction
\(L\dashv R\).

More precisely, let
\(
L:\mathcal E\rightleftarrows\mathcal C:R
\)
be an adjunction, with unit
\(
\iota:\Id{\mathcal E}\Rightarrow RL
\)
and counit
\(
\varepsilon:LR\Rightarrow\Id{\mathcal C}.
\)
Denote
\(
S=RL,
\)
and regard \(S\) as the monad on \(\mathcal E\) whose unit is \(\iota\) and
whose multiplication is
\(
R\varepsilon L:RLRL\Rightarrow RL.
\)

The adjunction gives a canonical right \(S\)-module structure on \(L\): precisely, \(
\varepsilon L:LS=LRL\Rightarrow L.
\)
The corresponding left \(S\)-module structure on \(R\) is
\(
R\varepsilon :SR=RLR
\xrightarrow{}
R.
\)
Therefore the associated comparison functor is
\[
K_{\varepsilon L}:\mathcal C\to\Alg(S),
\qquad
K_{\varepsilon L}(C)=\bigl(RC,R\varepsilon_C\bigr),
\]
which is the usual comparison functor of the adjunction \(L\dashv R\).

Now let \(T=(T,\eta,\mu)\) be any monad on \(\mathcal E\), and let
\[
\theta:T\to S=RL
\]
be a comorphism of monads over $\mathcal{E}$. By
Proposition~\ref{prop:left-modules-right-modules-lifts}, this morphism
corresponds to the right \(T\)-module structure
\(
\alpha:
LT
\xrightarrow{\;L\theta\;}
LS
\xrightarrow{\;\varepsilon L\;}
L.
\)
Then $\theta$ induces a functor
\[
\Alg(\Id{\mathcal E},\theta):\Alg(S)\to\Alg(T)
\]
of restriction of algebra structures along \(\theta\). Explicitly,
\[
\Alg(\Id{\mathcal E},\theta)(A,b)=\bigl(A,b\circ\theta_A\bigr).
\]
Let
\(
K_\alpha:\mathcal C\to\Alg(T)
\)
be the comparison functor associated with \(\alpha\) and $R$. The next proposition
says that \(K_\alpha\) is obtained from the usual comparison functor
\(K_{\varepsilon L}\) by restriction of scalars along \(\theta\).
\begin{prop}
\label{prop:comparison-restriction-canonical}
With the notation above, one has
\[
K_\alpha=\Alg(\Id{\mathcal E},\theta) K_{\varepsilon L}.
\]
\end{prop}

\begin{proof}
By Remark~\ref{rem:comparison-functor-as-em-functor} and functoriality of
\(\Alg(-)\) with respect to comorphisms of monads,
\[
K_\alpha
=
\Alg(R,\rho)
=
\Alg(\Id{\mathcal E},\theta)\,\Alg(R,R\varepsilon)
=
\Alg(\Id{\mathcal E},\theta)\,K_{\varepsilon L},
\]
where
\(
\rho
=
TR\xrightarrow{\;\theta_R\;}SR\xrightarrow{\;R\varepsilon\;}R
\)
is the left \(T\)-module corresponding to
\(
\alpha=(\varepsilon L)\circ L\theta.
\)
\end{proof}

Assume now that the coequalizers required to define the realization functor
of Theorem~\ref{thm:Phi-adj-K} exist for the right \(T\)-module
\(\alpha:LT\Rightarrow L\). Thus, for every \(T\)-algebra \((Z,a)\), we write
\(
q_{(Z,a)}:LZ\longrightarrow \Phi_\alpha(Z,a)
\)
for the coequalizer of
\(
LTZ
\overset{La}{\underset{\alpha_Z}{\rightrightarrows}}
LZ.
\)

Let
\(
F:\mathcal E\to\Alg(T)
\)
denote the free \(T\)-algebra functor, so that
\[
F X=(TX,\mu_X)
\qquad\text{and}\qquad
F f=Tf.
\]
We shall use the following consequence of the module axioms.

\begin{lem}
\label{lem:realization-free-algebra}
For every object \(X\in\mathcal E\), there is a unique isomorphism
\(
\tau_X:\Phi_\alpha F X\longrightarrow LX
\)
such that
\[
\tau_X q_{F X}=\alpha_X.
\]
Moreover, \(\tau_X\) is natural in \(X\).
\end{lem}

\begin{proof} 
Since \(F X=(TX,\mu_X)\), the object \(\Phi_\alpha  X\) is defined by the coequalizer 
\[ \begin{tikzcd} LT^2X \ar[r, shift left=.5ex, "L\mu_X"] \ar[r, shift right=.5ex, swap, "\alpha_{TX}"] & LTX \ar[r, "q_{ X}"] & \Phi_\alpha  X . \end{tikzcd} \] 
We show that \(\alpha_X:LTX\to LX\) is also a coequalizer of this pair. First, the multiplication axiom for the right \(T\)-module \(\alpha\) gives \[ \alpha_X\circ L\mu_X = \alpha_X\circ \alpha_{TX}, \] so \(\alpha_X\) coequalizes the two arrows. Moreover this coequalizer is split. Indeed, set \[ s=L\eta_X:LX\to LTX, \qquad t=LT\eta_X:LTX\to LT^2X. \] Then \[ \alpha_Xs=\Id{LX}, \qquad L\mu_Xt=\Id{LTX}, \] by the unit axiom for the module and the unit axiom for the monad. Finally, naturality of \(\alpha\) with respect to \(\eta_X:X\to TX\) gives \[ \alpha_{TX}t = s\alpha_X. \] Thus \(\alpha_X\) is a split coequalizer of \[ L\mu_X,\alpha_{TX}:LT^2X\rightrightarrows LTX. \] 
In particular, it has the same universal property as \(q_{ X}\). Hence there is a unique isomorphism \( \tau_X:\Phi_\alpha  X\longrightarrow LX \) satisfying \( \tau_Xq_{ X}=\alpha_X. \) 

Naturality in \(X\) follows from the uniqueness of \(\tau_X\), since for every \(f:X\to Y\) both composites \( \Phi_\alpha  X \rightrightarrows LY \) obtained through \(\Phi_\alpha  f\) and through \(Lf\) have the same composite with \(q_{ X}\). 
\end{proof}

Let
\(
\bar\eta:\Id{\Alg(T)}\Rightarrow K_\alpha\Phi_\alpha
\)
be the unit of the adjunction
\(
\Phi_\alpha\dashv K_\alpha,
\)
and let
\(
U:\Alg(T)\to\mathcal E
\)
denote the forgetful functor. In the canonical situation above, where
\[
S=RL
\qquad\text{and}\qquad
\alpha=(\varepsilon L)\circ L\theta,
\]
the comorphism
\(
\theta:T\to S
\)
can be recovered from the unit \(\bar\eta\) on free \(T\)-algebras.

\begin{lem}
\label{lem:theta-from-unit}
For every \(X\in\mathcal E\), the composite
\[
TX
\xrightarrow{\;U\bar\eta_{F X}\;}
R\Phi_\alpha F X
\xrightarrow{\;R\tau_X\;}
SX
\]
is \(\theta_X\).
\end{lem}

\begin{proof}
We use the ordinary transposition under the adjunction \(L\dashv R\): a map
\(
h:A\to RB
\)
corresponds to the map
\[
LA
\xrightarrow{\;Lh\;}
LRB
\xrightarrow{\;\varepsilon_B\;}
B.
\]
We shall show that the two maps \(TX\to SX
\) have the same transpose.

By the construction of the adjunction
\(
\Phi_\alpha\dashv K_\alpha
\)
in Theorem~\ref{thm:Phi-adj-K}, the underlying map of the unit at the free
\(T\)-algebra \(F X\),
\(
U\bar\eta_{F X}:TX\to R\Phi_\alpha F X,
\)
is transpose to the coequalizer map
\(
q_{F X}:LTX\to \Phi_\alpha F X.
\)
Therefore the composite
\[
TX
\xrightarrow{\;U\bar\eta_{F X}\;}
R\Phi_\alpha F X
\xrightarrow{\;R\tau_X\;}
RLX
\]
is transpose to
\[
LTX
\xrightarrow{\;q_{F X}\;}
\Phi_\alpha F X
\xrightarrow{\;\tau_X\;}
LX.
\]
By Lemma~\ref{lem:realization-free-algebra}, this last composite is
\(
\tau_Xq_{F X}=\alpha_X.
\)

On the other hand, the transpose of
\(
\theta_X:TX\to SX
\)
is
\[
LTX
\xrightarrow{\;L\theta_X\;}
LSX
\xrightarrow{\;\varepsilon_{LX}\;}
LX.
\]
Since
\(
\alpha=(\varepsilon L)\circ L\theta,
\)
this transpose is also \(\alpha_X\).

Hence the two maps
\(
TX\to RLX
\)
have the same transpose under \(L\dashv R\). Since transposition is a bijection,
they are equal. Thus
\[
R\tau_X\circ U\bar\eta_{F X}
=
\theta_X. \qedhere
\]
\end{proof}

We can now prove that effective descent for a left adjoint right $T$-module can be reduced to the usual monadic effective descent. Interestingly, this rigidity only applies to effective descent and not to descent.

\begin{theo}\label{thm:monadic-descent-rigidity}
Let
$
L:\mathcal E\rightleftarrows\mathcal C:R
$
be an adjunction, and let $RL$ be its induced monad on $\mathcal E$. Let $T$ be a monad on $\mathcal E$, equipped with a comorphism of monads
$
\theta:T\to RL
$ over $\mathcal{E}$. Let
$$
\alpha:
LT\xrightarrow{L\theta}LRL\xrightarrow{\varepsilon L}L
$$
be the induced right $T$-module structure on $L$. Assume that the coequalizers defining $\Phi_\alpha$ exist. Then the following are equivalent:
\begin{enumerate}
\item the comparison functor
$
K_\alpha:\mathcal C\to\Alg(T)
$
is an equivalence;
\item the usual comparison functor
$
K_{{\varepsilon L}}:\mathcal C\to\Alg(RL)
$
of the adjunction $L\dashv R$ is an equivalence, and
$
\theta:T\to RL
$
is an isomorphism of monads.
\end{enumerate}
\end{theo}

\begin{proof}
Suppose first that $K_\alpha$ is an equivalence. Since
$
\Phi_\alpha\dashv K_\alpha,
$
the left adjoint $\Phi_\alpha$ is also an equivalence. Hence the unit
$
\bar\eta:\Id{\Alg(T)}\Rightarrow K_\alpha\Phi_\alpha
$
is an isomorphism. Applying this to the free algebra $ X$ and using Lemma~\ref{lem:theta-from-unit}, together with the isomorphism $\tau_X$ of Lemma~\ref{lem:realization-free-algebra}, shows that each component
$
\theta_X:TX\to RLX
$
is an isomorphism. Thus $\theta$ is an isomorphism of monads.

By Proposition~\ref{prop:comparison-restriction-canonical}, one has
$
K_\alpha=\Alg(\Id{\mathcal E},\theta) K_{\varepsilon L}.
$
Since $\theta$ is an isomorphism, the restriction functor
$
\Alg(\Id{\mathcal E},\theta):\Alg(RL)\to\Alg(T)
$
is an equivalence. It follows that $K_{\varepsilon L}$ is an equivalence.

Conversely, suppose that $K_{\varepsilon L}$ is an equivalence and that $\theta$ is an isomorphism. Then
$
\Alg(\Id{\mathcal E},\theta):\Alg(RL)\to\Alg(T)
$
is an equivalence. Since
$
K_\alpha=\Alg(\Id{\mathcal E},\theta) K_{\varepsilon L},
$
the functor $K_\alpha$ is an equivalence.
\end{proof}

\begin{rem} \label{rem:descent-less-rigid-than-effective-descent} The preceding theorem is deliberately stated for effective descent rather than for descent alone. The reason is that the fully faithful case does not impose the same rigidity on the comorphism of monads $ \theta:T\to RL$. Indeed, in the situation of Theorem~\ref{thm:monadic-descent-rigidity}, the comparison functor factors as $ K_\alpha=\Alg(\Id{\mathcal E},\theta) K_{\varepsilon L}, $ where $ \Alg(\Id{\mathcal E},\theta):\Alg(RL)\to\Alg(T) $ is restriction of algebra structures along $\theta$. Thus descent for $K_\alpha$ is not only the canonical descent problem associated with the adjunction $L\dashv R$. It also depends on the way in which the generalized descent data encoded by $T$ compare with the canonical descent data encoded by the monad $RL$. In other words, one must also control whether restriction along $\theta$ creates extra morphisms between algebra structures. More precisely, $K_\alpha$ is fully faithful if and only if $K_{\varepsilon L}$ is fully faithful and $\Alg(\Id{\mathcal E},\theta)$ is full on the full subcategory of $\Alg(RL)$ spanned by the canonical objects $ K_{\varepsilon L}(C), \ C\in\mathcal C. $ In particular, a sufficient global condition is that $\theta$ be dense, in the sense that $ \Alg(\Id{\mathcal E},\theta):\Alg(RL)\to\Alg(T) $ is fully faithful. For instance, this holds if each component $ \theta_X:TX\to RLX $ is an epimorphism in $\mathcal E$. This is weaker and less rigid than the effective descent statement. The isomorphism of $\theta$ is detected by the unit of the adjunction $ \Phi_\alpha\dashv K_\alpha $ on free $T$-algebras, as in Lemma~\ref{lem:theta-from-unit}. By contrast, full faithfulness of $K_\alpha$ only forces the counit of $ \Phi_\alpha\dashv K_\alpha $ to be an isomorphism. Thus ordinary descent controls only the embedding of objects of $\mathcal C$ into the category of $T$-descent data, while effective descent additionally forces all $T$-descent data to come from objects of $\mathcal C$. This is why ordinary descent has no direct analogue of Theorem~\ref{thm:monadic-descent-rigidity} with the condition ``$\theta$ is an isomorphism'' simply replaced by a formal injectivity condition. \end{rem}

The preceding theorem is special to the monad induced by the adjunction.
We record why one cannot replace \(RL\) by an arbitrary reference monad.

\begin{rem}\label{rem:noncanonical-reference-monad}
Let $S$ be a monad on $\mathcal E$, let $\beta:LS\Rightarrow L$ be a right $S$-module structure, and let
$
\theta:T\to S
$
be a comorphism of monads over $\mathcal{E}$. Then
$
\alpha=\beta\circ L\theta
$
is a right $T$-module structure on $L$, and the corresponding comparison functors satisfy
$
K_\alpha=\Alg(\Id{\mathcal E},\theta) K_\beta.
$
However, from the fact that $K_\alpha$ is an equivalence one cannot conclude, in general, that $K_\beta$ is an equivalence or that $\theta$ is an isomorphism.

Indeed, take $\mathcal E=\mathcal C=\Set$ and $L=R=\Id{\Set}$. Let $G$ be a non-trivial group, and let
$
S(X)=G\times X
$
be the usual action monad. Let
$
\beta_X:G\times X\to X
$
be the projection. The associated comparison functor
$
K_\beta:\Set\to G\text{-}\Set
$
sends a set to the corresponding set with trivial $G$-action, and is not an equivalence if $G$ is non-trivial. Now take $T=\Id{\Set}$ and let
$
\theta:T\to S
$
be the unit of the action monad, $x\mapsto(e,x)$. Then
$
\alpha=\beta\circ\theta=\Id,
$
so
$
K_\alpha:\Set\to\Set
$
is an equivalence. But $K_\beta$ is not an equivalence, and $\theta$ is not an isomorphism.
\end{rem}

The preceding theorem has the following fibrational form. The point is that the
reference monad is not arbitrary: it is the canonical monad $u\st u_!$
coming from the change-of-base adjunction.

\begin{cor}
\label{cor:fibrewise-monadic-descent-rigidity}
Let $p\:\E\to\C$ be a bifibration, and let $u:X\to Y$ be a morphism in
$\C$. The bifibration yields an adjunction
$
u_!:p^{-1}(X)\rightleftarrows p^{-1}(Y):u\st .
$
Let $T$ be a monad on $p^{-1}(X)$, equipped with a comorphism of monads
$
\theta:T\to u\st u_!
$ over $p^{-1}(X)$. Let
$$
\alpha:
u_!T
\xrightarrow{u_!\theta}
u_!u\st u_!
\xrightarrow{\varepsilon^u u_!}
u_!
$$
be the induced right $T$-module structure on $u_!$, where
$\varepsilon^u:u_!u\st\Rightarrow\Id{p^{-1}(Y)}$ is the counit of
$u_!\dashv u\st$. Assume that the coequalizers defining
$\Phi^p_{u,\alpha}$ exist. Then the following are equivalent:
\begin{enumerate}
\item the fibrational comparison functor
$
K^p_{u,\alpha}:p^{-1}(Y)\to\Alg(T)
$
is an equivalence;
\item the usual comparison functor
$
K_{\varepsilon^u u_!}:p^{-1}(Y)\to\Alg(u\st u_!)
$
of the adjunction $u_!\dashv u\st$ is an equivalence, and
$\theta:T\to u\st u_!$ is an isomorphism of monads.
\end{enumerate}
\end{cor}

\begin{proof}
This is Theorem~\ref{thm:monadic-descent-rigidity} applied to the bifibrational context,
with
$$
\mathcal E=p^{-1}(X),
\quad
\mathcal C=p^{-1}(Y),
\quad
L=u_!,
\quad
R=u\st .
$$
The monad $RL$ is then $u\st u_!$, and the canonical right $RL$-module
structure on $L$ is
$
\varepsilon^u u_!:u_!u\st u_!\Rightarrow u_!.
$
Restricting this structure along $\theta:T\to u\st u_!$ gives precisely the
right $T$-module $\alpha$ displayed above. The stated equivalence is therefore
exactly the equivalence of Theorem~\ref{thm:monadic-descent-rigidity}.
\end{proof}

\begin{rem}
    Exactly as in \remx\ref{remappltogenactmndfibr}, the above can be applied to the generalized action monads $\T^p_{\DD}$ on $p^{-1}(D_0)$, relative to $p$, associated to any internal category $\DD$ in $\C$.
\end{rem}

\subsection{The rigidity criterion for internal categories}
\label{subsec:application-internal-categories}

We finish the section by spelling out what the preceding monadic descent criterion says
for generalized action monads associated to internal categories. The point is to make explicit the content of
Theorem~\ref{thm:monadic-descent-rigidity} in a familiar case, rather than proving other new descent theorems. The conclusion
is rather rigid: the comparison from objects of $\C$ to actions of an
internal category $\DD$ is an equivalence only when $\DD$ is the
$\check{\mathrm C}$ech groupoid of its object of objects, and the
corresponding ordinary descent data are effective. In the equivariant case
this recovers the Galois descent criterion for torsors, which reduces Galois effective descent to classical effective descent along a morphism (see subsection \ref{subsec:Galoisrecall}). But such rigidity is not forced in the case of descent, unlike effective descent.

Let
$
\DD=(D_1\rightrightarrows D_0)
$
be an internal category in $\C$, with source and target maps
$
s,t:D_1\to D_0 .
$
We denote by
$
I(D_0)=(D_0\times D_0\rightrightarrows D_0)
$
the $\check{\mathrm C}$ech internal groupoid associated
with the terminal map
$
!_{D_0}:D_0\to 1.
$
Thus the source and target maps of $I(D_0)$ are, respectively,
$\pi_2$ and $\pi_1$. The map
$
\langle t,s\rangle:D_1\to D_0\times D_0
$
therefore defines an internal functor
$
\DD\longrightarrow I(D_0)
$
which is the identity on objects.

We first recall the fibrational form of the criterion. Let $p\:\E\to\C$ be a
bifibration satisfying Beck--Chevalley for pullback squares. By
Construction~\ref{consTDfibrational}, the internal category $\DD$ determines
the monad $\T^p_{\DD}=t_!s\st$ on the fibre $p^{-1}(D_0)$. Applying the same
construction to the $\check{\mathrm C}$ech groupoid $I(D_0)$ gives
$\T^p_{I(D_0)}=(\pi_1)_!(\pi_2)\st$. If $q=!_{D_0}:D_0\to 1$, the
Beck--Chevalley isomorphism for the pullback square
$$
\begin{tikzcd}
D_0\times D_0 \ar[r,"\pi_2"] \ar[d,"\pi_1"']
&
D_0 \ar[d,"q"]
\\
D_0 \ar[r,"q"']
&
1
\end{tikzcd}
$$
identifies this monad with the monad induced by the adjunction
$q_!:p^{-1}(D_0)\rightleftarrows p^{-1}(1):q\st$. Thus
$\T^p_{I(D_0)}\cong q\st q_!$. Under this identification, the internal functor
$\DD\to I(D_0)$ induces a monad comorphism
$\theta^p_{\DD}:\T^p_{\DD}\to q\st q_!$ over $p^{-1}(D_0)$.

The corresponding right $\T^p_{\DD}$-module structure on $q_!$ is
$$
\alpha^p_{\DD}:
q_!\T^p_{\DD}
\xrightarrow{q_!\theta^p_{\DD}}
q_!q\st q_!
\xrightarrow{\varepsilon^q q_!}
q_!,
$$
where $\varepsilon^q:q_!q\st\Rightarrow \Id{p^{-1}(1)}$ is the counit of
$q_!\dashv q\st$. In the notation of
Corollary~\ref{prop:fibrewise-comparison-realization}, the associated comparison
functor is
$K^p_{q,\alpha^p_{\DD}}:p^{-1}(1)\to\Alg(\T^p_{\DD})$. Applying
Theorem~\ref{thm:monadic-descent-rigidity} with
$\mathcal E=p^{-1}(D_0)$, $\mathcal C=p^{-1}(1)$, $L=q_!$, $R=q\st$, and
$T=\T^p_{\DD}$, we obtain the following consequence: assuming that the
coequalizers defining the corresponding realization functor exist,
$K^p_{q,\alpha^p_{\DD}}$ is an equivalence if and only if
$q:D_0\to 1$ is an effective descent morphism with respect to $p$ and
$\theta^p_{\DD}:\T^p_{\DD}\to q\st q_!$ is an isomorphism of monads.

We now specialize to the codomain bifibration. Assume that $\C$ has finite
limits and take $p=\cod:\C^2\to\C$. Then $p^{-1}(D_0)=\slice{\C}{D_0}$, and
the fibrational action monad $\T^p_{\DD}=t_!s\st$ is precisely the generalized action monad $\T_{\DD}:\slice{\C}{D_0}\to\slice{\C}{D_0}$ of
Construction~\ref{consTD}. Explicitly,
$$
\T_{\DD}(X\xrightarrow{\gamma}D_0)
=
\left(
D_1\times_{s,D_0,\gamma}X
\xrightarrow{\pi_1}
D_1
\xrightarrow{t}
D_0
\right).
$$
By Theorem~\ref{thm:action-monad-algebras}, its Eilenberg--Moore category is
the category of internal actions. That is, $\Alg(\T_{\DD})\cong \Act[\C]{\DD}$.

Under the isomorphism $\slice{\C}{1}\iso\C$, the adjunction
$q_!:\slice{\C}{D_0}\rightleftarrows \slice{\C}{1}:q\st$ identifies with
$\dom:\slice{\C}{D_0}\rightleftarrows \C:(D_0\times -)$. The induced monad is
the usual $\check{\mathrm C}$ech descent monad of $D_0\to 1$, namely the
action monad of $I(D_0)$. Hence the map $\DD\to I(D_0)$ gives a monad
comorphism $\theta_{\DD}:\T_{\DD}\to \T_{I(D_0)}$ over $\C/D_0$. At an object
$\gamma:X\to D_0$, its component is
$$
D_1\times_{s,D_0,\gamma}X
\longrightarrow
(D_0\times D_0)\times_{\pi_2,D_0,\gamma}X,
\qquad
(f,x)\longmapsto ((t(f),s(f)),x).
$$

By Corollary~\ref{corollactmndisleftTmod}, the domain functor
$\dom\:\slice{\C}{D_0}\to\C$ carries the corresponding canonical right
$\T_{\DD}$-module structure; denote it by $\alpha_{\DD}$. The associated
comparison functor is therefore
$K_{\alpha_{\DD}}:\C\to\Alg(\T_{\DD})$. Under the equivalence
$\Alg(\T_{\DD})\cong\Act[\C]{\DD}$, this functor sends an object $A\in\C$ to
the canonical $\DD$-action on $D_0\times A\to D_0$.

\begin{cor}
\label{thm:internal-category-descent}
Assume that \(\C\) has finite limits and coequalizers.
Let \(\DD=(D_1\rightrightarrows D_0)\) be an internal category in \(\C\).
Then the following are equivalent:
\begin{enumerate}
\item the comparison functor
$K_{\alpha_{\DD}}:\C\to\Alg(\T_{\DD})$ is an equivalence;
\item the map $D_0\to 1$ is an effective descent morphism and
$\langle t,s\rangle:D_1\to D_0\times D_0$ is an isomorphism.
\end{enumerate}
\end{cor}

\begin{proof}
The preceding discussion in the bifibrational context, applied to the codomain bifibration,
shows that $K_{\alpha_{\DD}}$ is an equivalence if and only if the ordinary descent
comparison for $D_0\to 1$ is an equivalence and
$
\theta_{\DD}:\T_{\DD}\to\T_{I(D_0)}
$
is an isomorphism of monads. The first condition is exactly that
$D_0\to 1$ is an effective descent morphism.

It remains to identify the second condition. If
$
\langle t,s\rangle:D_1\to D_0\times D_0
$
is an isomorphism, then every component of $\theta_{\DD}$ is obtained from
it by pullback, and is therefore an isomorphism. Conversely, evaluating
$\theta_{\DD}$ at the object
$
\id{D_0}:D_0\to D_0
$
of $\slice{\C}{D_0}$ gives precisely the map
$
\langle t,s\rangle:D_1\to D_0\times D_0.
$
Thus $\theta_{\DD}$ is an isomorphism if and only if
$\langle t,s\rangle$ is an isomorphism. The result follows.
\end{proof}

\begin{rem}
\label{rem:internal-category-rigidity}
Corollary~\ref{thm:internal-category-descent} says that the comparison
functor $K_{\alpha_{\DD}}:\C\to\Alg(\T_{\DD})$ is an equivalence only in the
$\check{\mathrm C}$ech case. More precisely, $\DD$ must be isomorphic over
$D_0$ to $I(D_0)=(D_0\times D_0\rightrightarrows D_0)$, and the ordinary
descent data associated with $D_0\to 1$ must be effective. Under the
identification $\Alg(\T_{\DD})\cong\Act[\C]{\DD}$, this means that every
$\DD$-action is, up to isomorphism, a canonical action on an object of the
form $D_0\times A\to D_0$. If $\DD$ is not the $\check{\mathrm C}$ech
groupoid, then a general $\DD$-action need not come from such a global object
$A$.
\end{rem}

We now specialize the preceding discussion to action groupoids. This is the
case considered in the Galois descent theorem of
\cite{MarquesGaloisDescent}. The result below is therefore not intended as a
new proof of that theorem, but as its reinterpretation through the monadic
criterion developed here. More precisely, the torsor map identifies the
action groupoid with the corresponding $\check{\mathrm C}$ech groupoid, while
the effective descent condition on the structure map accounts for the usual
descent hypothesis.

Let $G$ be a group object in
$\C$, and let
$
h:G\times E\to E
$
be a left action. The corresponding action groupoid is
$
G\ltimes E=(G\times E\rightrightarrows E),
$
with source $\pi_2$ and target $h$. Its action monad on
$\slice{\C}{E}$ is given by
$$
\T_{G\ltimes E}(X\xrightarrow{\gamma}E)
=
\left(
G\times X
\xrightarrow{h(\id{G},\gamma)}
E
\right),
$$
where we use the canonical identification
$(G\times E)\times_E X\cong G\times X$. This is the usual Galois monad
associated with the $G$-object $(E,h)$.

The canonical internal functor $\DD\to I(E)$ is the identity on objects and
is given on arrows by the torsor map
$\langle h,\pi_2\rangle:G\times E\to E\times E.$. Hence the induced monad
comorphism $\theta_{\DD}:\T_{\DD}\to\T_{I(E)}$ is precisely the monadic form of
the torsor map.

\begin{cor}
\label{cor:galois-descent-torsor}
Let $G$ be a group object in $\C$, let $(E,h)$ be a $G$-object, and put
$\DD=G\ltimes E$. Then the comparison functor
$
K_{\alpha_{\DD}}:\C\to\Alg(\T_{\DD})
$
is an equivalence if and only if $(E,h)$ is a $G$-torsor, that is, if and
only if $!_E:E\to 1 $ is an effective descent morphism and
$\langle h,\pi_2\rangle\:G\times E\to E\times E$ is an isomorphism.
\end{cor}

\begin{proof}
Apply Corollary~\ref{thm:internal-category-descent} to
$\DD=G\ltimes E$. In this case
$D_0=E$, $D_1=G\times E$, $s=\pi_2$, and $t=h$. Hence the map
$\langle t,s\rangle\to D_0\times D_0$ is exactly the torsor map
$\langle h,\pi_2\rangle\:G\times E\to E\times E$. The result follows.
\end{proof}

\begin{rem}
The preceding recovery of Galois descent can be read from the commutative triangle
given by the canonical internal functors $G\ltimes E\to I(E)\to 1$.
Since the terminal internal category $1=(1\rightrightarrows 1)$ has identity
action monad, we have $\Alg(\T_1)\simeq\C$, and functoriality gives
$$
\begin{tikzcd}
\Alg(\T_1)
\ar[r]
\ar[rr,bend right=15,swap]
&
\Alg(\T_{I(E)})
\ar[r]
&
\Alg(\T_{G\ltimes E}).
\end{tikzcd}
$$
The first arrow is ordinary $\check{\mathrm C}$ech descent for $E\to 1$. The
second is restriction along $\theta_{\DD}:\T_{\DD}\to\T_{I(E)}$, equivalently
the comparison induced by the internal functor $G\ltimes E\to I(E)$. The
torsor condition says that this second functor is an isomorphism of internal
categories, while effective descent of $E\to 1$ says that the first arrow is
an equivalence. Thus the usual Galois descent criterion is exactly the
specialization of Theorem~\ref{thm:monadic-descent-rigidity} to this
triangle.
\end{rem}

\section{Further examples}\label{sec:ex}

The following examples further illustrate descent and effective descent for generalized action monads associated with internal categories $\DD$ in a category $\C$, relative to a bifibration $p\:\E\to \C$.

In subsections \ref{subsec:recovclassicaldescent} and \ref{subsec:recovGaloisdesc}, we showed that our theory captures and gives new insights to both the classical descent along a morphism and Galois descent. We now show some further examples.

In all subsections below, we will follow the same pattern. We start with a
bifibration $p\:\E\to\C$.  An internal category
$\DD=(D_1\rightrightarrows D_0)$ in the base then induces the generalized action monad
$\T^p_{\DD}=t_!s\st$ on the fibre over $D_0$, relative to $p$.  In each case, the operation
$s\st$ pulls an object of the fibre to the arrows of $\DD$, while $t_!$
collects it again over the targets.  A $\T^p_{\DD}$-algebra is therefore an
object over $D_0$ equipped with an action of the arrows of $\DD$. The induced generalized action monads then yield the structure of a left adjoint right $\T^p_{\DD}$-module, and thus associated comparison and realization functors. Whence we can then explore descent and effective descent.

\subsection{Internal categories in $\Set$}
    Consider $\C=\Set$ and the codomain bifibration over $\Set$. As recalled in \exax\ref{exaactionsofinternalcatinSet}, actions of an internal category $\DD$ in $\Set$ (which is just a category) coincide with copresheaves on $\DD$. \thex\ref{thm:action-monad-algebras} then provides us with a monad $\T_{\DD}\:\Set/{D_0}\to \Set/{D_0}$ whose Eilenberg--Moore algebras are copresheaves on $\DD$. Explicitly, $\T_{\DD}$ sends a function $\gamma\:C\to D_0$ to the function
    \begin{fun}
 	\T_{\DD}(\gamma) & \: & D_1\x[D_0]C \hphantom{CC} & \too & {D_0}\\[1ex]
    && (X\in C,\gamma(X)\aar{f}_{\DD} Z) & \mto & Z
    \end{fun}
    An algebra for $\T_{\DD}$ consists of a pair $(\gamma\:C\to D_0, \xi\:\T_{\DD}(\gamma)\to \gamma)$ satisfying the relevant axioms. The corresponding copresheaf on $\DD$ is
    \begin{fun}
 	(\gamma,\xi) & \: & \DD & \too & \Set\\[1ex]
    && \fib{Y}{f}{Z} & \mto & \fib{\gamma^{-1}(Y)}{\xi(-,f)}{\gamma^{-1}(Z)}
    \end{fun}
    Interestingly, the axioms of algebra precisely match the functoriality conditions for $(\gamma,\xi)$.

    Following \defx\ref{def:comparison-functor-module} and \remx\ref{rem:comparison-functor-as-em-functor}, it is straightforward to show that the comparison functor associated to the generalized action monad $\T_{\DD}$ is
    \begin{fun}
 	K & \: & \Set & \too & \Alg(\T_{\DD})\\[1ex]
    && \fib{C}{h}{C'} & \mto & \fib{({D_0\x C}\aar{\pi_1}{D_0}, \, D_1\x[D_0](D_0\x C)\aar{t\x \pi_2}D_0\x C)}{D_0\x h}{({D_0\x C'}\aar{\pi_1}{D_0}, \, D_1\x[D_0](D_0\x C')\aar{t\x \pi_2}D_0\x C')}
    \end{fun}
    It is straightforward to show that $K(C)=({D_0\x C}\aar{\pi_1}{D_0}, \, D_1\x[D_0](D_0\x C)\aar{t\x \pi_2}D_0\x C)$ corresponds to a copresheaf on $\DD$ that is constant at $C$. So that $K$ is actually equivalent to the constant diagram functor $\Set\to [\DD,\Set]$.
    
    Using \defx\ref{def:Phi-alpha}, the realization functor $\Phi\:\Alg(\T_{\DD})\to \Set$, which is the left adjoint of $K$, is then built as follows. Given an algebra $(\gamma\:C\to D_0, \xi\:\T_{\DD}(\gamma)\to \gamma)$, which is precisely a copresheaf $\DD\to \Set$, $\Phi(\gamma,\xi)$ is the coequalizer in $\Set$
    \begin{cd}
    \dom(\T_{\DD}(\gamma)) \ar[r, shift left=.5ex, "\dom(\xi)"]
    \ar[r, shift right=.5ex, swap, "\alpha_\gamma"]
    \&
    \dom(\gamma) \ar[r, "q_{(\gamma,\xi)}"] \&
    \dom\otimes_{\T_{\DD}}(\gamma,\xi)=\Phi(\gamma,\xi)
    \end{cd}
    That is, the coequalizer in $\Set$
    \begin{cd}
    D_1\x[D_0]C \ar[r, shift left=.5ex, "\dom(\xi)"]
    \ar[r, shift right=.5ex, swap, "{\pi_2}"]
    \&
    C \ar[r, "q_{(\gamma,\xi)}"] \&
    \Phi(\gamma,\xi)
    \end{cd}
    By the explicit construction of coequalizers in $\Set$, we get that $\Phi(\gamma,\xi)$ is the quotient of $C$ by the smallest equivalence relation generated by 
    $$\xi(X,f)\simeq X \text{ for every } (X\in C,\gamma(X)\aar{f}_{\DD} Z)\in D_1\x[D_0] C$$ 

    This precisely coincides with the set $\pi_0(\Groth{(\gamma,\xi)})$ of connected components of the Grothendieck construction of the copresheaf $(\gamma,\xi)\:\DD\to \Set$. It is known that this is also the recipe to build the colimit of the copresheaf functor $(\gamma,\xi)\:\DD\to \Set$. After all, $K$ is the constant diagram functor $\Set\to [\DD,\Set]$, so its left adjoint must indeed be the functor that takes colimits.

    Finally, we study which categories $\DD$, seen as internal categories in $\Set$ or equivalently as their corresponding generalized action monads $\T_{\DD}$, satisfy descent or effective descent.
    
    By the above, $\DD$ satisfies descent precisely when the constant diagram functor $\Delta\:\Set\to [\DD,\Set]$ is fully faithful. This means that, given any $A,B\in \Set$, the function
    $$\HomC{\Set}{A}{B}\to \HomC{[\DD,\Set]}{\Delta A}{\Delta B}$$
    is bijective. But a natural transformation $\Delta A\aR{} \Delta B$ is precisely a collection of as many functions $A\to B$ as the number of connected components of $\DD$. The bijectivity of the function above then translates precisely into the request that $\DD$ has precisely one connected component. So the categories $\DD$ that satisfy descent are precisely the non-empty connected categories.

    A category $\DD$ then satisfies effective descent precisely when the constant diagram functor $\Delta:\Set\to [\DD,\Set]$ is an equivalence. This condition implies that any representable $\Hom{X}{-}\:\DD\to \Set$ is isomorphic to a constant copresheaf, say $\Delta W$. Now by the Yoneda lemma, given any $B\in \Set$,
    $$B=\Delta B (X)\iso \HomC{[\DD,\Set]}{\Hom{X}{-}}{\Delta B}\iso \HomC{[\DD,\Set]}{\Delta W}{\Delta B}\iso \HomC{\Set}{W}{B}$$
    where the last isomorphism is given by the assumption that $\Delta$ is fully faithful. This implies that $W$ is a singleton. So every category $\DD$ that satisfies effective descent must be such that for every $X,Y\in \DD$ the homset between $X$ and $Y$ has exactly one element. This condition translates precisely into $\DD$ being equivalent to the singleton category $\1$. We conclude that the categories $\DD$ that satisfy effective descent are precisely the categories equivalent to $\1$. Notice that \corx\ref{thm:internal-category-descent} offers an alternative quicker way to determine which categories $\DD$ satisfy effective descent. Indeed, it immediately forces all such $\DD$ to have that $\langle t,s\rangle:D_1\to D_0\times D_0$ is a bijection. And the homset between $X$ and $Y$ in $\DD$ is precisely given by the fibre of $\langle t,s\rangle$ over $(Y,X)$.
    
    This example also shows how internal categories satisfying descent can be far away from \vvv{C}ech groupoids. So the rigidity criterion of \corx\ref{thm:internal-category-descent} does not similarly hold for descent, as opposed to effective descent.

\subsection{The families bifibration}

This example generalizes the previous one. Let $\A$ be a category with the coproducts needed below. We consider the families bifibration $p\:\operatorname{Fam}(\A)\to\Set$, which has fibre over a set $I$ equal
to $\A^I$. Thus an object over $I$ is an $I$-indexed family $A=(A_i)_{i\in I}$ of objects of $\A$. It is well known that, when $\A=\Set$, the families bifibration is equivalent to the codomain bifibration over $\Set$.

For a function $u:I\to J$, reindexing with respect to $p$ is precomposition:
$u\st:\A^J\to\A^I$, with $(u\st B)_i=B_{u(i)}$.  Its left adjoint, when the
required coproducts exist, is given by summing over the fibres of $u$:
$$(u_!A)_j=\coprod_{i\in u^{-1}(j)}A_i.$$
The Beck--Chevalley condition follows from the canonical identification of
fibres in pullback squares of sets.

Let now $\DD$ be a small category, viewed as an internal category in $\Set$.
We write $D_0$ for its set of objects, $D_1$ for its set of arrows, and
$s,t:D_1\to D_0$ for source and target.  For a family
$A=(A_x)_{x\in D_0}$, the monad $\T^p_{\DD}=t_!s\st$ is computed as follows:
$s\st A$ assigns to an arrow $f:y\to x$ the object $A_y$, and $t_!$ then sums
over all arrows with fixed target.  Hence
$$(\T^p_{\DD}A)_x=\coprod_{f:y\to x}A_y.$$

A $\T^p_{\DD}$-algebra structure on $A$ is a morphism
$\T^p_{\DD}A\to A$ in $\A^{D_0}$.  Componentwise, this is the same as giving,
for every arrow $f:y\to x$ of $\DD$, a morphism $\alpha_f:A_y\to A_x$.  The
unit and multiplication axioms say respectively that
$\alpha_{\id{x}}=\id{A_x}$ and $\alpha_{gf}=\alpha_g\circ\alpha_f$ for
composable arrows $f:y\to x$ and $g:x\to z$.  Therefore a
$\T^p_{\DD}$-algebra is exactly a functor $\DD\to\A$, and
$$\Alg(\T^p_{\DD})\cong\A^\DD.$$
Thus, in the families bifibration, the $\DD$-descent data are precisely the
functors $\DD\to\A$.

The comparison functor comes from the unique map $D_0\to 1$.  Since the fibre
over $1$ is equivalent to $\A$, it is the constant-diagram functor
$\Delta:\A\to\A^\DD$.  It sends $B\in\A$ to the functor with value $B$ at
every object of $\DD$ and with every arrow acting as $\id{B}$.

Descent means that $\Delta$ is fully faithful.  If $\DD$ is non-empty and
connected, this holds: a natural transformation $\eta:\Delta B\to\Delta C$
is a family of morphisms $\eta_x:B\to C$, and naturality along any arrow
$f:x\to y$ gives $\eta_x=\eta_y$.  Since any two objects are joined by a
zigzag, all components of $\eta$ are equal; since $\DD$ is non-empty, such a
natural transformation is uniquely determined by one morphism $B\to C$. Assuming that in $\DD$ there are at least two objects with at least two arrows from one to the other one, the same argument as the previous example shows that also the converse holds: if $\DD$ satisfies descent then $\DD$ must be non-empty and connected.

Effective descent is stronger: it asks that $\Delta$ be an equivalence.  This
is not true for an arbitrary connected category, since a general diagram need
not be isomorphic to a constant one.  For the \vvv{C}ech groupoid $I(D_0)$ of
the map $D_0\to 1$, however, it is true when $D_0\neq\varnothing$.  Indeed,
$I(D_0)$ has one arrow $y\to x$ for every pair $x,y\in D_0$.  A functor
$I(D_0)\to\A$ is therefore a family $(A_x)$ together with coherent
isomorphisms $A_y\cong A_x$.  Choosing one object $x_0\in D_0$ identifies the
whole datum with the constant diagram on $A_{x_0}$.  Hence
$\Delta:\A\to\A^{I(D_0)}$ is an equivalence whenever $D_0$ is non-empty.

\subsection{Subobjects}

Let $\C$ be a regular category.  The subobject fibration
$p\:\operatorname{Sub}(\C)\to\C$ has as fibre over $X$ the poset
$\operatorname{Sub}(X)$ of subobjects of $X$.  Thus an object over $X$ is a
subobject $A\leq X$.

For a map $u:X\to Y$, reindexing is pullback:
$u\st:\operatorname{Sub}(Y)\to\operatorname{Sub}(X)$.  Thus, if $B\leq Y$ is
represented by a monomorphism $B\hookrightarrow Y$, then $u\st B\leq X$ is
the subobject represented by its pullback along $u$.  In $\Set$ this is the
ordinary inverse image $u^{-1}(B)$.  The left adjoint $u_!$ is regular image:
it sends $A\leq X$ to the image of $A\to X\xrightarrow{u}Y$.  The
Beck--Chevalley condition is the stability of regular images under pullback.

Let $\DD=(D_1\rightrightarrows D_0)$ be an internal category in $\C$.  For a
subobject $A\leq D_0$, the monad $\T^p_{\DD}=t_!s\st$ sends $A$ to the
subobject of $D_0$ obtained as follows: first pull $A$ back along
$s:D_1\to D_0$, giving the subobject of arrows whose source lies in $A$; then
take the regular image along $t:D_1\to D_0$.  Elementwise, in $\Set$,
$$\T^p_{\DD}(A)=\{x\in D_0\mid \text{there is an arrow }f:y\to x\text{ with }y\in A\}.$$

Since the fibres are posets, a $\T^p_{\DD}$-algebra structure on $A$ is just
an inequality $\T^p_{\DD}(A)\leq A$.  This says that $A$ is closed under the
arrows of $\DD$: whenever $y\in A$ and $f:y\to x$ is an arrow, then $x\in A$.
Because identity arrows give $A\leq\T^p_{\DD}(A)$, the algebra condition is
equivalent to $\T^p_{\DD}(A)=A$.

Thus, in the subobject fibration, the $\DD$-descent data are the invariant
subobjects of $D_0$.  In $\Set$, these are exactly the subsets
$A\subseteq D_0$ closed under the arrow relation of $\DD$.

The comparison functor for $D_0\to 1$ is
$\operatorname{Sub}(1)\to\Alg(\T^p_{\DD})$.  In $\Set$, the poset
$\operatorname{Sub}(1)$ has two elements, and their pullbacks to $D_0$ are
$\varnothing\subseteq D_0$ and $D_0\subseteq D_0$.  Hence effective descent
means that the only invariant subsets of $D_0$ are $\varnothing$ and $D_0$,
and one also needs $D_0\neq\varnothing$ so that these two pullbacks are
distinct.

For the \vvv{C}ech groupoid $I(D_0)$, this condition holds whenever
$D_0\neq\varnothing$.  If $A\subseteq D_0$ is non-empty and invariant, choose
$y\in A$.  For every $x\in D_0$ there is an arrow $y\to x$ in $I(D_0)$, so
invariance forces $x\in A$.  Hence $A=D_0$.  Thus, in $\Set$, the subobject
fibration has effective descent along $I(D_0)$ exactly when $D_0$ is
non-empty.

\subsection{Sheaves and gluing data}

Let $\mathsf{Spaces}$ be a suitable category of spaces.  We consider the
fibration $p\:\mathsf{Sh}\to\mathsf{Spaces}$ whose fibre over a space $X$ is
$\operatorname{Sh}(X)$.  Reindexing along $u:X\to Y$ is inverse image of
sheaves, $u\st=u^{-1}:\operatorname{Sh}(Y)\to\operatorname{Sh}(X)$.  We work
with maps for which the needed left adjoints $u_!$ and Beck--Chevalley
isomorphisms are available.

Let $\DD=(D_1\rightrightarrows D_0)$ be an internal category in
$\mathsf{Spaces}$.  For a sheaf $F$ on $D_0$, the action monad is
$\T^p_{\DD}F=t_!s\st F$.  A $\T^p_{\DD}$-algebra is a morphism
$t_!s\st F\to F$, or equivalently, by $t_!\dashv t\st$, a morphism of sheaves
on $D_1$
$$a:s\st F\longrightarrow t\st F.$$
This is the precise form of the action: $s\st F$ is the pullback of $F$ along
the source map, while $t\st F$ is the pullback along the target map.

The algebra axioms say that $a$ is compatible with the internal category
structure.  If $e:D_0\to D_1$ is the identity map, then $e\st a=\id{F}$.  If
$D_2=D_1\times_{D_0}D_1$ is the space of composable pairs, with projections
$p_1,p_2:D_2\to D_1$ and composition $m:D_2\to D_1$, then
$$m\st a=p_1\st a\circ p_2\st a.$$
Thus a $\T^p_{\DD}$-algebra is a sheaf on $D_0$ with a coherent action of
$\DD$.  When $\DD$ is a groupoid, these transition morphisms are
isomorphisms.

This recovers the usual gluing data for sheaves.  Suppose that we have spaces
$U_i$ and prescribed overlaps $U_{ij}$ with maps
$p^i_{ij}:U_{ij}\to U_i$ and $p^j_{ij}:U_{ij}\to U_j$.  Assume that these
data form an internal category with $D_0=\coprod_iU_i$ and
$D_1=\coprod_{i,j}U_{ij}$, where on $U_{ij}$ the source is $p^j_{ij}$ and
the target is $p^i_{ij}$.  Thus $U_{ij}$ is the space of arrows from the
$j$-th piece to the $i$-th piece.

A sheaf on $D_0$ is the same as a family of sheaves $F_i$ on the pieces
$U_i$.  Restricting $a:s\st F\to t\st F$ to $U_{ij}$ gives
$$a_{ij}:(p^j_{ij})\st F_j\longrightarrow (p^i_{ij})\st F_i.$$
The unit axiom says that the transition over $U_{ii}$ is the identity along
the prescribed identity map $U_i\to U_{ii}$, and the multiplication axiom is
the cocycle condition: on triple overlaps, going directly from $F_k$ to $F_i$
agrees with going first from $F_k$ to $F_j$ and then from $F_j$ to $F_i$.

In the classical algebraic-geometric situation, $U_{ij}=U_i\times_XU_j$ and
$p^i_{ij},p^j_{ij}$ are the two projections.  Then the maps $a_{ij}$ are the
usual isomorphisms between the restrictions of $F_j$ and $F_i$ to the
overlap, satisfying the identity and cocycle conditions.  The present
formulation is slightly more general: it also allows prescribed overlaps, as
long as they form an internal category.

If the pieces lie over a space $X$, say $q_i:U_i\to X$, and if the prescribed
overlaps are compatible with these maps in the sense that
$q_ip^i_{ij}=q_jp^j_{ij}$, then every sheaf $G$ on $X$ gives local sheaves
$F_i=q_i\st G$.  The equality $q_ip^i_{ij}=q_jp^j_{ij}$ canonically identifies
$(p^i_{ij})\st q_i\st G$ with $(p^j_{ij})\st q_j\st G$, and hence gives the
transition isomorphisms.  This defines the comparison functor
$\operatorname{Sh}(X)\to\Alg(\T^p_{\DD})$.

Descent means that this comparison functor is fully faithful.  Effective
descent means that it is an equivalence: every compatible system of local
sheaves on the $U_i$, with gluing isomorphisms on the overlaps $U_{ij}$,
comes from a unique sheaf on $X$, up to unique isomorphism.

\subsection{Modules and quasi-coherent modules}

We finish with the algebraic version.  The clean affine formulation is over
commutative rings.  Let $\mathsf{CRing}$ denote the category of commutative
rings, and consider the module bifibration over $\mathsf{CRing}$ whose fibre
over a ring $R$ is $R\operatorname{-Mod}$.  For a ring homomorphism
$u:R\to S$, reindexing is restriction of scalars,
$u\st:S\operatorname{-Mod}\to R\operatorname{-Mod}$, and its left adjoint is
extension of scalars, $u_!(M)=S\otimes_RM$.  The Beck--Chevalley isomorphism
is the usual tensor-product base-change isomorphism for pushout squares of
rings.

Thus, for an internal category $\DD=(D_1\rightrightarrows D_0)$ in rings, the
monad $\T^p_{\DD}=t_!s\st$ means: first restrict a module along the source
ring map, and then extend scalars along the target ring map.  Its algebras
are modules over $D_0$ equipped with compatible transition maps along the
arrows of $\DD$.

Now let $R\to S$ be a ring homomorphism.  Geometrically, this corresponds to
$q:\operatorname{Spec}S\to\operatorname{Spec}R$, whose \vvv{C}ech groupoid has
arrow scheme $\operatorname{Spec}(S\otimes_RS)$.  Algebraically, the two
projections correspond to the two maps $p_1,p_2:S\to S\otimes_RS$, given by
$p_1(s)=s\otimes1$ and $p_2(s)=1\otimes s$.

The descent category $\operatorname{Desc}(S/R)$ has as objects $S$-modules
$N$ together with an isomorphism
$$\alpha:(S\otimes_RS)\otimes_{S,p_2}N\longrightarrow (S\otimes_RS)\otimes_{S,p_1}N$$
which restricts to the identity after $S\otimes_RS\to S$ and satisfies the
cocycle condition after base change to $S\otimes_RS\otimes_RS$.  This is the
usual algebraic descent datum for modules.

The comparison functor is extension of scalars with its canonical descent
datum:
$$R\operatorname{-Mod}\longrightarrow\operatorname{Desc}(S/R),\qquad M\longmapsto S\otimes_RM.$$
Descent means that this functor is fully faithful: morphisms of $R$-modules
are recovered from their scalar extensions, provided the extended morphisms
respect the descent data.  Effective descent means that it is an equivalence:
every $S$-module with descent datum is, up to unique isomorphism, of the form
$S\otimes_RM$.  If $S$ is faithfully flat over $R$, this comparison functor
is an equivalence; this is the classical faithfully flat descent theorem for
modules.

The quasi-coherent version is the same statement in geometric language.  For
a scheme $X$, the fibre is $\operatorname{QCoh}(X)$, and for a morphism
$u:X\to Y$ the reindexing functor is pullback $u\st$.  In affine schemes,
$X=\operatorname{Spec}R$ and $U=\operatorname{Spec}S$, this pullback is exactly
extension of scalars $M\mapsto S\otimes_RM$.  Thus the ring calculation above
is the affine form of descent for quasi-coherent modules.  More generally,
for a morphism $q:U\to X$, the \vvv{C}ech groupoid $I(q)$ gives the usual
descent datum: a quasi-coherent module $F$ on $U$ together with an
isomorphism $\pi_2\st F\to\pi_1\st F$ on $U\times_XU$, satisfying the identity
and cocycle conditions.  Effective descent says that such data are exactly
quasi-coherent modules on $X$; for instance this holds in the standard
faithfully flat setting.

\subsection*{Acknowledgements}
We thank Rui Prezado for his help in finding connections between our theory and the existing literature in descent theory. We also thank James Deikun for pointing out to us the links between our work and the theory developed by Leinster for $T$-multicategories.

\end{document}